\documentclass{article}
\usepackage[a4paper,top=3cm,bottom=3cm,left=2.1cm,right=2.1cm,marginparwidth=1.75cm]{geometry}

\usepackage{graphicx} 

\usepackage{bm}
\usepackage{amsmath, amsfonts, amssymb, amsthm}
\usepackage{mathrsfs, dsfont}
\usepackage[dvipsnames]{xcolor}
\usepackage{graphicx}
\usepackage{verbatim}
\usepackage{float}
\usepackage{subcaption}
\usepackage{comment}
\usepackage{ulem}
\usepackage{bigints}

\usepackage[colorlinks=true, allcolors=blue]{hyperref}
\usepackage{xparse} 
\usepackage{hyperref}
\usepackage{tasks}
\settasks{style=itemize}
\usepackage{parskip}
\usepackage{todonotes}
\usepackage{subcaption}
\DeclareFontFamily{U}{mathx}{}
\DeclareFontShape{U}{mathx}{m}{n}{ <-> mathx10 }{}
\DeclareSymbolFont{mathx}{U}{mathx}{m}{n}
\DeclareFontSubstitution{U}{mathx}{m}{n}
\DeclareMathAccent{\widecheck}{0}{mathx}{"71}

\usepackage[round]{natbib}

\newtheorem{theorem}{Theorem}[section]

\newtheorem{lemma}[theorem]{Lemma}
\newtheorem{definition}[theorem]{Definition}

\newtheorem{proposition}[theorem]{Proposition}

\newtheorem{corollary}[theorem]{Corollary}

\theoremstyle{remark}
\newtheorem{remark}{Remark}[section]
\newtheorem{example}{Example}[section]

\numberwithin{equation}{section}

\newenvironment{sqremark}{\begin{remark}}{\hfill \tiny $\blacksquare$ \end{remark}}
\newenvironment{sqexample}{\begin{example}}{\hfill \tiny $\blacksquare$ \end{example}}

\usepackage{pifont}

\newcommand{\R}{\mathbb{R}}
\newcommand{\N}{\mathbb{N}}
\newcommand{\C}{\mathbb{C}}

\newcommand{\E}{\mathbb{E}}

\def\P{\mathbb{P}} 
\newcommand{\F}{\mathcal{F}}

\newcommand{\abs}[1]{\left|#1 \right|}
\newcommand{\norm}[1]{\abs{\abs{#1}}}

\renewcommand{\Re}{\, \mathfrak{Re}}
\renewcommand{\Im}{\, \mathfrak{Im}}
\newcommand{\indic}[1]{\mathds{1}_{\left\{ #1 \right\}}}

\newcommand{\TA}[1][d]{T(\R^{#1})}
\newcommand{\eTA}[1][d]{T((\R^{#1}))}

\newcommand{\TAC}[1][d]{T(\C^{#1})}
\newcommand{\eTAC}[1][d]{T((\C^{#1}))}

\newcommand{\emptyword}{{\color{NavyBlue}{\mathbf{\varnothing}}}}
\newcommand{\word}[1]{{\color{NavyBlue}{\mathbf{#1}}}}
\newcommand{\proj}[1]{|_{\word{#1}}}

\newcommand{\bp}{\bm{p}}
\newcommand{\bq}{\bm{q}}

\newcommand{\bxi}{\bm{\xi}}
\newcommand{\wv}{\word{v}}

\newcommand{\bsigma}{\bm{\sigma}}
\newcommand{\bgamma}{\bm{\gamma}}

\newcommand{\bpsi}{\bm{\psi}}
\newcommand{\bu}{\bm{u}}

\newcommand{\bchi}{\bm{\chi}}

\newcommand{\bell}{\bm{\ell}}
\newcommand{\br}{\bm{r}}

\newcommand{\shuprod}{\mathrel{\sqcup \mkern -3mu \sqcup}}
\newcommand{\shupow}[1]{^{\shuprod #1}}
\newcommand{\shuexp}[1]{\exp^{\shuprod}\left({#1}\right)}
\newcommand{\conpow}[1]{^{\otimes #1}}

\NewDocumentCommand{\sighat}{O{t} O{W}}{\widehat{\mathbb{#2}}_{#1}}
\NewDocumentCommand{\sigtilde}{O{t} O{W}}{\widetilde{\mathbb{#2}}_{#1}}
\NewDocumentCommand{\sig}{O{t} O{W}}{\sighat[#1][#2]}
\newcommand{\sigW}[1][t]{\widehat{\mathbb{W}}_{#1}}

\newcommand{\sigX}[1][t]{{\mathbb{X}}_{#1}}
\newcommand{\sigY}[1][t]{{\mathbb{Y}}_{#1}}
\newcommand{\sigXhat}[1][t]{{\widehat{\mathbb{X}}}_{#1}}

\newcommand{\bracket}[2]{\langle #1, #2 \rangle}

\NewDocumentCommand{\bracketsig}{O{t} O{W} m}{\bracket{#3}{\sig[#1][#2]}}
\newcommand{\bracketsigW}[2][t]{\bracket{#2}{\widehat{\mathbb{W}}_{#1}}}

\NewDocumentCommand{\bracketsigtrunc}{O{M} O{t} O{W} m}{\bracket{#4}{\sig[#2][#3]^{\leq #1}}}
\NewDocumentCommand{\bracketsigtilde}{O{t} O{W} m}{\bracket{#3}{\sigtilde[#1][#2]}}

\usepackage{mathtools}
\usepackage{shuffle}
\usepackage{bbm}
\usepackage{authblk}
\mathtoolsset{showonlyrefs}
\usepackage{enumitem}
\usepackage{stmaryrd}

\title{Affine Structure of the Brownian Signature
}

\author[1]{Eduardo Abi Jaber\thanks{eduardo.abi-jaber@polytechnique.edu.  EAJ is grateful for the financial support from the Chaires FiME-FDD and  Financial Risks at Ecole Polytechnique.}}
\author[1]{Elie Attal\thanks{elie.attal@polytechnique.edu.}}
\author[1, 2]{Dimitri Sotnikov\thanks{dmitrii.sotnikov@polytechnique.edu. DS is grateful for the financial support provided by Engie Global Markets.}}
\affil[1]{Ecole Polytechnique, CMAP}
\affil[2]{Engie Global Markets}

\begin{document}

\maketitle

\begin{abstract}
 We establish an infinite-dimensional affine transform theory for the time-augmented Brownian signature.  Our first main result shows that,  for a suitable class of linear functions of the signature, the conditional Fourier--Laplace transform admits an entire signature expansion. We prove that the associated coefficients solve an infinite-dimensional  linear differential equation on the extended tensor algebra. Our second main result shows that the logarithm admits a local signature expansion whose coefficients satisfy a Riccati equation on the extended tensor algebra, revealing a generalized affine structure of the Brownian signature  in a genuinely path-dependent setting. In contrast to conventional affine processes, we show that this representation is intrinsically local: zeros of the Fourier--Laplace transform in the complex plane prevent any global expansion. To recover global representations, we introduce a new class of randomized Riccati equations with path-dependent terminal conditions  through a recentering argument.  Furthermore, we establish uniqueness of solutions to the linear and Riccati equations within a suitable class of solutions.  Our results provide a theoretical framework for transform methods in  non-Markovian settings, with applications to the computation of conditional distributions.
\end{abstract}



\tableofcontents

\section{Introduction}

Let $\widehat{\mathbb W}_T$ denote the Stratonovich signature of the time-augmented Brownian motion $\widehat W_t=(t,W_t)$, that is, the collection of all iterated integrals
\begin{align}\label{eq:introsig}
\widehat {\mathbb W}_{T}^{i_1,\ldots,i_n}
= \int_{0<t_1<\cdots<t_n<T} \circ d\widehat W_{t_1}^{i_1}\cdots \circ d\widehat W_{t_n}^{i_n},
\qquad i_k \in \{\word{0},\word{1}\},    
\end{align}
where $\word{0}$ denotes the time coordinate and $\word{1}$ the Brownian coordinate.

We establish signature expansions of the conditional Fourier--Laplace transform 
\begin{align}
u(t,\sigW[t])
=
\mathbb E\!\left[
\exp\!\big(\langle \bm p,\sigW[T]\rangle\big)
\,\Big|\,\sigW[t]
\right]
\end{align}
 and its logarithm, where
\[
\langle \bm p,\widehat{\mathbb W}_{t} \rangle
:=
\sum_{n\geq 0}
\sum_{i_1,\ldots,i_n \in \{\word{0}, \word{1}\}}
\bm p^{\word{i_1\ldots i_n}}
\, \widehat{\mathbb W}^{\word{i_1\ldots i_n}}_{t},
\]
with only finitely many non-zero coefficients $\bm p^{\word{i_1\ldots i_n}} \in \mathbb C$, with $\bm p^{{\emptyword}}$ denoting the constant coefficient corresponding to $n=0$. For example, for $\bp = \alpha \word{1}\word{1}\word{1}\word{1}$, we get that  $\langle \bm p,\widehat{\mathbb W}_T \rangle = \alpha \frac{W^4_t}{4!}$ and for  $\bp = \alpha \word{1}\word{1}\word{1}\word{1}\word{0}$, we get that  $\langle \bm p,\widehat{\mathbb W}_T \rangle = \alpha \int_0^T \frac{W^4_s}{4!}ds$, with $\alpha \in \mathbb C$.  Combinations of mixed words generate genuinely path-dependent functionals.

Our first main result establishes that, for a suitable class $\mathcal B$ of  coefficients $\bm p$ that we construct, the map $\mathbb X \mapsto u(t,\mathbb X)$ admits an entire signature expansion: there exist  deterministic  coefficients $(\bm u_s)_{0 \leq s \leq T}$ such that
$u(t,\mathbb X)=\langle \bm u_t,\mathbb X\rangle,$
and the resulting  series {$\langle \bm u_t,\mathbb X\rangle$ contains infinitely many non-zero terms and} converges for every group-like element $\mathbb X$. Furthermore, we show that $\bm u$ is the unique solution  of the following linear equation on the extended tensor algebra
\begin{equation}\label{eq:introheat_eq}
\begin{cases}
 \dot{\bm u}_t + \mathscr L \bm u_t = 0, \\
\bm u_T = \shuexp{\bm p},
\end{cases}
\end{equation}
where $\exp^{\shuprod}$ denotes the shuffle exponential and  $\mathscr{L}$ is  the diffusion generator of the Brownian signature defined by
$\mathscr{L}\,\bell = \bell\proj{\word{0}} + \frac{1}{2}\bell\proj{\word{1}\word{1}}.$ Here $\proj{\word i}$ denotes the right-shift operator in the letter $\word i$, and acts as a derivation on the tensor  algebra, playing the role of differentiation on power series coefficients. We state an informal version combining Theorems~\ref{theorem:analyticity_general} and \ref{theorem:existence_heat_eq} and Corollary  \ref{cor:uniqueness_heat}.

\textbf{Theorem 1.}
\textit{Let $\bm p \in \mathcal B$. Then, there exists a deterministic family $(\bm u_s)_{0\le s\le T}$ such that
$$
\mathbb E\!\left[
\exp\!\big(\langle \bm p,\sigW[T]\rangle\big)
\Big|\sigW[t]
\right]
=
\langle \bm u_t,\sigW[t]\rangle,
\qquad t\in [0,T].
$$
Furthermore, $\bm u$ is the unique solution (within a suitable class) to the linear equation \eqref{eq:introheat_eq}.}

Our second main result passes to the logarithm and establishes a local analytic signature expansion
$\log u(t,\sigW[t]) = \langle \bpsi_t,\sigW[t]\rangle,$
valid on a random stochastic interval, for a deterministic family of coefficients $(\bpsi_t)_{0\le t\le T}$. We show that these coefficients solve a Riccati equation on the tensor algebra, namely
\begin{equation}\label{eq:introriccati_eq}
\begin{cases}
\dot{\bpsi}_t + \mathscr{L}\bpsi_t + \frac{1}{2} (\bpsi_t
\proj{\word{1}})\shupow{2} = 0, \\
\bpsi_T = \bm p. 
\end{cases}
\end{equation} 
The nonlinearity in \eqref{eq:introriccati_eq} is induced by the shuffle product structure. This equation may be viewed as the tensor-algebra analogue of the classical Riccati equations arising in affine transform methods and yields more precise and stable numerical results than the linear equation.
Our main Theorems~\ref{theorem:log_analyticity_general} and \ref{theorem:existence_riccati} and Proposition~\ref{prop:riccati_uniqueness} are summarized in the following informal statement.

\textbf{Theorem 2.} 
\textit{Let $\bm p \in \mathcal B$. Assume that $u$ does not vanish. Then, there exists a deterministic family $(\bm \psi_s)_{0\le s\le T}$ such that
\begin{equation}\label{eq:log-Laplace_expansion}
    \log \mathbb E\!\left[
\exp\!\big(\langle \bm \bp,\sigW[T]\rangle\big)
\Big|\sigW[t]
\right]
=
\langle \bm \psi_t,\sigW[t]\rangle,
\qquad  t \in \llbracket 0\,, \tau \rrbracket\,,
\end{equation}
where $\tau$ is a  positive stopping time.
Furthermore, $\bm \psi$ is the unique solution (within a suitable class) of the Riccati equation \eqref{eq:introriccati_eq}}.

Theorem 2   shows that the time-augmented Brownian signature exhibits a generalized affine--Riccati duality on the tensor algebra: the logarithm of its conditional transforms is affine in the signature, while the coefficients satisfy Riccati equations. 

 A fundamental departure from classical affine theories is that this affine structure is intrinsically local. Although the Fourier–Laplace transform admits an entire signature expansion, we show that its logarithm generally does not. More precisely, even if $\bp$ has finitely many non-zero components, the series $ \langle \bpsi_t, \widehat{\mathbb W}_t \rangle$
typically involves infinitely many non-zero terms (except when $\bp$ is of linear-quadratic form in $\word{1}$). This raises non-trivial convergence issues for the series.  In general, we show that the radius of convergence is finite, see Theorem~\ref{theorem:zeros_quartic_exp_main}. As a consequence, the logarithm of the conditional Fourier--Laplace transform is analytic but not entire, yielding only a local expansion near the initial time, rather than a global representation on $[0,T]$.

To overcome this intrinsic locality, we introduce a recentering procedure based on Chen’s identity and the independence of Brownian increments. This leads to a new class of randomized Riccati equations with path-dependent terminal conditions, where the deterministic terminal condition in \eqref{eq:introriccati_eq} is replaced by a random one depending on the Brownian position at time $t$. This yields global representations of the logarithmic transform. More precisely, we summarize the findings of Theorem~\ref{theorem:recentering}.

\textbf{Theorem 3.} 
\textit{Let $\bm p \in \mathcal B$. Assume that $u$ does not vanish. Fix $t\in [0,T]$. Then, 
$$
\log \mathbb E\!\left[
\exp\!\big(\langle  \bp,\sigW[T]\rangle\big)
\Big|\sigW[t]
\right]
= (\bm{\psi}^{\sigW[t]}_t)^{\emptyword},
$$
 where the random family $(\bm{\psi}^{\sigW[t]}_s)_{s\leq T}$ solves the same Riccati equation \eqref{eq:introriccati_eq} but where the terminal condition is replaced by a  random one determined by $\bp$ and $\sigW[t]\,$.}

  In contrast with the local expansion in Theorem~2 that involves an infinite series, the  representation in Theorem~3 involves only the component associated with the empty word. The price to pay is that, for each $t$, one must solve a new Riccati equation, as opposed to solving a single deterministic equation in the local representation and using it for all times.

\subsection{Motivation and related literature}
Our motivation for studying such expansions stems from the growing role of signatures in stochastic analysis. Signatures, that is, sequences of iterated integrals introduced by \cite*{chen1957integration, chen1977iterated}, have long been central in stochastic Taylor expansions for Markovian systems  \citep*{kloedenstochastic, arous1989flots} and for non-Markovian ones \citep*{litterer2014chen,  dupire2022functional, abi2024path}, and have gained renewed importance with the development of rough path theory initiated by  \cite{lyons1998differential}.  A key feature of signatures is that they play a role analogous to polynomials on path space, allowing for a linearization of nonlinear functionals  and serving as universal approximators. This has led to increasing interest in the characterization of their laws and in the use of signatures for stochastic control problems in non-Markovian environments.

In this context, the representation of the Fourier--Laplace transform and its logarithm for signatures has emerged as a central but challenging problem. Partial results have been obtained in specific settings: series expansions appear in \cite*{cuchiero2025signaturesdesaffinepolynomial} for signature SDEs, in  \cite*{abijaber2024fourier} for integrated polynomial Ornstein-Uhlenbeck models and in \cite*{abi2025signature} for signature volatility models. In these works, first  connections with the linear equation \eqref{eq:introheat_eq} and the Riccati equation \eqref{eq:introriccati_eq} are identified, relying on affine and polynomial type algebraic structures of the signature, mainly assuming existence of solutions to   \eqref{eq:introheat_eq}--\eqref{eq:introriccati_eq} and convergence of the series, except in very specific Markovian Brownian settings that we detail in the next paragraph. 
A connection between the Riccati equation \eqref{eq:introriccati_eq} and  non-Markovian stochastic control has also been  established in \cite*{abijaber2025signature}, where existence and convergence of the series are also assumed. Related representation results for characteristic functions have appeared in \cite*{lyons2024pde}, where a PDE viewpoint on signatures is developed, linking expectations of signature functionals to infinite-dimensional evolution equations. In addition, \cite*{friz2022forests} develop alternative algebraic and combinatorial expansions for the logarithm of characteristic functions, based on generalized forest expansions, {and \cite*{chevyrev2016characteristic} expand characteristic functions of the type $\E[\exp(i \lambda X)]$ in the Fourier factor $\lambda\,$, for $X$ a random signature, in the context of the moments problem.} As demonstrated by \cite*{abi2025signature}, numerically solving the Riccati equation is a promising approach for computing conditional distributions in challenging non-Markovian settings, where PDE-based methods typically fail. This is particularly relevant for applications in mathematical finance.

However, a general theory ensuring existence and uniqueness for such infinite-dimensional linear and Riccati equations, together with convergence of the associated series expansions, remains largely open.  Partial results are available in the literature but only in Markovian settings where additional regularization is present. For instance, \cite*[Section 5.3.1]{cuchiero2025signaturesdesaffinepolynomial} derive series expansions  for the Fourier--Laplace transform and its logarithm for polynomial functions of $W_T$, which can be seen as a version of Theorems 1 and 2 above in the Markovian case. A representative example is {$p(W_T)=-W_T^4$}, which corresponds to a linear functional of the signature associated with the multi-index {$\bp = -4!\cdot \word{1111}$}. In this case, since the functional depends only on the current value $W_T$, the proof crucially relies on the regularization induced by the Gaussian density. Analogous local expansions are obtained in \cite*[Section 3.2]{abijaber2024fourier} for Fourier--Laplace transforms of integrated functionals of the form$
\int_0^T p^2(X_s)ds$ and $\int_0^T p(X_s)\circ dW_s,$
where $X$ is an Ornstein--Uhlenbeck process, relying on Gaussian calculus. In the Brownian case with $p(W_s)=W_s^3$, these correspond to linear functionals of the signature \eqref{eq:introsig}  associated with multi-indices such as $\bp = {6!}\cdot \word{1111110}$ and $\bp = {3!} \cdot\word{1111}$. 

These existing approaches rely heavily on power series expansions for Brownian functionals, based on Gaussian calculus and explicit density arguments. These methods crucially exploit the fact that Gaussian densities are entire functions, which provide strong regularization properties. This mechanism breaks down when working with the full time-augmented signature \eqref{eq:introsig}, where no such global regularization is available. In our setting, regularity must instead come from the functional in the exponential, which requires the development of new analytical and algebraic ideas.

Our results {use a combination of algebraic, analytic, and probabilistic techniques to} address this limitation and provide series expansions together with  a novel representation for Fourier--Laplace transforms of signatures, including existence, uniqueness and  convergence results, thereby resolving a question that remained open in the literature. In particular,  our approach yields a global representation through the analysis of a novel Riccati equation with random terminal condition. 

Our paper connects to and complements the literature on affine processes and affine-transform techniques.
A central theme in such literature is the duality between affine structures and deterministic Riccati equations:  classical affine processes with finite-dimensional Riccati equations \citep*{duffie2002affine}, super Brownian motions with Riccati PDEs \citep{etheridge2000introduction}, and affine Volterra processes with Riccati--Volterra equations \citep*{abi2019affine}. In the same spirit,  the time-augmented Brownian signature enjoys a  generalized affine--Riccati duality on the tensor algebra: the logarithm of its conditional transforms is affine in the signature with  coefficients satisfying Riccati equations. Compared to the literature, the affine structure is inherently local here, and global representations are recovered only via a randomized Riccati formulation. At the same time, the framework goes beyond standard affine settings by incorporating both nonlinear and path-dependent features.

\subsection{Strategy and implications}
Our starting point is to return to the one-dimensional
power-series setting in order to isolate the structural ingredients that survive beyond
the Markovian framework. We start by developing  in Section~\ref{section:heuristices}  an approach based on algebraic properties
of polynomials rather than on the smoothing effect of Gaussian densities, which is specific
to the one-dimensional Markovian setting and does not extend to time-augmented signatures.
In this polynomial setting, the key algebraic observation is that for any random variable $Y$, one
can factorize
\begin{align}\label{eq:intropfact}
  \exp({p(x+Y)})
  =
  \exp({p(Y)})
  \exp\!\big(p(x+Y)-p(Y)\big),
\end{align}
thereby separating the leading contribution $p(Y)$ from the terms involving the
deterministic variable $x$. Expanding only the second factor as a power series yields an
entire expansion whose coefficients are random functions of $Y$. The crucial step is then
to show that the decay induced by the leading term $\exp({p(Y)})$ dominates the growth of the
remaining coefficients, for polynomials $p$ with real part bounded from above, allowing
one to interchange expectation and infinite summation, which yields that the map
$x\mapsto \mathbb{E}[\exp({p(x+Y)})]$ is entire in $x$.

{We then extend this algebraic strategy to the time-augmented signature of Brownian motion.
In a first step, we construct in Section~\ref{subsection:class_of_coefs} a class
$\mathcal{B}$ of admissible coefficients  $\bp$ for which the real part of
$\langle \bp, \widehat{\mathbb{W}}\rangle$ is bounded from above. The central tool
for our bound is the use of \cite{Lyndon1954} words. More precisely using the fact that   the shuffle algebra is generated by Lyndon words thanks to the \cite{Radford1979ANR} theorem, we  control arbitrary signature
coordinates by leading powers of $|W_T|^N$ and $\int_0^T |W_s|^N \, ds$, see
Lemma~\ref{lem:sig_bound}. Then, in Section~\ref{subsection:signature_expansion_u}, for
$\bp \in \mathcal{B}$, we combine this control with the properties of
shuffle-compatible norms to obtain an entire signature expansion of the conditional
Fourier--Laplace transform, extending the factorization idea in~\eqref{eq:intropfact},
see Theorem~\ref{theorem:analyticity_general}. Passing to the shuffle logarithm, we
construct in Section~\ref{subsec:existence_Riccati_stopping_time} a local signature
expansion of the log-Fourier--Laplace transform given by
Theorem~\ref{theorem:log_analyticity_general}. Unlike the Fourier--Laplace transform
itself, its logarithm is generally not entire. We show in
Section~\ref{sect:global_exp_counterexample} that this locality is structural and
originates from the appearance of zeros of the transform outside the quadratic Gaussian
setting.

We then show in Section~\ref{sect:heat_and_riccati_sig} that the resulting coefficients
solve infinite-dimensional differential equations on the tensor algebra. We start by
proving in Section~\ref{subsec:existence_heat_eq}, using It\^{o}'s formula on signature
elements, that the coefficients of the conditional Fourier--Laplace expansion satisfy a
linear equation on the tensor algebra (Theorem~\ref{theorem:existence_heat_eq}). For
a sufficiently regular terminal condition, this linear equation admits an explicit
closed-form solution. However, this is no longer the case for the Fourier--Laplace
transform, and we show in Section~\ref{sect:explicit_sol_heat} that the closed-form
expressions are no longer available. We then establish uniqueness of the solution within a
certain class in Theorem~\ref{subsec:uniqueness_heat}, see
Section~\ref{subsec:uniqueness_heat}. We then prove in
Section~\ref{subsec:existence_Riccati} that the coefficients of the log-Fourier--Laplace
transform expansion satisfy a Riccati equation, see
Theorem~\ref{theorem:existence_riccati}. In Section~\ref{section:recentering}, a
recentering procedure for the Riccati equation is proposed, allowing us to recover a
global representation of the log-Fourier--Laplace transform by solving multiple Riccati
equations, as stated in Theorem~\ref{theorem:recentering}.
}

Finally, in Section~\ref{sect:laplace_sig_mart},  we apply our results, namely Theorems~2 and~3 above, to derive a representation of the joint Fourier--Laplace transform of
$\int_0^T \langle \sigma,\sigW[t]\rangle\, dW_t$
and its quadratic variation. As a consequence, we obtain in Corollary~\ref{cor:sig_vol_Fourier_laplace} the joint Fourier--Laplace transform of the log-price process and integrated variance in the signature volatility model introduced in \cite{abi2025signature}. In \cite[Theorem~4.1]{abi2025signature}, such a representation was obtained under the assumption that the associated Riccati equation admits a solution and that the corresponding signature series converges globally. Verifying these assumptions was left as an open problem. Our results provide a complete justification of the transform formula and identify the precise mechanism behind its validity. In particular, we show that the signature expansion of the log-transform is intrinsically local, while global representations are recovered only through the randomized Riccati equation that we introduce here. Remarkably, the required assumptions reduce to a structural condition on the signature volatility model ensuring that certain Doléans--Dade exponentials of signatures are true martingales, as obtained by \cite*{jaber2025martingalepropertymomentexplosions}. In particular, our class of admissible coefficients $\mathcal B$ is consistent with, and naturally contains, the coefficients arising in this framework.

Beyond signature volatility models, our results provide a new approach to conditional transforms beyond the classical affine framework and in genuinely non-Markovian settings, where PDE methods are typically unavailable.  We also expect the methodology to be useful in stochastic control, where related signature expansions have recently appeared in \cite*{abijaber2025signature}; we illustrate the applicability of Theorems 2 and 3  to optimal control problems with signatures in   the companion paper~\cite*{riccatisigcontrol}.

{\textbf{Outline.} The paper is outlined as follows. In Section~\ref{section:heuristices}, we
illustrate the main ideas of our approach in the power series case.
Section~\ref{section:preliminaries} presents path signatures and some necessary
preliminary results. Section~\ref{section:sig_expansions} constructs the expansions of
the conditional Fourier--Laplace transform and its logarithm in the Brownian signature
framework. In Section~\ref{sect:heat_and_riccati_sig}, we show that the coefficients of
these expansions satisfy infinite-dimensional linear and Riccati equations respectively, and
discuss the existence and uniqueness of solutions to these equations.
Section~\ref{sect:laplace_sig_mart} deals with the joint Laplace transform of
$\int_0^T \langle \sigma, \sigW[t]\rangle\, dW_t$ and its quadratic variation, using the
results of the two previous sections. Appendix~\ref{sect:proof_of_sig_bound} contains the
proof of the signature bounds and minimal preliminaries on Lyndon words.
Appendix~\ref{app:existence_zeros} presents the expansion of the conditional expectation,
showing that the Fourier--Laplace transform may vanish on the complex plane.
}

\section{The power series case: heuristics and core ideas}\label{section:heuristices}
Throughout this paper, $(\Omega, \F,(\F_t)_{t \geq 0}, \P)$ denotes a filtered probability space supporting a one-dimensional Brownian motion $W = (W_t)_{t \geq 0}$ generating the filtration $(\F_t)_{t \geq 0}$.

In this section, we focus on the heuristics and core ideas in the one-dimensional power series framework, with the aim of isolating the main structural ingredients before extending them to the full signature setting. The key idea is to develop arguments based on algebraic properties that do not rely on the standard smoothing effect of Gaussian densities, which is specific to the one-dimensional Markovian setting and cannot be extended to the time-augmented signature framework, in contrast to more algebraic strategies.

We consider the case where
\begin{align}\label{eq:ppowerseries}
\bp = \sum_{k = 0}^N p^k \underbrace{\word{1}\cdots\word{1}}_{\text{$k$ times}} , 
\end{align}
with $N \in \mathbb{N}$ and coefficients $p^k \in \mathbb{C}$. For this choice, linear functionals of the time-augmented signature reduce to  standard polynomials of one variable of degree $N$:
\begin{equation}\label{eq:entire_series}
    \bracket{\bp}{\sigW} = p(W_t) \quad   \text{with} \quad 
    p(x) := \sum_{k = 0}^N p^k \frac{x^k}{k!},
\end{equation}
where, by a slight abuse of notation throughout this section, the superscript $k$ in $p^k$ refers to the $k$-th coefficient, whereas in $x^k$ it denotes exponentiation, i.e.~the $k$-th power of $x$. Exponentiation will only be used with the variable $x$.

Our two primary objects of interest are the Fourier--Laplace transform $u$ and its logarithm $\psi$ given by
\begin{equation}\label{eq:u_psi_polyn}
    u(t, x) := \mathbb{E}\left[\exp\left(p(W_T)\right) \big| W_t = x\right], \quad \psi(t, x) := \log(u(t, x)),
\end{equation}
where the latter is defined provided that $u(t, x) \neq 0$.
The function $u$ satisfies the heat equation
\begin{equation}\label{eq:heat_eq_markov}
\begin{aligned}\begin{cases}
    \partial_t u(t, x) + \dfrac12\partial_{xx} u(t, x) = 0\,, \quad t \in [0, T]\,, \quad x \in \mathbb{R}\,, \\
    u(T, x) = \exp(p(x))\,,\quad  x \in \mathbb{R}\,,
\end{cases}
\end{aligned}
\end{equation}
while $\psi$ is a solution to the Hamilton--Jacobi--Bellman (HJB)  equation
\begin{equation}\label{eq:riccati_eq_markov}
\begin{aligned}\begin{cases}
    \partial_t \psi(t, x) + \dfrac12\partial_{xx} \psi(t, x) + \dfrac{1}{2}(\partial_x \psi(t, x))^2 = 0\,,\quad t \in [0, T]\,, \quad x \in \mathbb{R}\,, \\
    \psi(T, x) = {p(x)}\,, \quad x \in \mathbb{R}\,,
\end{cases}
\end{aligned}
\end{equation}
which follows from the Cole--Hopf transform.

We refer to the mapping $u$ in \eqref{eq:u_psi_polyn} as the ``Fourier--Laplace transform''. This terminology is justified by identifying a  polynomial $p(x)$, or more generally an entire series, with the sequence of its coefficients $\bm{p} = (p^0, p^1, \ldots) \in \mathbb{C}^{\mathbb{N}}$ and interpreting the evaluation map of the polynomial or series as a complex linear functional
$$
p(x) = \langle \bm{p}, \mathbbm{x} \rangle := \sum_{k = 0}^{+\infty} p^k \dfrac{x^k}{k!},
$$
applied to the infinite sequence 
$$
\mathbbm{x} = \left(1, x, \dfrac{x^2}{2!}, \ldots, \dfrac{x^n}{n!}, \ldots\right),
$$
which represents the signature of a one-dimensional path with increment $x$. The product of two such (convergent) series with coefficients $\bm{p}, \bm{q} \in \mathbb{C}^{\mathbb{N}}$ can be linearized via the Cauchy product:
\begin{equation}\label{eq:cauchy_product_linearization}
    \langle \bm{p}, \mathbbm{x} \rangle \langle \bm{q}, \mathbbm{x} \rangle = \langle \bm{p} * \bm{q}, \mathbbm{x} \rangle,
\end{equation}
where the coefficient $\bm{p} * \bm{q} \in \mathbb{R}^{\mathbb{N}}$ is defined by
\begin{equation}\label{eq:cauchy_product_def}
    (\bm{p} * \bm{q})^n = \sum_{k = 0}^n \binom{n}{k} p^k q^{n-k}, \quad n \in \mathbb{N}.
\end{equation}
For $n \in \mathbb{N}$, we define the $n$-th Cauchy power recursively by 
$
\bp^{*n} := \bp^{*(n-1)} * \bp,
$
with  $\bp^{*0} := (1, 0, 0, \ldots)$. Furthermore, we define the Cauchy 
exponential as
\begin{equation}\label{eq:cauchy_exp}
    \exp^*\left(\bp\right) := \sum_{n \geq 0} \frac{\bp^{*n}}{n!}.
\end{equation}

\subsection{Fourier--Laplace transform: an entire series expansion} \label{sect:expansion_u_power_ser}
We aim to show that $u(t, \cdot)$ is an entire function:
\begin{equation}\label{eq:u_entire_series_exp}
    u(t, x)  = \sum_{k\geq 0} \bu_t^k \frac{x^k}{k!} = \bracket{\bu_t}{\mathbbm{x}} , \quad x \in \mathbb{R},
\end{equation}
for an infinite sequence of coefficients $\bu_t = (\bu_t^0, \bu_t^1, \ldots)$ satisfying the following infinite-dimensional ODE:
\begin{equation}\label{eq:heat_eq_coef_power_ser}
\begin{cases}
    \dot \bu_t + \dfrac{1}{2}\bu_t\proj{11} = 0, \quad t \in [0, T], \\
    \bu_T = \exp^*\left(\bp\right),
\end{cases}
\end{equation}
where we define the shift $\bell|_{\word{11}} = (\bell^2, \bell^{3}, \ldots)$.  We note that the ODE \eqref{eq:heat_eq_coef_power_ser} arises naturally when substituting the entire series expansion \eqref{eq:u_entire_series_exp} into the heat equation \eqref{eq:heat_eq_markov} and observing that differentiation with respect to $x$ corresponds to a shift of the entire series coefficients:
$$
\partial_{xx}u(t, x) = \sum_{k\geq 0} \bu_t^{k+2} \frac{x^k}{k!} = \bracket{\bu_t\proj{11}}{\mathbbm{x}}.
$$
The claim that the solution $u$ to the heat equation is entire in $x$ for each $t \in [0, T]$ is not surprising; the representation \eqref{eq:u_psi_polyn} reads
\[
u(t, x) = \int_{\mathbb{R}} \exp({p(y)})\dfrac{1}{\sqrt{T-t}}\varphi\left(\dfrac{x - y}{\sqrt{T - t}}\right)\,dy, \quad t\in [0, T),\ x \in \mathbb{R},
\]
where $\varphi(z)=\frac{1}{{\sqrt{2\pi}}}\exp({-z^2/2})$ denotes the density of the standard normal random variable. Therefore, the solution $u(t, \cdot)$ inherits the regularity of the Gaussian density, and entireness can be proved under the assumption that  the polynomial $p$ satisfies $|\exp({p(\cdot)})| \leq \exp({a(|\cdot| + 1)})$ for some $a \in \mathbb{R}$ (see, e.g., \cite[Section~5.3]{cuchiero2025signaturesdesaffinepolynomial}).

This regularization effect is, however, very specific to the  Markovian setting and does not extend to the fully path-dependent framework associated with the signature of the time-augmented Brownian motion. There are two main challenges. First, one can no longer rely on the regularity of the density of the truncated signature $\sigW^{\leq N}$. In particular, \citet{zienkiewicz2003Noteanalyticregularity} shows that the corresponding heat kernel fails to be real-analytic as soon as $N \geq 3$. Second, the path-dependent heat semigroup does not regularize the dependence on the terminal condition in the sense of functional derivatives. In particular, discontinuities of the terminal condition are propagated by the semigroup, as illustrated by the example below.


\begin{sqexample}
    Consider 
    \[
    u(t, (X_s)_{s \in [0, t]}) = \mathbb{E}\Bigl[ \mathbf{1}_{\left\{\sup_{u \in [0, T]} W_u \ge 1 \right\}} \,\Big|\, W_s = X_s, \ s \in [0, t] \Bigr].
    \]
    Let 
    \[
    X^\varepsilon_s := 4(1 - \varepsilon) \frac{s}{t}\left(1 - \frac{s}{t}\right) \,, \quad s \in [0,t].
    \]
    Clearly, $X^\varepsilon \to X^0$ in the uniform topology as $\varepsilon \to 0$. {However, for all $\varepsilon > 0$,
    \[
    u(t, (X^{\varepsilon}_s)_{s \in [0,t]}) = \mathbb{P}\Bigl(\sup_{s \in [t,T]} W_{t, s} > 1\Bigr) < 1\,,
    \]
    while 
    \(
    u(t, (X_s^0)_{s \in [0,t]})) \equiv 1\,.
    \) Hence, the functional $u(t, \cdot)$ is not continuous for the uniform topology.}
\end{sqexample}

Therefore, in the absence of both an analytic density and a smoothing property for the path-dependent heat equation, our approach is instead to exploit directly the regularity of the terminal condition $\exp(p(x))$ and to reformulate the argument in algebraic terms so that it can be extended to the signature setting. The key point is that the argument does not fundamentally rely on the density. The following lemma, which holds for an arbitrary real-valued random variable $Y$ and does not {require} Gaussianity, illustrates the main idea underlying our approach. This argument will later be generalized in Lemma~\ref{lemma:enormous_bound}.

More specifically, the argument exploits the polynomial structure of $p$. We first expand $p(x+Y)$ as a polynomial in the deterministic variable $x$, with random coefficients depending on $Y$. This allows us to isolate the constant term $p(Y)$  from the higher-order terms in $x$. We then keep $\exp({p(Y)})$ in exponential form and treat the remaining part as a formal power series in $x$ using the Cauchy product. This yields an entire expansion in $x$ with random coefficients depending on $Y$. The key point then is to identify a suitable leading-order contribution of $p(Y)$ that dominates the modulus of the remaining terms, in order to justify the interchange of expectation with the infinite series, yielding the desired entire representation.


\begin{lemma}\label{lemma:enormous_bound_lite}
  Let $Y$ be a real-valued random variable. If the degree of the polynomial $N$ is even with leading coefficient $p_N$ such that  $\Re(p_N) < 0$, then the map $x \mapsto \mathbb{E}\left[\exp({p(x + Y)})\right]$ is entire; that is, there exists a complex sequence $\bu = (\bu^0, \bu^1, \ldots)$ such that
    \begin{align}\label{eq:charY}
\mathbb{E}\left[\exp({p(x + Y)})\right] = \sum_{k\geq 0} \bu^k \frac{x^k}{k!} = \bracket{\bu}{\mathbbm{x}} \,, \quad x \in \R\,.
   \end{align}
\end{lemma}

\begin{proof}
{The idea of the proof consists of rewriting $p(x+Y)$ as a polynomial in $x$ with coefficients that depend on $Y$:
\[
p(x + Y) = p_Y(x) = \sum_{k=0}^N p^{(k)}(Y)\dfrac{x^k}{k!} = \langle \bp|_{Y}, \mathbbm{x} \rangle,
\]
where $\bp|_{Y} = (p(Y), p'(Y), \ldots, p^{(N)}(Y), 0, \ldots)$. Note that the constant coefficient $(\bp|_{Y})^0 = p(Y)$ is itself a polynomial in $Y$ of degree $N$, while the other coefficients $(\bp|_{Y})^k = p^{(k)}(Y)$ for $k \geq 1$ are polynomials in $Y$ of degree $N-k$. We then split $p_Y(x)$ into the constant term $p(Y)$ and the remaining part: 
$$
p(x + Y) = p_Y(x) = p(Y) + \bar{p}_Y(x) = p(Y) + \langle \overline{\bp|_{Y}}, \mathbbm{x} \rangle,
$$
where $\overline{\bp|_{Y}} = (0, p'(Y), \ldots, p^{(N)}(Y), 0, \ldots)$ denotes the coefficients of $\bar{p}_Y = p_Y - p(Y)$. This allows us to isolate the dominant term in the exponential, while expanding $\exp({\bar{p}_Y(x)})$ into an entire series in $x$ using the Cauchy exponential \eqref{eq:cauchy_exp}:}
\begin{align}
\mathbb{E}\left[\exp({p(x + Y)})\right] = \mathbb{E}\left[\exp({p(Y) + \bracket{\overline{\bp|_{Y}}}{\mathbbm{x}}})\right] = \mathbb{E}\left[\bracket{\exp({p(Y)}) \exp^*(\overline{\bp|_{Y}})}{\mathbbm{x}}\right].
\end{align}
It remains to show that the expectation can be moved inside the bracket:
\begin{equation}\label{eq:expect_bracket_series}
    \mathbb{E}\left[\bracket{\exp({p(Y)}) \exp^*(\overline{\bp|_{Y}})}{\mathbbm{x}}\right] = \bracket{\mathbb{E}\left[\exp({p(Y)}) \exp^*(\overline{\bp|_{Y}})\right]}{\mathbbm{x}} =: \bracket{\bu}{\mathbbm{x}}.
\end{equation}
Setting $\bxi = \exp({p(Y)}) \exp^*\left(\overline{\bp|_{Y}}\right)$, it suffices to prove the convergence of the series 
\(
\mathbb{E}\left[\sum_{k = 0}^{\infty} |\bxi^k| \frac{|x|^k}{k!}\right]
\). A straightforward but cumbersome computation shows that 
\begin{equation}\label{eq:expectation_in_lem}
    \mathbb{E}\left[\sum_{k = 0}^{\infty} |\bxi^k| \frac{|x|^k}{k!}\right] \leq \mathbb{E}\left[\exp\left({\Re(p(Y)) + \sum_{k=1}^N|p^{(k)}(Y)|\frac{|x|^k}{k!}}\right)\right] < \infty,
\end{equation}
since the expression in the exponential is bounded from above, owing to the dominant term $p_N \frac{Y^N}{N!}$ with even $N$ and $\Re(p_N) < 0$, and the fact that all $|p^{(k)}(Y)|$ have a degree strictly less than $N$. This proves that \eqref{eq:expect_bracket_series} holds with $\bu = \mathbb{E}[\bxi]$ and completes the proof.
\end{proof}


We stress that the series in \eqref{eq:charY} contains infinitely many nonzero terms, even though $p$ has finite degree.

An important aspect of the strategy used in Lemma~\ref{lemma:enormous_bound_lite} is that it extends naturally to the signature of the time-augmented Brownian motion. In this setting, the role of the polynomial expansion is replaced by the algebraic structure of Lyndon words, which allows us to isolate suitable leading-order contributions within the signature that dominate the remaining terms in the same sense as in the polynomial case. By exploiting this structure, we show that any finite linear functional of the time-augmented signature of a one-dimensional path can be controlled by two terms (Lemma~\ref{lem:sig_bound}). This control is then used to identify, in Subsection~\ref{subsection:class_of_coefs}, the class $\mathcal{B}$ of coefficients $\bp$ ensuring that the analogue of the integrability estimate in \eqref{eq:expectation_in_lem} holds. For $\bp \in \mathcal{B}$, we prove Lemma~\ref{lemma:enormous_bound}, which is the signature counterpart of Lemma~\ref{lemma:enormous_bound_lite}, together with the technical bounds required for subsequent arguments. This leads to Theorem~\ref{theorem:analyticity_general}, which shows that the Fourier--Laplace transform of the time-augmented signature is entire and establishes uniform boundedness of the corresponding coefficients in time.
 
Combining Lemma~\ref{lemma:enormous_bound_lite} applied to $Y = W_T - W_t := W_{t,T}$ with Itô's formula,  we obtain that the corresponding coefficients $\bu_t$ appearing in \eqref{eq:charY} satisfy the equation~\eqref{eq:heat_eq_coef_power_ser}. This is collected in the following theorem, which provides a simplified power series formulation of the more general Theorem~\ref{theorem:existence_heat_eq} for time-augmented signatures.

\begin{theorem}\label{theorem:entireness_polynomial}
    Let $p$ satisfy the assumptions of Lemma~\ref{lemma:enormous_bound_lite}. Then, the Fourier--Laplace transform $u(t,\cdot)$ defined in \eqref{eq:u_psi_polyn} is entire in $x \in \mathbb{R}$, meaning that \eqref{eq:u_entire_series_exp} holds for some deterministic coefficients $(\bu_t)_{t \in [0, T]}$. Furthermore,  $(\bu_t)_{t \in [0, T]}$ satisfies \eqref{eq:heat_eq_coef_power_ser}.
\end{theorem}

The reader has, of course, noticed that the equation~\eqref{eq:heat_eq_coef_power_ser} satisfied by the coefficients $\bu_t$ is linear, and may wonder if it can be solved in closed form. Indeed, denoting $\mathscr{L}\colon \bu \mapsto \frac{1}{2}\bu\proj{11}$, one may expect the solution to be given by the operator exponential
\begin{equation}
    \bu_t = e^{(T - t)\mathscr{L}}\bu_T.
\end{equation}
As we show in Subsection~\ref{sect:explicit_sol_heat}, this is indeed the case for a certain, relatively small class of terminal conditions, but it does not hold in general. In particular, this approach fails in our case of interest $\bu_T = \exp^*\left(\bp\right)$, when $\deg(\bp) > 2$, as demonstrated by the following example.

\begin{sqexample}\label{ex:eplicit_sol_pol}
    Consider $p(x) = -x^4$. In this case,
    \begin{align*}
        u(T, x) = \exp({-x^4})=\sum_{n=0}^\infty \frac{(-1)^n}{n!} x^{4n} = \sum_{n=0}^\infty \frac{(-1)^n (4n)!}{n!} \frac{x^{4n}}{(4n)!} = \bracket{\bu_T}{\mathbbm{x}},
    \end{align*}
    with $\bu_T^{4n} = \frac{(-1)^n (4n)!}{n!}$ and $\bu_T^k = 0$ for $k \notin 4\mathbb{N}$.
    Let us compute the zeroth coefficient of $e^{T\mathscr{L}}\bu_T$:
    \begin{align*}
        (e^{T\mathscr{L}}\bu_T)^0 = \sum_{m=0}^\infty \frac{T^m}{m!} (\mathscr{L}^m \bu_T)^0 
        = \sum_{m=0}^\infty \frac{T^m}{m! 2^m} \bu_T^{2m} = \sum_{n=0}^\infty \frac{T^{2n}}{(2n)! 2^{2n}} \frac{(-1)^n (4n)!}{n!}.
    \end{align*}
    The last series is divergent, whereas $\bu_0^0 = u(0, 0) = \E[\exp({-W_T^4})] < \infty$. 
\end{sqexample}

    \begin{sqremark}
    \label{remark:non_uniqueness_Tychonoff}
        One cannot expect to obtain uniqueness for solutions of such an equation. A classical counter-example by \cite{tychonoff1935theoremes} shows that $(t,x) \mapsto \sum_{n \geq 0} g^{(n)}(t) \frac{x^{2n}}{(2n)!}\,$, with $g(t) := \exp(-1 / t^2)\,$, is a solution to the heat equation with null initial condition. Moreover, this power series has an infinite radius of convergence. Hence, $\bm u$ is not unique, even under the assumption that $\bracket{\bm u_t}{\mathbbm  x}$ converges for all $t \in [0\,, T]$ and $x \in \R\,$. However, we establish uniqueness of the solution $\bu_t$ to the linear equation when $\bracket{\bu_t}{\mathbbm{x}}$ satisfies certain growth conditions. This result is presented in Subsection~\ref{subsec:uniqueness_heat}.
    \end{sqremark}

\subsection{Log-Fourier--Laplace transform: an analytic series expansion}
\label{subsec:power_series_psi}
We now turn to the log-Fourier--Laplace transform $\psi(t, x) = \log(u(t, x))$. Our goal is to deduce from the series expansion \eqref{eq:u_entire_series_exp} of $u$ the existence of a sequence of coefficients $\bpsi_t = (\bpsi_t^0, \bpsi_t^1, \ldots)$ such that
\begin{equation}\label{eq:psi_entire_series_exp}
    \psi(t, x) = \sum_{k\geq 0} {\bpsi^k_t} \frac{x^k}{k!} = \bracket{\bpsi_t}{\mathbbm{x}},
\end{equation}
and to show that they satisfy the infinite-dimensional Riccati equation
\begin{equation}\label{eq:riccati_coef_power_ser}
\begin{cases}
    \dot \bpsi_t + \dfrac{1}{2}\bpsi_t\proj{11} + \dfrac{1}{2}(\bpsi_t\proj{1})^{*2} = 0, \quad t \in [0, T], \\
    \bpsi_T = \bp.
\end{cases}
\end{equation}
If $u(t, 0) = \bu_t^0 \neq 0$, then $\psi(t, \cdot) = \log u(t, \cdot)$ is analytic in a neighborhood of $0$. We define its coefficients via
\begin{equation}\label{eq:psi_def_power_series}
    \bpsi_t := \log^*\bu_t = \log \bu_t^0 + \sum_{n \geq 1}\dfrac{(-1)^{n+1}}{n}\left(\dfrac{\bu_t - \bu_t^0}{\bu_t^0}\right)^{*n},
\end{equation}
 where the complex logarithm is defined so as to ensure the continuity of $t \mapsto \log \bu_t^0$, see Theorem~\ref{theorem:log_analyticity_general}. 
From this, we deduce the analyticity of the log-Fourier–Laplace transform in a neighborhood of $0$ summarized in Theorem~\ref{theorem:entire_series_riccati}, which corresponds to the power-series analogue of Theorem~\ref{theorem:existence_riccati} in the full signature setting. The proof consists of algebraic verification that $\bpsi_t$ satisfies the Riccati equation \eqref{eq:riccati_coef_power_ser} given that $\bu_t$ solves the linear equation and can be viewed as an algebraic analogue of the Cole–Hopf transform at the level of the power series coefficients. 
\begin{theorem}\label{theorem:entire_series_riccati}
    Assume that $p$ satisfies the assumptions of Lemma~\ref{lemma:enormous_bound_lite} and that the coefficients $\bu_t$ constructed as in Theorem~\ref{theorem:entireness_polynomial}, satisfy $\bu_t^0 \neq 0$ for $t \in [0, T]$. Let $\bpsi_t$ be defined by \eqref{eq:psi_def_power_series}. Then, $\bpsi_t$ is a solution to the Riccati equation \eqref{eq:riccati_coef_power_ser} and there exists $r > 0$ such that the expansion \eqref{eq:psi_entire_series_exp} holds for all $|x| \leq r$.
\end{theorem}

\begin{sqremark}
Theorems~\ref{theorem:entireness_polynomial} and~\ref{theorem:entire_series_riccati} were proved in the entire power-series setting in \cite[Proposition 5.8]{cuchiero2025signaturesdesaffinepolynomial}. The argument there relies on regularity properties of the Gaussian density and does not readily extend to the signature setting. By contrast, our proof relies on a purely algebraic strategy as developed in  Lemma~\ref{lemma:enormous_bound_lite}. This algebraic viewpoint is crucial for extending the argument to the signature framework; see Sections~\ref{section:sig_expansions} and~\ref{sect:heat_and_riccati_sig}.
\end{sqremark}

We note that Theorem~\ref{theorem:entire_series_riccati} only yields a local expansion of $\psi(t,\cdot)$ around the origin. In particular, for the Fourier--Laplace transform
\begin{equation}\label{eq:laplace_illustration_series}
\log \mathbb{E}\left[\exp\left(p(W_T)\right)\Big| \mathcal{F}_t\right]
= \psi(t,W_t)
= \sum_{k\geq0} \psi_t^k \frac{W_t^k}{k!},
\end{equation}
the last equality is only guaranteed to hold on the event  $\{|W_t|<r\}$, that is, on the stochastic interval $[[0,\tau_r]]$, where
$$
\tau_r:=\inf\{t>0:\ |W_t|=r\}.
$$
One may wonder whether this locality is merely a limitation of our approach and whether a global representation can be obtained:
\begin{itemize}
    \item \textbf{The non-entireness of the logarithm:} we show that the first question has a negative answer. Indeed, the locality of the expansion is a genuine feature of the problem and stems from the nonlinear nature of the HJB equation \eqref{eq:riccati_eq_markov}. Although $u(t,\cdot)$ is entire, its logarithm can fail to be entire because $u(t,\cdot)$ may vanish in the complex plane. More precisely, even if $u(t,x)\neq 0$ for every $x\in\mathbb{R}$, there is in general no reason for its analytic continuation $u(t,z)$ to remain non-vanishing on $\mathbb{C}$.   By Little Picard's theorem, a non-constant entire function can omit at most one complex value. Therefore, for $\log u(t,\cdot)$ to be entire, the value $0$ must be omitted, which would force $u(t,\cdot)$ to be of the form $\exp(g)$ for some entire function $g$. Such a property is not satisfied in general. To illustrate this, we show in Theorem~\ref{theorem:zeros_quartic_exp_main}, that the map $
z \mapsto \mathbb{E}[\exp({-(z+W_{t,T})^4})]$
has infinitely many zeros on the imaginary axis. Consequently, its logarithm is analytic but not entire. The proof relies on a saddle-point analysis.
\item \textbf{A global representation:} while a global power-series expansion centered at the origin is generally impossible, we show that a  global representation can nevertheless be recovered by recentering the expansion. For any $x\in\mathbb{R}$,
\begin{equation}
\label{eq:psi_power_series_recentering}
p(x+W_{t,T})
= \sum_{n=0}^N p^{(n)}(x)\frac{W_{t,T}^n}{n!}
=: p_x(W_{t,T}),
\end{equation}
where $p_x$ is a polynomial whose leading coefficient coincides with that of $p$. Hence $p_x$ satisfies the assumptions of Lemma~\ref{lemma:enormous_bound_lite} whenever $p$ does, and all previous results apply with $p$ replaced by $p_x$. This yields
$$
\psi(t,x)=\psi^x(t,0)=(\bpsi_t^x)^0,
$$
where $\bpsi^x$ solves the same Riccati equation \eqref{eq:riccati_coef_power_ser} but  with the state-dependent terminal condition $$\bp_x = (p(x), p'(x), \ldots, p^{(N)}(x), 0, \ldots),$$
instead of  $\bp$. In this sense, the local representation can be transported to any point $x$, thereby yielding a global description of $\psi$ through a family of locally convergent expansions. This is extended in Subsection~\ref{section:recentering} to the more general signature framework.
\end{itemize}

\section{Preliminary results on signatures}\label{section:preliminaries}
This section develops the algebraic and analytic framework used throughout the paper. The central objects are path signatures, infinite sequences of iterated stochastic integrals, which naturally live in the extended tensor algebra. We introduce this structure, along with the tensor and shuffle products, respectively encoding the concatenation and the multiplication of iterated integrals, as well as the convergence tools required to give meaning to the infinite series that arise when computing dot products on the extended tensor algebra.

\subsection{Tensor algebra and path signatures}

\paragraph{Extended tensor algebra.} Let $\mathbb K$ denote either $\C$ or $\R\,$. {For $n\geq 0$, let $(\mathbb{K}^2)^{\otimes n}$ denote the $n$-th tensor power of $\mathbb{K}^2$, with $(\mathbb{K}^2)^{\otimes 0}=\mathbb{K}$. The symbol $\otimes$ denotes the tensor product. Iterated integrals of order $n$ of a two-dimensional path take value in $(\mathbb K^2)^{\otimes n}\,$.} Collecting all orders at once leads naturally to the \textit{extended tensor algebra}
$$
T((\mathbb K^2)) := \prod_{n \geq 0} (\mathbb K^2)^{\otimes n} = \{ \bm p = (\bm p_0\,, \bm p_1\,, \ldots\,, \bm p_n\,, \ldots)\colon\, \bm p_n \in (\mathbb K^2)^{\otimes n},\ n \geq 0 \}\,.
$$
This is an algebra under the product
$$
(\bm p \otimes \bm \ell)_n := \sum_{k = 0}^n\, {\bm p_k \otimes \bm\ell_{n-k}} \,, \quad n \geq 0\,, \quad \bm p\,, \bm \ell \in T((\mathbb K^2))\,.
$$
Let $(e_0\,, e_1)$ be the canonical basis of $\mathbb K^2\,$. For each $n \geq 0\,$, the tensors $e_{i_1} \otimes \ldots \otimes e_{i_n}\,$, ranging over all $i_1\,,\ldots\,,i_n \in \{0\,,1\}$ form a natural basis of $(\mathbb K^2)^{\otimes n}\,$. They are indexed by binary words $\word v := \word{i_1\ldots i_n}$  over the  alphabet $\{\word 0\,, \word 1\}\,$, i.e. $\word v \in V := \cup_{n \geq 1} \{\word 0\,,\word 1 \}^n\cup \{\emptyword\}\,$, where $\emptyword$ is the empty
word, of length zero. We identify each basis tensor $e_{i_1} \otimes \ldots \otimes e_{i_n}$ with the word $\word{i_1 \ldots i_n}\,$, so that the tensor product of two basis elements reduces to the concatenation of two words: $\word v \otimes \word w = \word{vw}\,$. Letters and words will be written in \textcolor{NavyBlue}{\textbf{blue}} throughout. Every element $\bm p \in T((\mathbb K^2))$ then expands as
$$
\bm p = \sum_{\word v}\bm p^{\word v}\cdot \word v \,,
$$
with scalar coefficients $\bm p^{\word v} \in \mathbb K\,$. For a word $\word v\,$, we write $|\word v|_{\word i}$ for the number of occurrences of the letter $\word i \in \{\word 0\,, \word 1\}$ in $\word v\,$, and $|\word v| := |\word v|_{\word 0} + |\word v|_{\word 1}$ for its length. The degree and partial degrees of $\bm p \in T((\mathbb K^2))$ are defined by 
$$
\deg(\bp) := \max \{|\word v|\colon\ \bp^{\word v} \neq 0 \}\,, \quad \text{and} \quad \deg_{\word{i}}(\bp) := \max \{|\word v|_{\word{i}}\colon\ \bp^{\word v} \neq 0 \}\,, \quad \word i \in \{\word 0\,, \word 1 \}\,,
$$
with the convention $\deg(0) = \deg_{\word i}(0) = - \infty\,$.

\begin{sqremark}
    Throughout this paper, we use $|\cdot|$ to denote the standard length of the word, coinciding with the number of letters. However, we note that all the results hold true if one defines the length by $2|\word v|_{\word 0} + |\word v|_{\word 1}$, which is more natural for the diffusive scaling of the Brownian motion. 
\end{sqremark}

Finally, the tensor algebra $T(\mathbb K^2) := \bigoplus_{n \geq 0} (\mathbb K^2)^{\otimes n}$ is the sub-algebra of $T((\mathbb K^2))$ consisting of elements of finite degree, i.e., with finitely many nonzero coefficients.

\paragraph{Path signatures.} We now introduce the central object.
\begin{definition}
    Let $(X_t)_{t \geq 0} = (X^0_t\,, X^1_t)_{t \geq 0}$ be a two-dimensional continuous semimartingale. For $0 \leq s \leq t\,$, the Stratonovich signature of $X$ over $[s,t]$ is the element $\sigX[s,t] \in T((\R^2))$ whose components are given by
    $$
    \sigX[s,t]^{\word{i_1...i_n}} := \int_{s \leq t_1 \leq \ldots \leq t_n \leq t} \circ dX^{i_1}_{t_1}\circ \ldots \circ dX^{i_n}_{t_n}\,, \quad i_1\,,...\,, i_n \in \{0\,,1\}\,,
    $$
    and $\sigX[s,t]^{\emptyword} := 1\,$. {We also use the shorthand notation $\sigX[t] := \sigX[0,t]$.}
\end{definition}
In other words, the path signature of $X$ is given by the collection of all possible iterated Stratonovich integrals between the two components of $X\,$, and the word index encodes the order of integration. For $\word i \in \{\word 0\,,\word 1\}$ and $n \geq 0\,$, one has the closed form $\sigX[s,t]^{\word i^{\otimes n}} = (X_t - X_s)^n / n!\,$, but other components depend in general on the whole path of $X$ between $s$ and $t\,$.

A fundamental property linking path signatures to the algebraic structure of $T((\R^2))$ is Chen's relation
\begin{equation}
\label{eq:Chen}
\sigX[s,u] \otimes \sigX[u,t] = \sigX[s,t]\,, \quad 0 \leq s \leq u \leq t\,.
\end{equation}
This identity shows how iterated integrals compose over adjacent time intervals, and follows directly from the additive structure of integration.

Throughout this paper, we work with the \textit{time-augmented signature} of a one-dimensional continuous semimartingale $(X_t)_{t \geq 0}\,$, defined as the path signature of the two-dimensional process $\widehat X := (t, X_t)\,$, and denoted by $\sigXhat[\cdot, \cdot]\,$. The letters $\word 0$ and $\word 1$ then correspond to integration against $dt$ and $\circ dX_t\,$, respectively. 

\subsection{Linear functionals of the signature}\label{subsec:linear_fct}

\paragraph{Linear functionals of the signature and shuffle product.} The extended linear algebra serves as the natural space of coefficients for linear functionals of the signature.
For $\bm p\,, \sigX[] \in T((\mathbb K^2))\,$, we define the pairing
\begin{equation}\label{eq:pairing_def}
    \bracket{\bm p}{\sigX[]} := \sum_{n \geq 0}\sum_{|\word v| = n} \bm p^{\word v} \sigX[]^{\word v}. 
\end{equation}
When either $\bm p$ or $\sigX[]$ is of finite degree, this is well-defined. In particular, $\bracket{\word v}{\sigX[]} = \sigX[]^{\word v}\,$. 
When both $\bm p$ and $\sigX[]$ are infinite, the pairing is well-defined when the following seminorm is finite:
\begin{equation}
\label{eq:def_norm_X}
\|\bp\|_{\sigX[]} := \sum_{n \geq 0}\, \left|\sum_{|\word v| = n}\,\bp^{\word v}\, \sigX[]^{\word v} \right| < \infty.
\end{equation}

Similarly to power series, path signatures have the remarkable algebraic property to make the product of linear functionals again a linear functional. More precisely, for any words $\word v\,, \word w \in V\,$,
\begin{equation}
\label{eq:shuffle_ppty_semimg}
 \bracket{\word v}{\sigX[s,t]} \bracket{\word w}{\sigX[s,t]} = \bracket{\word v \shuprod \word w}{\sigX[s,t]}\,, 
\end{equation}
where $\word v \shuprod \word w \in T(\R^2)$ is the \textit{shuffle product} of $\word v$ and $\word w\,$. This identity is the signature analogue of the classical Cauchy product for power series. 
\begin{definition}
    The shuffle product is defined recursively by
    $$
    \word {vi} \shuprod \word{wj} := (\word v \shuprod \word{wj}) \otimes \word i + (\word {vi} \shuprod \word{w}) \otimes \word j\,,
    $$
    for words $\word v\,, \word w \in V$ and letters $\word i\,, \word j \in \{\word 0\,, \word 1\}\,$, with the initialization $\word v \shuprod \emptyword = \emptyword \shuprod \word v = \word v\,$.
\end{definition}
Concretely, $\word v \shuprod \word w$ is the sum of all possible mixing of $\word v$ and $\word w$ that preserve the internal ordering of each word. As an example, $
\word{10}\shuprod \word{1} = \word{101} + \word{110} + \word{110}\,.
$
By linearity, the shuffle product extends to $T(\mathbb K^2)$ as
$$
\quad \bm p \shuprod \bell := \sum_{\word v, \word w} \bm p^{\word v} \bell^{\word w} \word v \shuprod \word w\,, \quad \bm p \,, \bell \in T(\mathbb K^2).
$$
\paragraph{Group-like elements.} Throughout the paper, we will denote by $G$ the set of \textit{group-like elements} of $T((\R^2))\,$, i.e., those that satisfy \eqref{eq:shuffle_ppty_semimg} for all pairs of words $\word{v}, \word{w} \in V$:
$$
G := \{\sigX[] \in T((\R^2))\colon\ \sigX[]^\emptyword = 1, \  \bracket{\word v}{\sigX[]} \bracket{\word w}{\sigX[]} = \bracket{\word v \shuprod \word w}{\sigX[]},\ \word v, \word w \in V \}\,.
$$
Hence, for any realization of the path signature of a continuous semimartingale, we have $\sigX[s,t](\omega) \in G\,$. However, it is important to note that $G \subsetneq T((\R^2))\,$, as for instance any basis element $\word v \neq \emptyword$ is not group-like. Formulating results in terms of group-like elements rather than signatures directly allows us to disentangle the algebraic structure from the probabilistic one, and yields cleaner statements throughout the paper. 

For instance, the shuffle property \eqref{eq:shuffle_ppty_semimg} can now be formulated for group-like elements $\mathbb{X} \in G$ and infinite coefficients $\bell, \bp \in T((\mathbb{K}^2))$:
\begin{equation}\label{eq:shuffle_generic}
    \bracket{\bell\shuprod\bp}{\sigX[]} = \bracket{\bell}{\sigX[]}\bracket{\bp}{\sigX[]},
\end{equation}
whenever $\|\bp\|_{\sigX[]} < \infty$ and $\|\bell\|_{\sigX[]} < \infty$,  thanks to an important property of the seminorm \eqref{eq:def_norm_X}, the shuffle-compatibility: for all $\sigX[] \in G$ and for all $\bp\,, \bm \ell \in T((\mathbb{K}^2))$ such that $\|\bp\|_{\mathbb{X}} < \infty$ and $\|\bell\|_{\mathbb{X}} < \infty$, we have 
\begin{equation}
\label{eq:def_shuffle_compatibility_norm_X}
\| \bp \shuprod \bell\|_{\sigX[]} \leq \|\bp \|_{\sigX[]} \|\bell \|_{\sigX[]}\,,
\end{equation}
see \cite*[Section 4.1]{cuchiero2025signaturesdesaffinepolynomial}. 
This submultiplicativity is pivotal for the estimates in Sections \ref{section:sig_expansions} and \ref{sect:heat_and_riccati_sig}.

The following lemma shows that pairings against a range of group-like elements are enough to uniquely identify elements of the extended tensor algebra.
\begin{lemma}\label{lem:uniq_coef}
Let $\bell \in \eTAC[2]\,$. Assume that there exists, $n \geq 1$ and $r > 0$ such that for all $\sigX[] \in G$ with $|\sigX[]^{\word v}| \leq r$ for all $1 \leq |\word v| \leq n\,$, the sum $\bracket{\bell}{\sigX[]}$ is convergent and null. Then $\bell = 0\,$.
\end{lemma}
\begin{proof}Let $\sigX[] \in G\,$. Then for any $\varepsilon > 0\,$, it is easy to verify that the dilation $\sigX[]^\varepsilon := (\varepsilon^{|\word v|} \sigX[]^{\wv})_{\word{v} \in V}$, belongs to $G$ as well. 
We can choose $\varepsilon_{\max} > 0$ such that $\max_{1 \leq |\word v| \leq n}|(\sigX[]^{\varepsilon})^{\wv}| \leq r$ for any $\varepsilon \in [0, \varepsilon_{\max}]\,$, so that
    $$
    \sum_{n \geq 0} \varepsilon^n \left(\sum_{|\wv| = n} \bell^{\wv} \sigX[]^{\wv} \right) = \bracket{\bell}{\sigX[]^{\varepsilon}} =  0\,, \quad \varepsilon \in (0\,, \varepsilon_{\max}]\,.
    $$
    Hence, the analytic function $\varepsilon\mapsto\sum_{n \geq 0}\varepsilon^n\left(\sum_{|\word{v}|=n}\bell^{\word{v}}\sigX[]^{\word{v}}\right)$ vanishes on $(0, \varepsilon_{\max}]$, which yields 
    \begin{equation}\label{eq:level_sum__zero}
        \sum_{|\word{v}|=n}\bell^{\word{v}}\sigX[]^{\word{v}} = 0, \quad n \geq 0\,, \quad \sigX[] \in G\,.
    \end{equation}

    Finally, for each word $\word{v} = \word{i_1\ldots i_n}$, take $\mathbb{X}(t_1, \ldots, t_n) = e^{\otimes t_1 \word{i_1}}\otimes\ldots\otimes e^{\otimes t_n \word{i_n}}$. It is group-like as a signature of a piece-wise linear path.
    Equation \eqref{eq:level_sum__zero} then reads
    \begin{equation}\label{eq:piecewise-linear-sig}
            \sum_{|\word{w}|=n}\bell^{\word{w}}\bracket{\word{w}}{e^{\otimes t_1 \word{i_1}}\otimes\ldots\otimes e^{\otimes t_n \word{i_n}}} \equiv 0, \quad t_1, \ldots, t_n \in \R.
    \end{equation}
    Observing that $\dfrac{\partial}{\partial t_k}e^{\otimes t_k \word{i_k}} = \word{i_k}e^{\otimes t_k \word{i_k}}$, we differentiate \eqref{eq:piecewise-linear-sig} using $\dfrac{\partial^n}{\partial t_1 \ldots \partial t_n}$ and set $t_1 = \ldots = t_n = 0$:
    $$
    \sum_{|\word{w}|=n}\bell^{\word{w}}\bracket{\word{w}}{\word{i_1\ldots i_n}} = \bell^{\word{v}} = 0.
    $$
    Since $\word{v}$ was chosen arbitrary, this concludes the proof.
\end{proof}

The set $G$ of group-like elements indeed forms a group with respect to the tensor product $\otimes\,$, whose neutral element is $\emptyword\,$. Moreover, if $\mathbb{X} \in G$ corresponds to a signature of path $X = (X_s)_{s \in [0, T]}$, then its inverse $\mathbb{X}^{-1} \in G$, defined as the unique element such that 
$$
\mathbb{X} \otimes \mathbb{X}^{-1} = \mathbb{X}^{-1} \otimes \mathbb{X} = \emptyword,
$$
is given by the signature of the time-reversal path $\overleftarrow{X} = (X_{T - s})_{s \in [0, T]}$. This follows immediately from the uniqueness result of \citet*{boedihardjo2016signature}.

\paragraph{Shuffle exponential and logarithm.}
For $\bell \in T((\mathbb{K}^2))$, we define the shuffle power recursively by
$$
\bell\shupow{n} = \bell\shupow{n-1}\shuprod\bell,\quad n \geq 1, \quad
\bell\shupow{0} = \emptyword.
$$
For $\bell \in T((\mathbb{K}^2))$ such that $\bell^{\emptyword} = 0$, we define the
shuffle exponential and the shuffle logarithm by
$$
\shuexp{\bell} := \sum_{n \geq 0}\dfrac{\bell\shupow{n}}{n!}, \qquad
\log^{\shuprod}(\emptyword + \bell) := \sum_{n \geq 1} (-1)^{n-1} \dfrac{\bell\shupow{n}}{n}.
$$
The condition $\bell^{\emptyword} = 0$ ensures that each element of
$\shuexp{\bell}$ and $\log^{\shuprod}(\emptyword + \bell)$ is given by a finite
sum, so that the objects are well-defined. One can verify that for $\bell \in T((\mathbb{K}^2))$ with $\bell^{\emptyword} = 0$,
$$
\shuexp{\log^{\shuprod}(\emptyword + \bell )} = \emptyword + \bell 
\quad \text{and} \quad 
\log^{\shuprod}(\shuexp{\bell}) = \bell.
$$

In the case $\bell^\emptyword=0$, this property follows purely from the algebraic properties of formal power series and holds over any field $\mathbb{K}$.
For a general $\bell \in T((\mathbb{K}^2))$, we define its shuffle exponential as
\begin{equation}\label{eq:shuexp_def}
    \shuexp{\bell} := \exp({\bell^{\emptyword}})\shuexp{{\overline{\bell}}},
    \quad \text{where} \quad \overline{\bell} := \bell - \bell^{\emptyword}\emptyword.
\end{equation}
Similarly, for $\bell \in T((\mathbb{K}^2))$ with $\bell^{\emptyword} \neq 0$, the
shuffle logarithm is defined by
$$
\log^{\shuprod}(\bell) := \log(\bell^{\emptyword})
+ \log^{\shuprod}\!\left(\emptyword + \dfrac{\overline\bell}{\bell^{\emptyword}}\right)
= \log(\bell^{\emptyword})\,\emptyword
+ \sum_{n \geq 1} \frac{(-1)^{n-1}}{n}
\left(\frac{\overline{\bell}}{\bell^{\emptyword}}\right)\shupow{n}.
$$
Here $\log(\bell^{\emptyword}) \in \mathbb{K}$ can be multi-valued when $\mathbb{K} = \C$.
Hence, for generic
coefficients, the relationship
$$
\shuexp{\log^{\shuprod}(\bell)} = \bell, \quad \bell \in T((\mathbb{K}^2)), \quad \bell^{\emptyword} \neq 0,
$$
still holds, while
$$
\log^{\shuprod}(\shuexp{\bell}) = \bell + 2\pi i k\,\emptyword,
\quad \bell \in T((\mathbb{K}^2)),
$$
now holds up to the constant term $2\pi i k\,\emptyword$ for some $k \in \mathbb{Z}$.
We stress that this ambiguity in the choice of the complex logarithm concerns only the constant term of the shuffle
logarithm.

For elements $\sigX[] \in G$ and $\|\bell\|_{\mathbb{X}} < \infty$, the
exponential of a linear functional can be expressed as a linear functional using the
shuffle property \eqref{eq:shuffle_generic}:
$$
\exp(\bracket{\bell}{\sigX[]}) = \bracket{\shuexp{\bell}}{\sigX[]}.
$$
Similarly, when, in addition, $\bell^{\emptyword} \neq 0$ and
$\left\|\frac{\overline{\bell}}{\bell^{\emptyword}}\right\|_{\sigX[]} < 1$, we have
$$
\log(\bracket{\bell}{\sigX[]}) = \bracket{\log^{\shuprod}(\bell)}{\sigX[]}.
$$
This follows from the shuffle-compatibility \eqref{eq:def_shuffle_compatibility_norm_X} of the seminorm $\|\cdot\|_{\mathbb{X}}$ and the definition of the shuffle logarithm.


\paragraph{Convergence classes.} Having introduced the algebraic structure, we now address its analytic counterpart by developing our main tools to study the convergence of the infinite series emerging from dot products. For any $\sigX[] \in G\,$, we define
\begin{equation}
    \label{eq:def_A_X}
\mathcal A_{\sigX[]} := \left\{\bell \in \eTAC[2] \colon\ \left\|\bell \right\|_{\sigX[]} < \infty \right\}
\end{equation}

When $\bm p \in T(\C^2)$ is of finite degree, the size of a group-like element $\sigX[] \in G$ relative to $\bm p$ is given by
\begin{equation}
\label{eq:def_M_p}
M_{\bp}(\sigX[]) := \max_{1 \leq |\word v| \leq \deg(\bp)}\, \left|\sigX[]^{\word v} \right|,
\end{equation}
and the associated ball of radius $r > 0$ in $G$ by
\begin{equation}
    \label{eq:def_G_p_r}
G_{\bp, r}(\mathbb R^2) := \left\{\sigX[] \in G\colon\ M_{\bp}(\sigX[]) \leq r  \right\}\,.
\end{equation}

Similarly to the power series case, this allows us to define the class of coefficients with convergence radius $r$ as
\begin{equation}\label{eq:class_A_p_r}
    \mathcal A_{\bp, r} := \left\{\bell \in \eTAC[2]\colon\, \|\bell\|_{\sigX[]} < \infty, \text{ for all } \sigX[] \in G_{\bp, r} \right\},
\end{equation}
while $\mathcal A_{\infty} := \bigcap_{\sigX[] \in G}\,\mathcal A_{\sigX[]}$ corresponds to linear functionals with infinite radius.

\subsection{Left and right shifts and Itô's formula}
\label{subsec:ito}
For a word $\word w \in V\,$, the left and right shifts of $\bp \in T((\C^2))$ by $\word w$ are defined as:
\begin{equation}
    \label{eq:def_single_word_shifts}
\bp |_{\word w} := \sum_{\word v}\,\bp^{\word v \word w}\cdot \word v\,, \quad \text{and} \quad ~_{\word w}|\bp := \sum_{\word v}\,\bp^{\word w \word v}\cdot \word v\,.
\end{equation}
They act as the tensor algebra analogues of differentiation: in the power series case, right-shifting by $\word 1$ sends $\bracket{\bm p}{\mathbbm x} = \sum_{n \geq 0} \bm p^n x^n /n!$ to $\bracket{\bm p|_{\word 1}}{\mathbbm x} = \sum_{n \geq 0} \bm p^{n+1} x^n / n !\,$, corresponding to differentiation with respect to $x\,$. In the full signature setting, they arise naturally when differentiating $\bracket{\bm p}{\sigXhat[0,t]}$ via Itô's formula, as detailed in Theorem \ref{thm:ito_formulas}. More precisely, shifting $\bm p$ by $\word 0$ and $\word 1$ in the path-dependent linear functional $F((X_s)_{s \leq t}) := \bracket{\bm p}{\sigXhat[0,t]}$ would correspond to, respectively, the time and space derivatives of $F$ as defined in \cite{Dupire04052019}.
More generally, if $\bp|_{\word v} \in \mathcal A_{\sigX[]}$ for all $\word v\,$, we define
\begin{equation}
    \label{eq:def_right_X_shift}
\bp|_{\sigX[]} :=  \sum_{\word v} \sigX[]^{\word v}\cdot \bp|_{\word v} = \sum_{\word v} \bracket{~_{\word v}|\bp}{\sigX[]}\cdot \word v \,, 
\end{equation}
and if $~_{\word v}|\bp \in \mathcal A_{\sigX[]}$ for all $\word v\,$,
\begin{align}\label{eq:def_left_X_shift}
\quad ~_{\sigX[]}|\bp :=  \sum_{\word v} \sigX[]^{\word v} \cdot ~_{\word v}|\bp = \sum_{\word v}\,\bracket{\bp|_{\word v}}{\sigX[]}\cdot \word v \,.
\end{align}
The shift operations are dual to (left and right) tensor multiplication, as shown in the following proposition.
\begin{proposition}\label{prop:projections_ppties}
    \begin{enumerate}
        \item If $\bp \in \TAC[2]$, and $\sigX[]\,, \sigY[] \in \eTA[2]\,$, then $\bracket{\bp}{\sigX[] \otimes \sigY[]} = \bracket{~_{\sigX[]}|\bp}{\sigY[]} = \bracket{\bp|_{\sigY[]}}{\sigX[]}\,$.
        \item If $\bp\proj{v} \in \mathcal{A}_{\sigX[]}$ for all $\word{v} \in V$, then $\bracket{\bp}{\sigX[]\otimes \word{v}} = \bracket{~_{\sigX[]}|\bp}{\word{v}} = \bracket{\bp|_{\word{v}}}{\sigX[]}\,$, for all $\bm p \in \eTAC[2]\,$.
    \end{enumerate}
\end{proposition}
\begin{proof}
    We first note that for all $\sigX[] \in T((\R^2))$, $\bp \in \TA[2]$ and $\word{v} \in V$,
    $$
    \bracket{\bp}{\sigX[]\otimes\word{v}} = \sum_{\word{w}}\bp^{\word{wv}}\sigX[]^{\word{w}} = \bracket{\bp\proj{v}}{\sigX[]},
    $$
    and the corresponding identity holds for the left projection. The first statement then follows from
    $$
    \bracket{\bp}{\sigX[] \otimes \sigY[]} = \sum_{\word{v}}\sigY[]^{\word{v}}\bracket{\bp}{\sigX[] \otimes \word{v}} = \sum_{\word{v}}\sigY[]^{\word{v}}\bracket{\bp\proj{\word{v}}}{\sigX[]} = \bracket{\bp|_{\sigY[]}}{\sigX[]}.
    $$
    For the second part of the statement, the condition $\bp|_{\word v} \in \mathcal A_{\sigX[]}$ ensures that the pairings are well-defined and are all equal to
    $$
    \sum_{n \geq 0}\sum_{|\word{w}| = n}\bp^{\word{wv}}\sigX[]^{\word{w}} < \infty.
    $$
\end{proof}
\begin{sqremark}
    \label{remark:recentering} The duality $\bracket{\bp}{\sigX[] \otimes \sigY[]} = \bracket{~_{\sigX[]}|\bp}{\sigY[]} = \bracket{\bp|_{\sigY[]}}{\sigX[]}$ from Proposition \ref{prop:projections_ppties} is equivalent to recentering in the power series case. That is, noticing that $\mathbbm x \otimes \mathbbm y = (1\,, x + y\,, (x + y)^2 / 2\,, \ldots\,,(x + y)^n / n!\,, \ldots)\,$,
    $$
    \bracket{\bm p}{\mathbbm x \otimes \mathbbm y} = \sum_{n \geq 0} p^n \frac{(x + y)^n}{n!} = \sum_{n \geq 0} \left(\sum_{k \geq 0}  p^{k+n} \frac{x^k}{k!} \right) \frac{y^n}{n!} = \bracket{{}_{\mathbbm x}|\bm p}{\mathbbm y}\,.
    $$
    This relation is at the core of Section \ref{section:recentering}.
\end{sqremark}

We provide a version of Itô's formula, derived by \citet[Theorem 3.3]{abi2024path} and slightly generalized for our needs.

Let $X = (X^0, X^1)$ denote a continuous semimartingale and let $H = (H^0, H^1)$ denote the corresponding bounded variation part. The following class describes the time-dependent coefficients $(\bell_t)_{t \in [0, T]}$ (possibly infinite) such that Itô's formula can be applied to the process $t \mapsto \bracket{\bell_t}{\sigX}$.

We say that $(\bell_t)_{t \in [0, T]}$ belongs to $\mathcal{I}^r(X)$ if the following conditions are satisfied:
\begin{enumerate}
    \item $\norm{\bell_t}_{\sigX} < \infty$ for $t \in [0, T]$ a.s.
    \item For all $\word{v} \in V$, the map $t \mapsto \bell_t^{\word{v}}$ is $C^1$ on $[0\,,T]$ and 
    \[
    \int_{0}^{T} \|\dot{\bell}_t\|_{\sigX} dt < \infty \quad \text{a.s.}
    \]
    \item For all $\word{i}, \word{j} \in \{\word{0}, \word{1}\}$ such that $\bracket{X^i}{X^j} \not\equiv 0$, we have 
    \[
    \bell_t\proj{i}, \, \bell_t\proj{ji}, \, \dot{\bell}_t \in \mathcal{A}_{\sigX}, \quad t \in {[0, T]} \quad \text{ a.s.,}
    \]
    and 
    \[
    \int_{0}^{T} \norm{\bell_t\proj{i}}_{\sigX} |dH^i_t| + 
    \int_{0}^{T} \norm{\bell_t\proj{ji}}_{\sigX} d\bracket{X^i}{X^j}_t + 
    \int_{0}^{T} \norm{\bell_t\proj{i}}_{\sigX}^2 d\langle X^i \rangle_t < \infty \quad \text{a.s.}
    \]
\end{enumerate}
We denote by $\mathcal{I}^l(X)$ the class of time-dependent coefficients satisfying conditions 1--3 with the left shifts $~_{\word{i}}|\bell_t$ and $~_{\word{ij}}|\bell_t$ in condition 3. This class allows us to apply Itô's formula to the process $t \mapsto \bracket{\bell_t}{\sigX[t]^{-1}}$.
\begin{theorem}[Itô's formulas]\label{thm:ito_formulas}
        Let $(\bell_t)_{t\in[0, T]} \in \mathcal{I}^r(X)$. Then, $(\bracket{\bell_t}{\sigX})_{t \in [0, T]}$ is an Itô process and, for all $0 \leq s \leq t \leq T$,
    \begin{equation}\label{eq:right_ito}
        \bracket{\bell_t}{\sigX[t]} = \bracket{\bell_s}{\sigX[s]}  +
        \int_s^t\bracket{\dot\bell_u}{\sigX[u]}du
        +\sum_{\word{i}\in \{\word{0}, \word{1}\}}\int_s^t\bracket{\bell_u\proj{i}}{\sigX[u]}dX_u^i + \dfrac12\sum_{\word{i}, \word{j}\in \{\word{0}, \word{1}\}}\int_s^t\bracket{\bell_u\proj{ji}}{\sigX[u]}d\bracket{X^i}{X^j}_u.
    \end{equation}    
    Similarly, if $(\bell_t)_{t\in[0, T]} \in \mathcal{I}^l(X)$, then $(\bracket{\bell_t}{\sigX[t]^{-1}})_{t \in {[0, T]}}$ is an Itô process and, for all $0 \leq s \leq t \leq T$,
\begin{equation}\label{eq:left_ito}
        \bracket{\bell_t}{\sigX[t]^{-1}} = \bracket{\bell_s}{\sigX[s]^{-1}}  +
        \int_s^t\bracket{\dot\bell_u}{\sigX[u]^{-1}}du
        -\sum_{\word{i}\in \{\word{0}, \word{1}\}}\int_s^t\bracket{~_{\word{i}}|\bell_u}{\sigX[u]^{-1}}dX_u^i + 
        \dfrac12\sum_{\word{i}, \word{j}\in \{\word{0}, \word{1}\}}\int_s^t\bracket{~_{\word{ij}}|\bell_u}{\sigX[u]^{-1}}d\bracket{X^i}{X^j}_u.
    \end{equation} 
\end{theorem}
\begin{proof}
    The proof of \eqref{eq:right_ito} follows from \cite[Theorem 3.3 and Corollary 3.4]{abi2024path}. Regarding the second Itô formula, we observe that
    \[
    0 = d(\sigX\otimes\sigX[t]^{-1}) = d\sigX\otimes\sigX[t]^{-1} + \sigX\otimes d \sigX[t]^{-1} = \sigX[t]\otimes\circ d X_t\otimes\sigX[t]^{-1} + \sigX\otimes d \sigX[t]^{-1}.
    \]
    Multiplying by $\sigX^{-1}$ on the left yields the dynamics of $\sigX[t]^{-1}$: 
    \[
    d\sigX[t]^{-1} = -\circ dX_t\otimes\sigX[t]^{-1}.
    \]
    The remainder of the proof follows the steps of the proof of \citet[Theorem 3.3 and Corollary 3.4]{abi2024path}, noting that the left shifts are dual operations to the left tensor products (see Proposition~\ref{prop:projections_ppties}).
\end{proof}

\section{Signature expansions of the Fourier-Laplace transform and its logarithm}
\label{section:sig_expansions}

In this section, we establish the signature expansion of the Fourier--Laplace and log-Fourier--Laplace transforms
$$
\E\left[\exp \left(\bracket{\bm p}{\sigW[T]}\right) \, \Big|\, \mathcal F _t\right] = \bracket{\bu_t}{\sigW[t]} = \exp{\left(\bracket{\bm \psi_t}{\sigW[t]}\right)}\,,
$$
where $\bm p \in T(\C^2)\,$, $\sigW[]$ denotes the signature of the time-augmented Brownian motion $\widehat{W}_t := (t, W_t)\,$, and $\bm u_t\,, \bm \psi_t \in T((\C^2))$ are deterministic coefficients. 

The strategy follows and extends the algebraic approach introduced in Sections \ref{sect:expansion_u_power_ser} and \ref{subsec:power_series_psi} for power series, and proceeds in three steps:\\

\textbf{Step 1.} We first introduce in Section \ref{subsection:class_of_coefs} a class $\mathcal B \subset T(\C^2)$ of admissible coefficients for which the real part of the linear function $\sigXhat[] \mapsto \bracket{\bm p}{\sigXhat[]}$ is bounded from above over all time-augmented signatures $\sigXhat[]\,$. This extends to the signature setting the condition of having a leading polynomial term of even degree with negative real part in Lemma \ref{lemma:enormous_bound_lite}. The required upper bounds are established via Lemma \ref{lem:sig_bound}, which relies on Lyndon word decompositions.\\

\textbf{Step 2.} For $\bm p \in \mathcal B\,$, we exploit the same factorization idea as in Lemma \ref{lemma:enormous_bound_lite} to isolate the leading-order terms in the exponential, and treat the remaining part as a formal signature expansion via the shuffle product. This justifies the interchange of expectation and infinite summation, yielding an entire expansion  of the Fourier--Laplace transform in terms of signature coordinates, that is, $\bracket{\bu_t}{\sigW[t]}$ with deterministic coefficients $\bu_t \in \mathcal{A}_{\infty}\,$, as established in Theorem~\ref{theorem:analyticity_general}.\\

\textbf{Step 3.}  We construct $\bm \psi$ as the tensor algebra analogue of the algebraic logarithm \eqref{eq:psi_def_power_series} {of $\bu$}. As in Theorem~\ref{theorem:entire_series_riccati}, we obtain{, in Theorem~\ref{theorem:log_analyticity_general},} that the infinite sum $\bracket{\bm \psi_t}{\sigX[]}$ has a finite radius of convergence: it converges whenever $M_{\bm p}(\sigX[]) < r\,$, for some $r > 0\,$. Remarkably, the radius condition on $\sigX[]$ only depends on $|\sigX[]^{\word v}|$ for $1 \leq |\word v| \leq \deg(\bm p)\,$. Theorem \ref{theorem:zeros_quartic_exp_main} shows that this limitation is structural: zeros of $\bracket{\bm u_t}{\cdot}$ appear on the complex plane as soon as one leaves the quadratic Gaussian framework, preventing the logarithm from being entire. We will address this issue with a recentering technique in Section \ref{section:recentering}.
\subsection{The class of admissible coefficients}\label{subsection:class_of_coefs}

In this subsection, we construct a class $\mathcal B$ of coefficients $\bp \in T(\C^2)$ of finite degree, such that the real part of the linear functional $\sigXhat[] \mapsto \bracket{\bm p}{\sigXhat[]}$ is bounded from above uniformly over all time-augmented signatures $\sigXhat[]$ of continuous semimartingales $X\,$ {on $[0, T]$}. In particular, for such elements   the Fourier-Laplace transform $\E[\exp (\bracket{\bm p}{\sigXhat[]})]$ is finite. 

We proceed step by step, starting from polynomials of a single variable and their time integrals, and gradually build up to the general class $\mathcal{B}$ introduced in Definition~\ref{def:class_B}.

\begin{sqexample}[Polynomials of $X_T$]\label{ex:polynomials_of_W} 
    {First, consider the case in which $\bracket{\bp}{\sigXhat[T]} = p(X_T)$ for some polynomial $p$ of degree $N$. 
    This implies that $\bp$ contains nonzero coefficients corresponding only to the words $\word{v}$ of the form $\word{v} = \word{1}\conpow{k}$ for $k = 0, 1, \ldots, N$. 
   As in Lemma \ref{lemma:enormous_bound_lite}, $\Re(p)$ is bounded from above if the degree of $p$ is even, $N = 2n$, with leading coefficient such that $\Re(\bp^{\word{1}\conpow{N}}) < 0$.
    In other words, $\bp$ can be written in the form
\begin{equation}\label{eq:polynomial_coef}
        \bp = -\alpha \word{1}\conpow{2n} + \bq\,, \qquad \Re(\alpha) > 0,\quad \deg_{\word{1}}(\bq) < 2n,  \quad \deg_{\word{0}}(\bq) = 0.
    \end{equation}}
\end{sqexample}

In the general path-dependent setting, identifying a class of coefficients $\bm p$ satisfying this boundedness property is more subtle. Fortunately, Lemma~\ref{lem:sig_bound} shows that the time-augmented signature of a one-dimensional process enjoys bounds analogous to those of ordinary polynomials. Specifically, each component $\sigXhat[T]^{\word{v}}$ can be controlled by $|X_T|^{|\wv|_{\word 1}}$ and $\int_0^T |X_s|^{|\wv|_{\word 1}}\,ds$, whereas components corresponding to words of the form $\word{v0}$ can be controlled only by the integral term. This lemma plays a central role throughout the paper. Its proof relies on \cite{Lyndon1954} words, a subset of $V$ such that any word decomposes as a shuffle-polynomial of Lyndon words, due to the \cite{Radford1979ANR} theorem; see Appendix~\ref{sect:proof_of_sig_bound}. This allows to rewrite any signature coefficient as a polynomial of primitive coefficients using \eqref{eq:shuffle_ppty_semimg}, enabling us to establish the bounds recursively.

 \begin{lemma}\label{lem:sig_bound}
    Let $X = (X_t)_{t \in [0, T]}$ be a continuous semimartingale and let $\sigXhat[]$ denote its time-augmented signature. Then, for any word $\word{v} \in V\,$, there exists a constant $C_{\word v, T} \geq 0\,$, such that
    \begin{equation}\label{eq:sig_bound}
    \begin{aligned}
        \left|\sigXhat[s,t]^{\word{v}}\right| &\leq C_{\word{v}, T}\,\left(1 + \left|X_{s,t}\right|^{|\wv|_{\word 1}} + \int_s^t|X_{s,r}|^{|\wv|_{\word 1}}dr\right), \quad s \leq t \leq T\,, \\
        \left|\sigXhat[s,t]^{\word{v0}}\right| &\leq C_{\word{v}, T}\left(1 + \int_s^t|X_{s,r}|^{|\wv|_{\word 1}}dr\right), \quad s \leq t \leq T\,.
    \end{aligned}
    \end{equation}

\end{lemma}
\begin{proof}
    The proof is provided in Appendix~\ref{sect:proof_of_sig_bound}.
\end{proof}
 
 This allows us to construct the following examples of path-dependent linear functionals with real part  bounded from above.

\begin{sqexample}[Integrated linear functionals]\label{ex:integrated_func}
    {An integrated functional of the form
    $$
    \bracket{\bm r }{\sigXhat[T]}= \bracket{\bm r'\word 0}{\sigXhat[T]} = \int_0^T \bracket{\bm r'}{\sigXhat[s]}\,ds
    $$
    can be controlled by $\int_0^T|X_s|^{\deg_{\word 1}(\bm r)}\,ds$, in virtue of Lemma \ref{lem:sig_bound}. Similarly to Example \ref{ex:polynomials_of_W}, adding an integrated dominant coefficient as
    \begin{equation}
        \bp := (-\beta\word{1}\conpow{2m} + \br)\word{0}, \qquad \Re(\beta) > 0,\quad  \deg_{\word{1}}(\br) < 2m,
    \end{equation}
    makes $\bracket{\Re(\bp)}{\sigXhat[0,T]}$ bounded from above, since $\bracket{\Re(\bm p)}{\sigXhat[0,T]} \leq C_{\bm p, T} \int_0^T (- \Re(\beta) |X_t|^{2m} +  |X_t|^{\deg_{\word 1}(\bm r)}) dt\,$, which is the integral of a polynomial that is bounded from above over $\R\,$. Remarkably, $\deg_{\word 0}(\bm r)$ is left free, as long as it remains finite.}
\end{sqexample}

\begin{sqexample}[Generic linear functional]\label{ex:generic_funct}
    {Finally, a generic linear functional $\bracket{\bm p}{\sigXhat[0,T]}$ is bounded from above if both dominant terms from Lemma \ref{lem:sig_bound} with even degrees are present. Namely, for some $n, m \in \N$ and $\alpha, \beta \in \C$, $\bp$ has the form
    $$
    \bp = -\alpha \word{1}\conpow{2n} - \beta \word{1}\conpow{2m}\word{0} + \bm s, \qquad \Re(\alpha) > 0, \quad \Re(\beta) > 0, \quad \deg_{\word{1}}(\bm s) < 2m \land 2n.
    $$
    Once more, we only require that $\deg_{\word 0}(\bm s)$ is finite. }
\end{sqexample}

We now define a class encompassing all three examples considered above.

\begin{definition}\label{def:class_B}
    We say that $\bp \in \mathcal{B} \subset T(\C^2)$ if there exist three complex numbers
    $$
    \alpha, \beta \in \{z \in \C\colon \Re(z) > 0\} \cup \{ 0\}, \quad \gamma \in \C,
    $$
    two natural numbers $m, n \in \N^*$, and three coefficients $\bq, \br, \bm s \in T(\C^2)$ satisfying
    \begin{align}\label{eq:class_B_coefs_cond}
        \deg_{\word{1}}(\bq) < 2n,  \quad \deg_{\word{0}}(\bq) = 0, \quad \deg_{\word{1}}(\br) < 2m, \quad \deg_{\word{1}}(\bm s) < 2m \land 2n,
    \end{align}
    such that 
    \begin{equation}\label{eq:class_B_p}
        \bp = -\alpha(\word{1}\conpow{2n} + \bq) - \beta(\word{1}\conpow{2m} + \br)\word{0} + \alpha\beta\bm s + \gamma \emptyword.
    \end{equation}
    {The time degrees $\deg_{\word 0}(\bm r)$ and $\deg_{\word 0}(\bm s)$ are only required to be finite.}
\end{definition}

\begin{sqremark}
    Definition~\ref{def:class_B} recovers all three examples considered above: Example~\ref{ex:polynomials_of_W} when $\beta = 0$, Example~\ref{ex:integrated_func} when $\alpha = 0$, and Example~\ref{ex:generic_funct} when $\alpha, \beta \neq 0$ and $\bq = \br = 0$. Note, however, that the class $\mathcal{B}$ is more flexible than the one given in Example~\ref{ex:generic_funct}, as it allows for the inclusion of the coefficients $\bq$ and $\br$ of degrees greater than $2m \land 2n$.
    In particular, we leverage this property later in Section~\ref{sect:laplace_sig_mart}.
\end{sqremark}

\begin{sqremark}
    The last term in \eqref{eq:class_B_p}, $\gamma\emptyword$, allows for the addition of constant terms to the integrated functionals when $\alpha = 0$ and is used only to ensure stability of the class $\mathcal{B}$ established in Lemma~\ref{lemma:recentering_property_class_B}.
\end{sqremark}

\begin{sqremark}
    Note that while $\alpha$ and $\beta$ can vanish, they are not allowed to be purely imaginary. Having a positive real part is crucial for obtaining the estimates presented later in this section. 
\end{sqremark}

The coefficients from $\mathcal{B}$ define signature polynomials whose real parts are bounded from above when applied to time-augmented signatures. More precisely, the following proposition holds.

\begin{proposition}
\label{prop:boundedness_dot_product_B}
{Fix $T > 0$ and $\bp \in \mathcal B$. There exists a constant $C_T > 0$ such that for any continuous $\R$-valued semimartingale $X = (X_t)_{t \in [0, T]}$ and for all $0 \leq s \leq t \leq T$,}
$$
\bracket{\Re(\bp)}{\widehat{\mathbb{X}}_{s, t}} \leq C_T\,,\, a.s.
$$
\end{proposition}
\begin{proof}
    The polynomial $\widehat{\mathbb{X}} \mapsto\bracket{\Re(-\alpha(\word{1}\conpow{2n} + \bq))}{\widehat{\mathbb{X}}}$ is bounded from above by Example~\ref{ex:polynomials_of_W}, while $\widehat{\mathbb{X}}\mapsto\bracket{\Re(-\beta(\word{1}\conpow{2m} + \br)\word{0})}{\widehat{\mathbb{X}}}$ is bounded by Example~\ref{ex:integrated_func}. Finally, when $\alpha \neq 0$ and $\beta \neq 0$, boundedness of the remaining terms is ensured by both leading terms since $\widehat{\mathbb{X}} \mapsto\bracket{\Re(-\alpha\word{1}\conpow{2n}-\beta\word{1}\conpow{2m} \word{0} + \alpha\beta\bm s + \gamma \emptyword)}{\widehat{\mathbb{X}}}$ is bounded from above by Example~\ref{ex:generic_funct}.
\end{proof}

The class $\mathcal B$ is natural in the following two aspects. First, it is large enough to cover natural applications: the conditions defining $\mathcal B$ are closely related to the necessary and sufficient conditions for Dol\'eans--Dade exponentials of finite linear functions of the Brownian signature to be true martingales (see \cite*{jaber2025martingalepropertymomentexplosions}). In particular, the coefficient $\bp$ arising in Corollary \ref{cor:sig_vol_Fourier_laplace} belongs to $\mathcal B$ precisely under the martingality conditions. Second, $\mathcal B$ is stable under the recentering operation: the following Lemma shows that it is closed under left shifts, defined in \eqref{eq:def_left_X_shift}, which is the signature analogue of the fact that $p_x$ in \eqref{eq:psi_power_series_recentering} has the same leading term as $p\,$. This stability is what makes the recentering argument of Section \ref{section:recentering} work, since recentering replaces $\bp$ by $~_{\sigX[]}|\bp\,$, and one must ensure that the coefficient remains in $\mathcal B$ for the new expansion to be valid.

\begin{lemma}
\label{lemma:recentering_property_class_B}
    If $\bp \in \mathcal{B}$ and $\sigX[] \in T((\C^2))$ such that $\sigX[]^{\emptyword} = 1$, then $~_{\sigX[]}|\bp \in \mathcal{B}$.
\end{lemma}
\begin{proof}
    Since $\bp \in \mathcal{B}$, it can be written in the form \eqref{eq:class_B_p} with the coefficients $\bq, \br,$ and $\bm s$ satisfying \eqref{eq:class_B_coefs_cond}. 
    By the definition of $~_{\sigX[]}|\bp$ in \eqref{eq:def_left_X_shift}, we have 
    \begin{align}
        ~_{\sigX[]}|\bp &=  -\alpha(\word{1}\conpow{2n} + \bq) - \beta(\word{1}\conpow{2m} + \br)\word{0} + \alpha\beta\bm s + \gamma\emptyword\\
        &+ \sum_{0 < |\word{v}| \leq \deg(\bp)}\sigX[]^{\word{v}}\left(-\alpha(~_{\word{v}}|\word{1}\conpow{2n} + ~_{\word{v}}|\bq) - \beta(~_{\word{v}}|(\word{1}\conpow{2m}\word{0} ) + ~_{\word{v}}|(\br\word{0}) )+ \alpha\beta~_{\word{v}}|\bm s\right)
    \end{align}
    Note that for each nonempty $\word{v}$, the term $~_{\word{v}}|(\word{1}\conpow{2m}\word{0}) + ~_{\word{v}}|(\br\word{0})$ is of the form $\tilde \br_{\word{v}}\word{0} + \tilde\gamma_{\word{v}}\emptyword$ for some $\tilde \br_{\word{v}} \in T(\C^2)$ and $\tilde\gamma_{\word{v}} \in \C$ such that $\deg_{\word{1}}(\tilde\br_{\word{v}}) < 2m$. We also observe that $\deg_{\word{1}}(~_{\word{v}}|\word{1}\conpow{2n}+ ~_{\word{v}}|\bq ) < 2n$.
    Hence, $~_{\sigX[]}|\bp$ can also be written in the form 
    $$
        ~_{\sigX[]}|\bp = -\alpha(\word{1}\conpow{2n} + \tilde\bq) - \beta(\word{1}\conpow{2m} + \tilde\br)\word{0} + \alpha\beta\tilde{\bm s} + \tilde\gamma\emptyword,
    $$
    where we define
    \begin{align}
        \tilde\bq &= \bq +  \sum_{0 < |\word{v}| \leq \deg(\bp)}\sigX[]^{\word{v}}(~_{\word{v}}|\word{1}\conpow{2n}+ ~_{\word{v}}|\bq), \quad 
        &&\tilde\br = \br +  \sum_{0 < |\word{v}| \leq \deg(\bp)}\sigX[]^{\word{v}}\tilde \br_{\word{v}}, \\
        \tilde{\bm s} &= \bm s + \sum_{0 < |\word{v}| \leq \deg(\bp)}\sigX[]^{\word{v}}~_{\word{v}}|{\bm s}, \quad 
     &&   \tilde\gamma = \gamma - \beta \sum_{0 < |\word{v}| \leq \deg(\bp)} \sigX[]^{\wv}\tilde\gamma_{\word{v}},
    \end{align}
    with $\tilde\bq, \tilde\br, \tilde{\bm s}$ satisfying \eqref{eq:class_B_coefs_cond}. This completes the proof. 
\end{proof}

\subsection{Signature expansion of the Fourier--Laplace transform}\label{subsection:signature_expansion_u}

{Given that the increments of the Brownian motion are independent, the conditional Fourier--Laplace transform $\E[\exp (\bracket{\bm p}{\sigW[T]})\,|\,\mathcal F_t]$ can be expressed} as a functional of the time-augmented signature up until time $t\,$, thanks to Chen's identity \eqref{eq:Chen}: 
\begin{align}\label{eq:charchen}
\E\left[\exp \left(\bracket{\bm p}{\sigW[T]}\right)\,\bigl|\,\mathcal F_t \right] = u(t, \sigW[t])\,,\quad \text{where}\quad u(t, \sigX[]) := \E\left[\exp \left(\bracket{\bm p}{\sigX[] \otimes \sigW[t,T] }\right) \right]\,, \quad \sigX[] \in G\,.
\end{align}
Our aim in this section is to prove that the function $u$ is in fact linear with respect to $\mathbb X$ with an infinite radius of convergence when $\sigX[]$ ranges over $G$. More precisely, we show that  $u(t, \sigX[]) = \bracket{\bm u_t}{\sigX[]}\,$ for some coefficients $\bm u_t \in \mathcal A_{\infty}\,$. As mentioned in Section~\ref{sect:expansion_u_power_ser}, the difficulty with the path-dependent setting is that there is no Gaussian density to provide regularization. Instead, regularity must come entirely from the functional $\sigX[] \mapsto \exp (\bracket{\bm p}{\sigX[]})$ itself. The class $\mathcal B$ defined in Section \ref{subsection:class_of_coefs} is precisely designed to make this exponential sufficiently integrable.

The core idea is to extend the strategy outlined before Lemma~\ref{lemma:enormous_bound_lite} by exploiting algebraic properties of the signature. Concretely, we expand
$\exp\big(\bracket{\bm p}{\sigX[] \otimes \sigW[t,T]}\big)$
under the expectation as a linear functional on $\mathbb X$ of the form $\bracket{\bm \xi_t}{\sigX[]}$, and then apply Fubini's theorem to interchange expectation and linear functional, yielding
$u(t,\mathbb X) = \bracket{\mathbb E[\bm \xi_t]}{\sigX[]}.$ As already noted in Section~\ref{sect:expansion_u_power_ser}, in order to justify this interchange one should not expand the full expression directly. Instead, one factorizes the linear functionals to isolate the term independent of $\mathbb X$ from higher-order terms in $\mathbb X$. We then keep  the former in exponential form and treat the latter as a formal  signature expansion in $\sigX[]$ using the shuffle product, the natural extension of the Cauchy product to the signature setting.

In our setup, after an application of Proposition~\ref{prop:projections_ppties}, we get 
$$\exp (\bracket{\bm p}{\sigX[] \otimes \sigW[t,T]}) = \exp (\bracket{\bm p |_{\sigW[t,T]}}{\sigX[]}) = \bracket{\shuexp{\bm p|_{\sigW[t,T]}}}{\sigX[]},
$$ {and the leading term that we will factorize is given by $\exp((\bp|_{\sigW[t,T]})^{\emptyword}) = \exp(\bracket{\bp}{\sigW[t,T]})\,$.} {This is a very natural step given the definition \eqref{eq:shuexp_def} of the shuffle exponential.} The precise estimates we obtain before taking the expectation are given in the following lemma.


\begin{lemma}\label{lemma:enormous_bound}
    Let $\bm p \in \mathcal B\,$ be of the form \eqref{eq:class_B_p}, and define
    \begin{equation}
    \label{eq:def_xi}
    \bxi_t := \shuexp{\bm p|_{\sigW[t,T]}} = \exp({\bracket{\bm p}{\sigW[t,T]}})\cdot \shuexp{\overline{\bm p|_{\sigW[t,T]}}}\,.
    \end{equation}
    Then, there exist two constants $C_{\bm p, T}, C_{\bm p, T}' > 0\,$, independent of $W$, such that for all $t\in[0, T]$ and $\sigX[] \in G$,
        \begin{align}\label{eq:enormous1}
        \|\bxi_t\|_{\sigX[]} \leq \exp \left(\bracket{\Re(\bm p)}{\sigW[t,T]} + M_{\bp}(\sigX[])\,\zeta_{\bm p,t, T}\right) \leq \exp\left(C_{\bm p, T} \,(1 + M_{\bm p}(\sigX[]))^{\deg_{\word 1}(\bm p)}\right),  
        \end{align}
        where $M_{\bp}$ is defined in \eqref{eq:def_M_p} and
        $$
        \zeta_{\bm p,t, T} := C_{\bm p, T}'\,\left(1 + |W_T - W_t|^{2n-1}\,\indic{\alpha \neq 0} + \int_{t}^T\,|W_s - W_t|^{2m-1}\,ds\,\indic{\beta \neq 0}  \right),
        $$
        with $n, m$ and $\alpha, \beta$ from the definition \eqref{eq:class_B_p} of $\bp$.
        In particular, $\bxi_t \in \mathcal A_{\infty}$ for all $t \in [0\,,T]\,$, and 
        $$
        \sup_{t \in [0, T]}\,\|\bxi_t\|_{\sigX[]}  < + \infty\,, \quad \sigX[] \in G\,.
        $$
\end{lemma}

\begin{proof}
    Let $\sigX[] \in G\,$. Since the seminorm $\| \cdot\|_{\sigX[]}$ is shuffle-compatible (see \eqref{eq:def_shuffle_compatibility_norm_X}), we get that $\|\shuexp{\bell}\|_{\sigX[]} \leq \exp(\|\bell\|_{\sigX[]})$, so that 
    \begin{align}\label{eq:proofxi}
    \|\bm \xi_t\|_{\sigX[]} \leq \exp\left(\Re\left(\bracket{\bm p}{\sigW[t,T]}\right) + \left\|\overline{\bm p |_{\sigW[t,T]}}\right\|_{\sigX[]}\right) =\exp\left(\bracket{\Re(\bm p)}{\sigW[t,T]}+ \left\|\overline{\bm p |_{\sigW[t,T]}}\right\|_{\sigX[]}\right)  \,, \quad t \in [0, T].    
    \end{align}
    We now bound the quantity $\left\|\overline{p |_{\sigW[t,T]}}\right\|_{\sigX[]}\,$.
    Recalling that $\overline{\bell} = \bell - \bell^{\emptyword}\emptyword\,$, we have $\left(\overline{\bm p|_{\sigW[t,T]}}\right)^{\emptyword} = 0\,$, and for $\wv \neq \emptyword$, $\left(\overline{\bm p|_{\sigW[t,T]} }\right)^{\word v} = \left(\bm p|_{\sigW[t,T]}\right)^{\word v} = \sum_{\word w} \bm p^{\word{vw}}\,\sigW[t,T]^{\word w}$ from \eqref{eq:def_right_X_shift}. This gives us the following form for the seminorm
\begin{align*}
        \left\|\overline{\bm p |_{\sigW[t,T]}}\right\|_{\sigX[]} = \sum_{n = 1}^{+ \infty}\, \left|\sum_{|\word v| = n}\,(\bm p|_{\sigW[t,T]} )^{\word v}\, \sigX[]^{\word v} \right| 
        = \sum_{n = 1}^{+ \infty}\, \left|\sum_{|\word v| = n}\,\sum_{\word w}\,\bm p^{\word{vw}}\,\sigW[t,T]^{\word w} \, \sigX[]^{\word v} \right|\,, \quad t \in [0, T],
    \end{align*}
    which is a finite sum since $\bm p$ has finite degree. The latter can be bounded as follows
    \begin{equation}
    \label{eq:bound_norm_p_proj_Y_before_B}
    \left\|\overline{\bm p|_{\sigW[t,T]}}\right\|_{\sigX[]} \leq \sum_{|\word v| = 1}^{\deg(\bm p)}\, \sum_{|\word w| = 0}^{\deg(\bm p) - |\word v|}\,\left|\bm p^{\word{vw}}\, \sigW[t,T]^{\word w}\,\sigX[]^{\word{v}} \right| \leq M_{\bm p}(\sigX[])\,\sum_{|\word v| = 1}^{\deg(\bm p)}\, \sum_{|\word w| = 0}^{\deg(\bm p) - |\word v|}\,\left|\bm p^{\word{vw}}\, \sigW[t,T]^{\word w} \right|\,, \quad t \in [0, T],
    \end{equation}
so that we need to bound the non-zero terms $|\bm p^{\wv \word w} \sigW[t,T]^{\word v}|$ for $\word v$ non-empty. From Definition \ref{def:class_B}, we write $\bm p$ as 
\begin{equation}
    \label{eq:form_of_p_proof_lemma}
    \bp = - \alpha (\word 1^{\otimes 2n} + \bq) - \beta (\word 1^{\otimes 2m} + \br)\word 0 + \alpha \beta \bm s + \gamma \emptyword\,.
\end{equation}
We will use the structure of $\bm p$ in \eqref{eq:form_of_p_proof_lemma} along with Lemma \ref{lem:sig_bound} to control the signature coefficients.
 Let $\word v \neq \emptyword\,$, and $\word w$ such that $\bm p^{\word{vw}} \neq 0\,$. Several cases arise  according to where the word $\word{vw}$ appears in the decomposition \eqref{eq:form_of_p_proof_lemma}:
\begin{itemize}
        \item If $\alpha \neq 0$ and $\word{v w}$ appears in $\word 1^{\otimes 2n} + \bm q\,$, then $\word w = \word 1^{\otimes k}$ for some $k < 2n - 1\,$. Thus there exists a constant $C_{\word w} \geq 0$ depending on $\word w$ such that
        $$
        |\sigW[t,T]^{\word w}| \leq C_{\word w}\,\left(1 + |W_T - W_t|^{2n-1}\right)\,, \quad t \in [0, T].
        $$
        \item If $\beta \neq 0$ and $\word{v w}$ appears in $\word 1^{\otimes 2m}\word 0 + \bm r\word 0\,$, then  $|\word w|_{\word 1} \leq 2m-1\,$. From Lemma \ref{lem:sig_bound}, there exists a constant $C_{\word w, T} \geq 0$ such that 
        $$
        |\sigW[t,T]^{\word w}| \leq C_{\word w, T} \,\left(1 + \int_{t}^T\,\left| W_s - W_t\right|^{2m - 1}\,ds \right)\,, \quad t \in [0, T].
        $$
        \item If $\alpha \beta \neq 0$ and $\word{v w}$ appears in $\bm s\,$, then $|\word w|_{\word 1} \leq 2(n \wedge m) - 1\,$, and from Lemma \ref{lem:sig_bound} 
        $$
        |\sigW[t,T]^{\word w}| \leq C_{\word w, T}\,\left(1 + |W_T - W_t|^{2(n\wedge m)-1} + \int_{t}^T\,|W_s - W_t|^{2(n \wedge m)-1}\,ds \right)\,, \quad t \in [0, T].
        $$
    \end{itemize}
    Combining all the above cases, we can find a constant $C_{\bm p, T}' \geq 0\,$, such that  
    $$
    |\bm p^{\word v \word w}\sigW[t,T]^{\word w}| \leq C_{\bm p, T}'\,\left(1 + |W_T - W_t|^{2n-1}\,\indic{\alpha \neq 0} + \int_{t}^T\,|W_s - W_t|^{2m-1}\,ds\,\indic{\beta \neq 0} \right)\,, \quad t \in [0, T]\,,
    $$
    for any $\word v \neq \emptyword$ and $\word w \in V\,$.
    Plugging this into \eqref{eq:bound_norm_p_proj_Y_before_B} leads to
    \begin{align*}
    \left\| \overline{\bm p |_{\sigW[t,T]}}\right\|_{\sigX[]} &\leq M_{\bm p}(\sigX[])\,C_{\bm p, T}'\,\left(1 + |W_T - W_t|^{2n-1}\,\indic{\alpha \neq 0} + \int_{t}^T\,|W_s - W_t|^{2m-1}\,ds\,\indic{\beta \neq 0}  \right) \\
    &=: M_{\bm p}(\sigX[])\zeta_{\bm p,t, T}, \quad t \in [0, T]\,,
    \end{align*}
    for another, not relabeled, nonnegative constant $C_{\bm p, T}'$.  Combining with \eqref{eq:proofxi} we obtain the first inequality in \eqref{eq:enormous1}.
    
    We now prove the second inequality in \eqref{eq:enormous1}. To this end, we exploit the presence of the term $\bracket{\Re(\bm p)}{\sigW[t,T]}$ in \eqref{eq:proofxi}  and observe that the decomposition \eqref{eq:form_of_p_proof_lemma} leads to
    \begin{equation}\label{eq:big_align_bound}
     \begin{aligned}
        \bracket{\Re(\bm p)}{\sigW[t,T]} + M_{\bp}(\sigX[])\,\zeta_{\bm p,t, T} &= -\frac{\Re(\alpha)}{(2n)!} \,|W_T - W_t|^{2n} - \frac{\Re(\beta)}{(2m)!}\,\int_t^T\,|W_s - W_t|^{2m}\,ds \\
        &\quad - \Re(\bracket{\alpha\bm q + \beta\bm r\word 0 - \alpha \beta \bm s - \gamma \emptyword}{\sigW[t,T]}) + M_{\bp}(\sigX[])\zeta_{\bm p,t, T} \\
        &\leq -\frac{\Re(\alpha)}{(2n)!} \,|W_T - W_t|^{2n} - \frac{\Re(\beta)}{(2m)!}\,\int_t^T\,|W_s - W_t|^{2m}\,ds \\
        &\quad + C_{\bm p, T}'\,(1  + M_{\bm p}(\sigX[]))\,\Bigl(1 + |W_T - W_t|^{2n-1}\,\indic{\alpha \neq 0}\\
        &\quad+ \int_{t}^T\,|W_s - W_t|^{2m-1}\,ds \,\indic{\beta \neq 0} \Bigr)\,, \quad t \in [0, T]\,,
    \end{aligned}   
    \end{equation}
   where the inequality follows from the same arguments as above to bound $\Re(\bracket{\alpha\bm q + \beta\bm r\word 0 - \alpha \beta \bm s - \gamma \emptyword}{\sigW[t,T]})$, up to a possibly different constant $C_{\bm p,T}' \geq 0$, which we do not relabel. Recall that $\alpha\,,\beta \in \{z \in \C\,:\, \Re(z) > 0\} \cup \{0\}$ in Definition \ref{def:class_B}, so that $\indic{\alpha \neq 0} = \indic{\Re(\alpha) > 0}\,$ and $\indic{\beta \neq 0} = \indic{\Re(\beta) > 0}\,$.
   
   We are going to use the leading terms $- |W_T - W_t|^{2n}$ and $-\int_t^T |W_s -W_t|^{2m}ds$ in \eqref{eq:big_align_bound} to bound the previous uniformly in $W\,$. For $n \geq 1\,$, $\gamma > 0$ and $\delta \geq 0\,$, the function $x \mapsto -\gamma x^{2n} + \delta x^{2n-1}$ is bounded on $\R$ by $\delta^{2n} / \gamma^{2n - 1}\,$.      {Regrouping the terms in \eqref{eq:big_align_bound} and applying this bound twice, we obtain
    \begin{align*}
        \bracket{\Re&(\bm p)}{\sigW[t,T]} + M_{\bp}(\sigX[])\,\zeta_{\bm p,t, T}  \\ &\leq 
    \tilde C_{\bp, T}\left(\indic{\Re(\alpha) > 0}  (1 + M_{\bp}(\sigX[]))^{2n} + \indic{\Re(\beta) > 0}\,(1 + M_{\bp}(\sigX[]))^{2m} + (1 + M_{\bp}(\sigX[]))\right) \\
    &\leq
    C_{\bp, T}(1 + M_{\bp}(\sigX[]))^{2n \vee 2m} = C_{\bp, T} (1 + M_{\bp}(\sigX[]))^{\deg_{\word 1}(\bp)},
    \end{align*}
    for some constants $\tilde C_{\bp, T}, C_{\bm p, T} \geq 0$.
    {Injecting this in the first inequality in \eqref{eq:enormous1} finishes the proof.}
    }
\end{proof}

The conclusion of Lemma \ref{lemma:enormous_bound} is that $\|\bm \xi_t \|_{\sigX[]}$ is controlled by $M_{\bm p}(\sigX[])$ uniformly in $t \in [0, T]$ and in $\sigW[t,T]$. This allows one to take the expectation of $\bm \xi_t$ and interchange with the infinite sum, yielding $\bm u_t := \E[\bm \xi_t] \in \mathcal{A}_{\infty}\,$.

\begin{theorem}\label{theorem:analyticity_general}
    Let $\bm p \in \mathcal B\,$, and $\bm \xi$ be as in \eqref{eq:def_xi}. Then, defining $\bm u_t := \E[\bm \xi_t]$ for all $t \in [0\,,T]\,$, we have:
    \begin{enumerate}
        \item $\bm u_t \in \mathcal A_{\infty}$ for all $t \in [0\,,T]\,$, so that
        \begin{align}\label{eq:charsigX}
         \E \left[\exp\left(\bracket{\bm p}{\sigX[] \otimes \sigW[t,T]}\right) \right] = \bracket{\bm u_t}{\sigX[]}\,, \quad \sigX[] \in G\,, \quad t \in [0, T]\,.
    \end{align}
    In particular,
\begin{align}\label{eq:charsigexpansion}
\E\left[\exp \left(\bracket{\bm p}{\sigW[T]}\right)\, \Big|\, \mathcal F_t\right] = \bracket{\bu_t}{\sigW[t]}, \quad t \in [0, T].     \end{align}
    \item For all $\sigX[] \in G\,$, $\sup_{t \in [0, T]}\,\|\bm u_t\|_{\sigX[]} \leq \exp\left(C_{\bm p, T} \,(1 + M_{\bm p}(\sigX[]))^{\deg_{\word 1}(\bm p)}\right) < + \infty\,$.
    \item $t \mapsto \bm u_t^{\emptyword}$ is continuous on $[0\,,T]\,$. 
    \end{enumerate}
\end{theorem}
\begin{proof} We first notice that \eqref{eq:charsigexpansion} readily follows from \eqref{eq:charsigX} for $\mathbb X = \widehat{\mathbb W}_t$ thanks to  \eqref{eq:charchen}.
We recall that $\exp (\bracket{\bm p}{\sigX[] \otimes \sigW[t,T]}) = \bracket{\bm \xi_t}{\sigX[]}\,$, which is well-defined in virtue of Lemma \ref{lemma:enormous_bound}. Hence, for any $\sigX[] \in G\,$,
$$
\sup_{t \in [0, T]}\,\|\bm u_t\|_{\sigX[]} \leq \E\big[\sup_{t \in [0, T]}\, \|\bm \xi_t\|_{\sigX[]}\big] \leq \exp\left(C_{\bm p, T} \,(1 + M_{\bm p}(\sigX[]))^{\deg_{\word 1}(\bm p)}\right)  < + \infty\,,
$$
and by Fubini's theorem,
$$
\bracket{\bm u_t}{\sigX[]} = \bracket{\E[\bm \xi_t]}{\sigX[]} = \E[\bracket{\bm \xi_t}{\sigX[]}] = \E\left[\exp (\bracket{\bm p}{\sigX[] \otimes \sigW[t,T]})\right]\,,
$$
which concludes the proof of the first two points. Finally, we remark that 
$$
\bm u_t^{\emptyword} = \E[\bm \xi_t^{\emptyword}] = \E[\exp (\bracket{\bm p}{\sigW[t,T]})]\,, \quad t \in [0, T]\,,
$$
is the expectation of a continuous function. Using the domination 
$$
|\bm \xi_t^{\emptyword}| = \|\bm \xi_t\|_{\emptyword} \leq \exp\left(C_{\bm p, T} \,(1 + M_{\bm p}(\emptyword))^{\deg_{\word 1}(\bm p)}\right) = \exp(C_{\bp, T})\,, \quad t \in [0, T]\,,
$$
we conclude by the dominated convergence theorem that $t \mapsto \bm u_t^{\emptyword}$ is continuous on $[0,T]$.
\end{proof}

To our knowledge, Theorem \ref{theorem:analyticity_general} is the first instance of a convergent expansion $\E[\exp \bracket{\bm p}{\sigW[T]}\,\vert\,\mathcal F_t] = \bracket{\bm u_t}{\sigW[t]}$ in the full path-dependent setting. We stress that proving convergence of this infinite sum would not have been possible without taking $\bm p$ in an appropriate class, such as $\mathcal B\,$.

\subsection{Local signature expansion of the log-Fourier--Laplace transform}
\label{subsec:existence_Riccati_stopping_time}

We now aim to obtain a similar signature expansion for the log-Fourier--Laplace transform, i.e.
$$
\E \left[\exp \left(\bracket{\bm p}{\sigX[] \otimes \sigW[t,T]}\right) \right] = \exp (\bracket{\bm \psi_t}{\sigX[]})\,, \quad t \in [0, T]\,.
$$
This can be achieved {in a neighborhood of $\mathbb{X} = \emptyword$} whenever $\bm u_t^{\emptyword} = u(t, \emptyword) = \E[\exp (\langle \bm p\,, \sigW[t,T]\rangle)]$ is non-zero, which is always true when $\bm p$ has real-valued coefficients or when $T-t$ is small enough.

As in Section \ref{subsec:power_series_psi}, $\bm \psi$ is constructed algebraically from $\bm u\,$, via the tensor algebra analogue of \eqref{eq:psi_def_power_series}, the shuffle logarithm.
For $\bm u$ given by Theorem \ref{theorem:analyticity_general}, since $t \mapsto \bm u_t^{\emptyword}$ is continuous, and assuming that the latter never hits zero on $[0\,,T]\,$, we choose the unique {complex logarithm  $\bm \psi^{\emptyword} := \log \bm u^{\emptyword}$ such that $\bu_t^{\emptyword} = \exp(\bpsi_t^{\emptyword})$ for all $t \in [0, T]\,$, $t \mapsto \bpsi_t^{\emptyword}$ is continuous, and  $\bm \psi_T^{\emptyword} = \bm p^{\emptyword}\,$}. This yields
\begin{equation}
\label{eq:def_log_shuffle_psi}
\bm \psi_t := \log^{\shuprod} (\bm u_t) = \bm \psi_t^{\emptyword} \emptyword + \sum_{n \geq 1}\frac{(-1)^{n-1}}{n} \left(\frac{\overline{\bm u}_t}{\bm u^{\emptyword}_t} \right)\shupow{n}, \quad t \in [0, T]\,.
\end{equation}

Similarly to the power series case, the expansion of the logarithm generally has a finite radius of convergence.

\begin{theorem}\label{theorem:log_analyticity_general}
    Let $\bm p \in \mathcal B\,$, and $\bm u$ be as in Theorem \ref{theorem:analyticity_general}. Assume that $\bu_t^{\emptyword} =\E[\exp(\bracket{\bm p}{\sigW[t,T]})] \neq 0$ for all $t \in [0\,,T]\,$, and define  $\bm \psi := \log^{\shuprod}(\bm u)$ as in \eqref{eq:def_log_shuffle_psi}. It holds that:
    \begin{enumerate}
        \item For all $t \in [0,T]$ and $\sigX[] \in G$ such that $\|\bm u_t\|_{\sigX[]} < 2 |\bm u_t^{\emptyword}|\,$, we have $\bm \psi_t \in \mathcal A_{\sigX[]}$ and 
        $$
        \E\left[\exp\left(\bracket{\bm p}{\sigX[] \otimes \sigW[t,T]}\right)\right] = \exp (\bracket{\bm \psi_t}{\sigX[]})\,.
        $$
        \item For any $\varepsilon \in (0\,,1)\,$, there exists $r = r(\bm p, T, \varepsilon) >0\,$, such that $\|\bm u_t\|_{\sigX[]} \leq (2-\varepsilon)|\bm u_t^{\emptyword}|$ for all $t \in [0\,,T]$ and $\sigX[] \in G_{\bm p, r}\,$, so that $\bm \psi_t \in \mathcal A_{\bm p, r}\,$. Moreover,
        $$
        \sup_{t \in [0, T]} \|\bm \psi_t\|_{\sigX[]} \leq \sup_{t \in [0, T]} |\bm \psi_t^{\emptyword}| + \log\left({\varepsilon}^{-1}\right) < + \infty\,, \quad \sigX[] \in G_{\bm p, r}.
        $$
        The Fourier--Laplace transform {thus} reads 
        \begin{align}\label{eq:logcharsigexpansion}
        \E\left[\exp \left(\bracket{\bm p}{\sigW[T]}\right) \, |\, \mathcal F _t\right] = \exp(\bracket{\bpsi_t}{\sigW[t]}),     
        \end{align}
        for all $t \in [0, T]$ such that $|M_{\bp}(\sigW)| \leq r$. In particular, \eqref{eq:logcharsigexpansion} holds for all  $t \in \llbracket  0, \tau \rrbracket$, where $\tau$ is a positive stopping time given by
        \begin{equation}\label{eq:tau_def_log_fourier}
                \tau := \inf\{t \geq 0\colon\ |M_{\bm p}(\sigW[t])| \geq r \}.
        \end{equation}
        \item $t \mapsto \bm \psi_t^{\emptyword}$ is continuous on $[0\,, T]\,$.
    \end{enumerate}
\end{theorem}
\begin{proof}
We note that the third point is immediate from the definition of $\bm \psi^{\emptyword}$ above \eqref{eq:def_log_shuffle_psi}.

We start by proving the first point of the theorem.
   By definition, $\bm u_t = \shuexp{\bm \psi_t}\,$, so that if $\|\bm \psi_t\|_{\sigX[]} < + \infty$ for some $\sigX[] \in G\,$, we immediately have, by Theorem \ref{theorem:analyticity_general}, that 
    $$
    \exp({\bracket{\bm \psi_t}{\sigX[]}}) = \bracket{\bm u_t}{\sigX[]} = \E\left[\exp (\bracket{\bm p}{\sigX[] \otimes \sigW[t,T]}) \right]\,.
    $$
    Taking the seminorm $\|\bm \cdot\|_{\sigX[]}$ in \eqref{eq:def_log_shuffle_psi} and using its shuffle-compatibility (see \eqref{eq:def_shuffle_compatibility_norm_X}) yields
    \begin{align*}
        \| \bm \psi_t\|_{\sigX[]} &\leq |\bm \psi^{\emptyword}_t| + \sum_{n = 1}^{+ \infty}\, \frac{1}{n}\left(\dfrac{\|\overline{\bm u}_t\|_{\sigX[]}}{|\bm u_t^{\emptyword}|}\right)^n \,, \quad t \in [0, T]\,.
    \end{align*}
    Using the fact that $\|\overline{\bm u}_t\|_{\sigX[]} = \| \bm u_t\|_{\sigX[]} - |\bm u_t^{\emptyword}|\,$, it is clear that the sum converges whenever $\|\bm u_t\|_{\sigX[]} < 2\,|\bm u_t^{\emptyword}|\,$, which proves the first point. Moreover, in that case, we obtain
    \begin{equation}
        \label{eq:norm_psi_log}
    \| \bm \psi_t\|_{\sigX[]} \leq |\bm \psi^{\emptyword}_t| - \log\left(1 - \frac{\|\bm u_t\|_{\sigX[]} -  |\bm u_t^{\emptyword}|}{|\bm u_t^{\emptyword}|}\right) = |\bm \psi^{\emptyword}_t| + \log \left(\frac{|\bm u_t^{\emptyword}|}{2|\bm u_t^{\emptyword}| -\|\bm u_t\|_{\sigX[]}} \right)\,. 
    \end{equation}

    We now move to the proof of the second point of the theorem. Namely, showing that the condition $\|\bm u_t\|_{\sigX[]} < 2\,|\bm u_t^{\emptyword}|$ is satisfied whenever $M_{\bm p}(\sigX[])$ is small enough. {This will directly imply \eqref{eq:logcharsigexpansion} when applied to $\mathbb X = \widehat{\mathbb W}_{0,t}$, recall \eqref{eq:charchen}. The positivity of the stopping time \eqref{eq:tau_def_log_fourier} follows by continuity and the fact that  $\widehat{\mathbb W}_0 = \emptyword$ so that $|M_{\bp}(\emptyword)| = 0 < r$.}  
    
    Using the bounds \eqref{eq:enormous1} yields
    \begin{equation}\label{eq:log_laplace_long_bound}
    \begin{aligned}
        \|\bm u_t\|_{\sigX[]} &= |\bm u_t^{\emptyword}| + \|\overline{\bm u}_t\|_{\sigX[]} \\
        &= |\bm u_t^{\emptyword}| + \left\|\E\left[\overline{\bxi}_t\right]\right\|_{\sigX[]} \\
        &\leq |\bm u_t^{\emptyword}| + \E\left[\left\| \overline{\bm \xi}_t\right\|_{\sigX[]}\right] \\
        &= |\bm u_t^{\emptyword}| + \E\left[\left\|\bm \xi_t\right\|_{\sigX[]} - |\bm \xi^{\emptyword}_t|\right] \\
        &\leq |\bm u_t^{\emptyword}| + \E\left[ \exp \left(\bracket{\Re(\bm p)}{\sigW[t,T]} \right)\,\left(\exp \left(M_{\bm p}(\sigX[])\zeta_{\bp, t, T}\right) - 1 \right)\right],
    \end{aligned}    
    \end{equation}
    for all $t \in [0, T]$ and $\sigX[] \in G$.
    Formally letting $M_{\bm p}(\sigX[]) \to 0$ in the latter leaves $|\bm u_t^{\emptyword}|$ on the right-hand side. Hence, we expect that $\|\bm u_t\|_{\sigX[]}$ can be made smaller than $2 |\bm u_t^{\emptyword}|$ by taking $M_{\bm p}(\sigX[])$ small enough.

    To make this precise, define the term on the right-hand side of \eqref{eq:log_laplace_long_bound} by
    $$
    h(t,x) := \E\left[\exp \left( \bracket{\Re(\bm p)}{\sigW[t,T]}\right) \,\left(\exp\left(x\,\zeta_{\bp, t, T}\right) - 1\right) \right], \quad (t,x) \in [0,T] \times \R_+\,. 
    $$
    Fix $R > 0$. By the dominated convergence theorem, $h$ is continuous on $(t,x) \in [0,T] \times [0, R]$ (hence uniformly continuous), since the quantity inside the expectation is continuous and bounded by $\exp(C_{\bm p, T}(1 + R)^{\deg_{\word 1}(\bp)})$ on $[0,T] \times [0, R]$, by virtue of the second inequality in \eqref{eq:enormous1}. Moreover, $h(t,0) = 0$ for all $t \in [0\,,T]\,$. {By continuity of $\bu^{\emptyword}\,$, we have $\min_{s \in [0,T]}|\bu^{\emptyword}_s| > 0\,$. Hence, for any $\varepsilon \in (0\,,1)\,$, {uniform continuity of $h$ on $[0,T] \times [0, R]$} implies the existence of $r > 0$ such that 
    $$
    |h(t, x)| = |h(t, x) - h(t, 0)| \leq (1 - \varepsilon)\min_{s \in [0, T]}|\bu_s^{\emptyword}| \leq (1 - \varepsilon)|\bu_t^{\emptyword}| \,, \quad (t, x) \in [0, T]\times[0, r].
    $$
    Substituting this inequality in \eqref{eq:log_laplace_long_bound} yields
    \begin{equation}\label{eq:eps_cond_for_u}
        \|\bm u_t\|_{\sigX[]} \leq (2- \varepsilon)|\bu_t^{\emptyword}|\,, \quad t \in [0, T], \quad  \sigX[] \in G_{\bp, r}\,.
    \end{equation}
    Combining \eqref{eq:eps_cond_for_u} with \eqref{eq:norm_psi_log}, we obtain

    $$
    \| \bm \psi_t\|_{\sigX[]} \leq |\bm \psi_t^{\emptyword}| + \log \left({\varepsilon}^{-1} \right) \,, \quad t \in [0, T]\,, \quad \sigX[] \in G_{\bp, r}\,,
    $$
    and by continuity of $\bm \psi^{\emptyword}\,$,
    $$
    \sup_{t \in [0, T]}\,\|\bm \psi_t\|_{\sigX[]} \leq \sup_{t \in [0, T]} |\bm \psi_t^{\emptyword}| + \log \left({\varepsilon}^{-1} \right) < + \infty \,, \quad \sigX[] \in G_{\bm p, r}\,.
    $$}
\end{proof}

\begin{sqremark}
    {The convergence radius bound for} $\bracket{\bm \psi_t}{\sigX[]}$ depends on $\sigX[]$ only through $M_{\bm p}(\sigX[]) = \max_{1 \leq |\word v| \leq \deg(\bm p)} \, |\sigX[]^{\word v}|\,$, which involves only finitely many components of $\sigX[]\,$. This allows us to significantly simplify the further analysis: for instance, this immediately guarantees that the stopping time $\tau$ defined by \eqref{eq:tau_def_log_fourier} is almost surely positive.
\end{sqremark}

\begin{sqremark}
Since $\widehat{\mathbb W}_0 = \emptyword$ and $|M_{\bp}(\emptyword)| = 0 < r$,  we get that  \eqref{eq:logcharsigexpansion} holds for   $t = 0$, that is 
 $$ \E\left[\exp \left(\bracket{\bm p}{\sigW[T]}\right) \right] = \exp\left(\bracket{\bpsi_0}{\sigW[0]} \right)= \exp\left(\bpsi_0^{\emptyword} \right). $$
\end{sqremark}

\subsection{Absence of global expansion of the log-Fourier--Laplace transform: a counterexample}\label{sect:global_exp_counterexample}

Theorem~\ref{theorem:log_analyticity_general} establishes that $\bm \psi_t \in \mathcal{A}_{\bm p, r}$ with uniform radius $r > 0\,$. A natural question is whether this local result can be improved to a global one, i.e. whether $\bm \psi_t \in \mathcal A_{\infty}\,$. The answer is negative in general, and this is not a limitation of our approach but a fundamental property of nonlinear PDE solutions, as we highlighted in Subsection~\ref{subsec:power_series_psi}. 

In this subsection, we provide a simple counterexample: we prove that for $p(x) = -x^4$ (corresponding to $\bm p = - 24 \cdot \word{1111} \in \mathcal B$), the function $u(t, \cdot) = \E[\exp(p(\cdot + W_{t,T}))]$ has infinitely many zeros on the imaginary axis, so its logarithm cannot be entire. The proof uses the method of steepest descent to find the precise oscillatory asymptotics of $u(t, ix)$ as $x \to + \infty\,$. Due to the quartic form of the exponent, three complex saddle points are identified, two of which contribute to the asymptotic expansion, producing a cosine factor with explicitly computable argument. The local expansion {given by Theorem~\ref{theorem:log_analyticity_general}} is therefore sharp in the sense that the finite radius condition cannot be removed. 

Before presenting the asymptotic result of Theorem~\ref{theorem:zeros_quartic_exp_main}, we observe that the map 
$$
z \mapsto u(t, z) = \E[\exp(-(z + W_{t, T})^4)],
$$ 
is real-valued on the imaginary axis:
\begin{equation}
    u(t, ix) = \E\left[\exp({-x^4 - W_{t, T}^4 + 6x^2W_{t, T}^2})\,\cos(4x^3W_{t, T} - 4xW_{t, T}^3)\right] \in \R\,.
\end{equation}
The following theorem shows that this real function has an oscillatory behavior as $x \to \infty\,$, thus admitting an infinite number of zeros.
\begin{theorem}
    \label{theorem:zeros_quartic_exp_main}
    Let $X \sim \mathcal N(0, \sigma^2)\,$, with $\sigma > 0\,$. Then,
    $$
    I(x) := \mathbb E \left[\exp(-(ix + X)^4) \right] = \frac{4^{1/3}}{\sqrt{3}\sigma^{1/3}x^{1/3}}\,\exp({\Psi(x)})\,(\cos(\Phi(x)) + o(1))
    $$
    in the limit $x \to + \infty\,$, where 
    $$
    \Psi(x) := \frac{1}{2\sigma^2}x^2-\frac{3}{2^{11/3}\sigma^{8/3}}x^{4/3} - \frac{1}{4^{5/3}\sigma^{10/3}}x^{2/3} + \frac{1}{24 \sigma^4}
    $$
    and
    $$
    \Phi(x) := -\frac{3\sqrt{3}}{2^{11/3}\sigma^{8/3}}x^{4/3} + \frac{\sqrt{3}}{4^{5/3}\sigma^{10/3}}x^{2/3} + \frac{\pi}{6} \,.
    $$
\end{theorem}
\begin{proof}
    The proof is given in Appendix \ref{app:existence_zeros}.
\end{proof}
In Section \ref{section:recentering}, we will show how to recover a global representation of the log-Fourier--Laplace transform via a recentering method.

\section{Linear and Riccati equations on the extended tensor algebra}\label{sect:heat_and_riccati_sig}
In this section, we show that the coefficients $\bm u$ and $\bm \psi$ constructed in Theorems \ref{theorem:analyticity_general} and \ref{theorem:log_analyticity_general} satisfy differential equations on the tensor algebra. That is, infinite systems of ordinary differential equations indexed by words. To our knowledge, the results below constitute the first existence and uniqueness results for such equations in the fully path-dependent setting.

The section proceeds in four steps. We first show in Section \ref{subsec:existence_heat_eq} that $\bm u$ satisfies a linear equation on the tensor algebra, which we derive through its probabilistic representation. We then consider, in Subsection~\ref{sect:explicit_sol_heat}, the situation in which the solution to such a linear equation can be found in closed form and show that, in general, this explicit representation no longer holds in our setting. Uniqueness within an appropriate class of solutions is then established in Section \ref{subsec:uniqueness_heat}. {We then pass to the logarithm in Section \ref{subsec:existence_Riccati}, analogous to the Cole--Hopf transform, exploiting the linear equation satisfied by $\bm u$ to show that $\bm\psi$ solves a nonlinear Riccati equation on the tensor algebra.} Finally, we address in Section \ref{section:recentering} the intrinsic locality of the expansion of the log-Fourier--Laplace transform, which stems from the structural finiteness of its radius of convergence as established in Section \ref{sect:global_exp_counterexample}. By recentering the expansion at the current running signature using Chen's identity, we recover a global representation of the log-Fourier--Laplace transform, at the cost of solving a Riccati equation with random terminal condition at each time $t$.

\subsection{Linear equation: existence}
\label{subsec:existence_heat_eq}
Formally applying It\^o's formula (see Theorem \ref{thm:ito_formulas}) to the linear functional $\bracket{\bm u_t}{\sigW[t]}$, i.e.~the right-hand side of \eqref{eq:charsigexpansion},  yields
$$
d\bracket{\bm u_t}{\sigW[t]} = \bracket{\dot{\bm u}_t + \mathscr L \bm u_t}{\sigW[t]}\,dt + \bracket{\bm u_t |_{\word 1}}{\sigW[t]}\,dW_t\,,
$$
where the linear operator $\mathscr L$ is given by  
$$
\mathscr L \bm q := \bm q|_{\word 0} + \frac{1}{2}\,\bm q|_{\word{11}}\,, \quad \bm q \in \eTAC[2]\,.
$$
Using the fact that $\bracket{\bm u_t}{\sigW[t]} = \E[\exp (\bracket{\bm p}{\sigW[T]})\,\vert\,\mathcal F_t]$ is a martingale, we formally obtain the infinite system of ODEs 
\begin{equation}\label{eq:heat_eq_coef_sig}
\begin{aligned}
    \dot{\bm u}_t + \mathscr{L} \bm u_t = 0\,, \quad t \in [0, T]\,,
\end{aligned}
\end{equation}
with $\bm u_T = \shuexp{\bm p}$. As discussed in Section \ref{sect:expansion_u_power_ser}, this is the natural non-Markovian analogue of expanding the solution of the classical heat equation into a power series and deriving the dynamics of the corresponding coefficients.


We stress that proving existence of a solution to this differential equation directly, by fixed point or truncation arguments, seems to be very intricate in the infinite-dimensional setting. Moreover, it admits an infinite number of solutions, and only one corresponds to the probabilistic representation we are interested in. Our approach circumvents this by constructing $\bm u_t$ as $\E[\bm \xi_t]$ and then verifying that it satisfies the equation in Theorem \ref{theorem:existence_heat_eq} below. The class of terminal conditions $\mathcal B$ and the estimates from Lemma \ref{lemma:enormous_bound} are what make this indirect approach work. 

We will use the following estimates in order to control $\|\mathscr L \bm u_t\|_{\sigX[]}\,$.

\begin{lemma}
    \label{lem:enormous_bound_shift}
    Let $\bm p \in \mathcal B\,$, and $\bxi$ be as in \eqref{eq:def_xi}. Then, for any $\word v \in V\,$, there exists a constant $C_{\bp, \word v} \geq 0$ such that
        $$
        \|\bxi_t|_{\word v}\|_{\sigX[]} \leq C_{\bp, \word v} (1 + M_{\bp}(\sigW[t,T]))^{|\word v|}(1 + M_{\bp}(\sigX[]))^{|\wv|}\,\exp\left(C_{\bm p, T} \,(1 + M_{\bm p}(\sigX[]))^{\deg_{\word 1}(\bm p)}\right), \quad t \in [0, T], \quad \sigX[] \in G,
        $$
        where $M_{\bp}$ is defined in \eqref{eq:def_M_p} and $C_{\bp , T}$ is as in Lemma \ref{lemma:enormous_bound}.
\end{lemma}
\begin{proof}
As discussed in Section \ref{subsec:ito}, shifts act as differentiation operators. {More precisely, for a given letter $\word i \in \{\word 0, \word 1\}$, the shift $\proj{i}$ is a derivation:}
$$
(\bell \shuprod\bell ')|_{\word i} = \bell|_{\word i} \shuprod \bell' + \bell \shuprod \bell'|_{\word i}\,, \quad \bell\,, \bell' \in \eTAC[2]\,.
$$
For a general word $\word v \in V\,$, applying this recursively to the shuffle exponential $\shuexp{\bell} = \sum_{n \geq 0} \bell\shupow{n}/n!$ yields
    $$
    \shuexp{\bell}|_{\word v} = \shuexp{\bell}\shuprod { P\shupow{}_{\word v}(\bell|_{\word{w}},\ {\word w \subset \word v})\,,}
    $$
    where $P\shupow{}_{\word v}$ is a shuffle polynomial {with nonnegative coefficients} of degree $|\word v|$ that depends on all $\bell|_{\word w}$ {such that $\word{w}$ is a subsequence of letters of $\word v\,$, denoted here by $\word{w} \subset \word{v}$}. By the shuffle-compatibility \eqref{eq:def_shuffle_compatibility_norm_X} of the seminorm $\|\cdot\|_{\sigX[]}\,$,
    $$
    \|\shuexp{\bell}|_{\word v}\|_{\sigX[]} \leq \|\shuexp{\bell}\|_{\sigX[]}\,P_{\word v}(\|\bell|_{\word w}\|_{\sigX[]}, \ {\word w \subset \word v})\,, \quad \sigX[] \in G\,,
    $$
    {where $P_{\word v}$ is a polynomial obtained from $P\shupow{}_{\word v}$ by replacing the shuffle product with the ordinary product of real numbers.}
    Applying this to $\bxi_t = \shuexp{\bm p|_{\sigW[t,T]}}\,$, we get
    \begin{equation}
    \label{eq:bound_norm_xi_proj_proof}
    \|\bxi_t|_{\wv}\|_{\sigX[]} \leq \|\bxi_t\|_{\sigX[]}\,P_{\word v}(\|(\bp|_{\sigW[t,T]})|_{\word w}\|_{\sigX[]}, \ {\word w \subset \word v})\,, \quad t \in [0, T]\,, \quad \sigX[] \in G\,.
    \end{equation}
    It remains to bound 
    \begin{align*}
        \|(\bp|_{\sigW[t,T]})|_{\word w}\|_{\sigX[]} &= \sum_{n \geq 0}\left|\sum_{|\word u| = n} (\bp|_{\sigW[t,T]})^{\word u\word w}\sigX[]^{\word u} \right| \\
        &\leq \sum_{|\word u| = 0}^{\deg(\bp)} \sum_{|\word{u'}| = 0}^{\deg(\bp)} \left|\bp^{\word{uwu'}} \sigW[t,T]^{\word{u'}} \sigX[]^{\word u} \right| \\
        &\leq (1 + M_{\bm p}(\sigX[]))(1 + M_{\bm p}(\sigW[t,T]))(1 +\deg(\bp))^2 \max_{\word u}|\bp^{\word u}|,
    \end{align*}
    for all $t \in [0, T],\ \sigX[] \in G$, and $\word w \in V$.
    Plugging this into \eqref{eq:bound_norm_xi_proj_proof}, we can find a constant $C_{\bm p, \word v} \geq 0$ such that 
    $$
    \|\bxi_t|_{\word v}\|_{\sigX[]} \leq C_{\bm p, \wv}\|\bxi_t\|_{\sigX[]} (1 + M_{\bp }(\sigX[]))^{|\word v|}(1 + M_{\bp}(\sigW[t,T]))^{|\word v|}\,, \quad t \in [0, T], \quad \sigX[] \in G\,.
    $$
    Finally, $\|\bxi_t\|_{\sigX[]}$ can be bounded using \eqref{eq:enormous1}, yielding the result.
\end{proof}
We are now ready to prove that $\bm u$ satisfies the linear differential equation on the tensor algebra.
\begin{theorem}
\label{theorem:existence_heat_eq}
Let $\bm p \in \mathcal B\,$, and $\bm u$ be as in Theorem \ref{theorem:analyticity_general}. Then,
\begin{enumerate}
    \item For all $\word v \in V\,$, the function $t \mapsto \bm u_t^{\word v}$ is $C^1$ on $[0, T]\,$. 
    \item The following linear equation is satisfied:
    \begin{equation}
    \label{eq:heat_eq_from_thm}
    \begin{cases}
         \dot \bu_t + \mathscr{L}\bu_t= 0\,, \quad t \in [0, T], \\
    \bm u_T = \shuexp{\bm p}.
    \end{cases}
    \end{equation}
    \item For any $\word v \in V\,$, there exists a constant $C_{\bm p, T, \word v} \geq 0$ such that {for all $\sigX[] \in G$}
    \begin{equation}
    \label{eq:bound_norm_shit_u}
    \sup_{t \in [0, T]}\,\|\bm u_t|_{\word v}\|_{\sigX[]} \leq C_{\bm p, T, \word v}\,(1 + M_{\bm p}(\sigX[]))^{|\word v|}\,  \exp \left(C_{\bm p, T}\, (1 + M_{\bm p}(\sigX[]))^{\deg_{\word 1}(\bm p)} \right) < + \infty\,.
    \end{equation}
\end{enumerate}
\end{theorem}


\begin{proof}
The bound \eqref{eq:bound_norm_shit_u} is easily obtained from Lemma \ref{lem:enormous_bound_shift}: for any word $\word v \in V\,$, recalling that $\bm u = \E[\bxi]\,$, we have 
\begin{align*}
    \sup_{t \in [0, T]}\|\bu_t|_{\word v}\|_{\sigX[]} &\leq \sup_{t \in [0, T]} \E[\|\bxi_t|_{\word v}\|_{\sigX[]}] \\
    &\leq C_{\bm p, \wv} (1 + M_{\bp}(\sigX[]))^{|\wv|} \exp\left(C_{\bp, T} (1 + M_{\bp}(\sigX[]))^{\deg(\bp)}\right) \, \sup_{t \in [0, T]}\E\left[(1 + M_{\bp}(\sigW[t,T]))^{|\wv|} \right] \\
    &\leq C_{\bp, T, \wv}(1 + M_{\bp}(\sigX[]))^{|\wv|} \exp\left(C_{\bp, T} (1 + M_{\bp}(\sigX[]))^{\deg(\bp)}\right)\,, \quad \sigX[] \in G\,,
\end{align*}
where we used the fact that $\sup_{t \in [0, T]}\E\left[(1 + M_{\bp}(\sigW[t,T]))^{r} \right]$ is finite for any $r \geq 1\,$.

We now move to proving the first two points of the theorem. Set $\widecheck{W}_t = -\widehat{W}_t$, and let $\widecheck{\mathbb{W}}$ denote its signature. The process $\widecheck{W}$ has the same law as the time-reversal of $\widehat{W}$, since time-reversed Brownian increments are equal in law to negated increments. Therefore, using the fact that the signature of a time-reversed path is the inverse of the original signature, we obtain
\begin{equation}\label{eq:hat_upside_down}
    \widecheck{\mathbb{W}}_t^{-1} \overset{d}{=} \widehat{\mathbb{W}}_t\,, \quad t \geq 0\,.
\end{equation}
Let us fix $t \in [0\,, T]$ and $\sigX[] \in G\,$. Applying Itô's formula (Theorem~\ref{thm:ito_formulas}) to $s\mapsto\bracket{~_{\sigX[]}|\bp}{\widecheck{\mathbb{W}}_s^{-1}}$ between $0$ and $T-t$ yields
$$
\bracket{~_{\sigX[]}|\bp}{\widecheck{\mathbb{W}}_{T-t}^{-1}} = \bracket{~_{\sigX[]}|\bp}{\emptyword} + \int_0^{T-t} \bracket{~_{\word{0}}|(~_{\sigX[]}|\bp) + \dfrac12~_{\word{11}}|(~_{\sigX[]}|\bp)}{\widecheck{\mathbb{W}}_{s}^{-1}}\,ds + \int_0^{T-t} \bracket{~_{\word{1}}|(~_{\sigX[]}|\bp)}{\widecheck{\mathbb{W}}_{s}^{-1}}\,dW_s.
$$
Using the shift properties given by Proposition~\ref{prop:projections_ppties}, this can be rewritten as
$$
\bracket{\bp}{\mathbb{X}\otimes\widecheck{\mathbb{W}}_{T-t}^{-1}} = \bracket{\bp}{\mathbb{X}} + \int_0^{T-t}\bracket{\bp}{\mathbb{X}\otimes\left(\word{0}+\dfrac12\word{11}\right)\otimes\widecheck{\mathbb{W}}_{s}^{-1}}\,ds + \int_0^{T-t}\bracket{\bp}{\mathbb{X}\otimes\word{1}\otimes\widecheck{\mathbb{W}}_{s}^{-1}}\,dW_s.
$$
Applying Itô's formula  to $\exp\left(\bracket{\bp}{\mathbb{X}\otimes\widecheck{\mathbb{W}}_{T-t}^{-1}}\right)$, we obtain
\begin{align}
    \exp({\bracket{\bp}{\sigX[] \otimes \widecheck{\mathbb{W}}_{T-t}^{-1}}}) &= \exp({\bracket{\bp}{\sigX[]}}) \\
    &+ \int_0^{T-t} \exp({\bracket{\bp}{\sigX[] \otimes \widecheck{\mathbb{W}}_{s}^{-1}}})\left(\bracket{\bp}{\sigX[] \otimes \left(\word{0} + \dfrac{1}{2}\word{11}\right)\otimes\widecheck{\mathbb{W}}_{s}^{-1}} + \dfrac{1}{2}\bracket{\bp}{\sigX[] \otimes \word{1}\otimes\widecheck{\mathbb{W}}_{s}^{-1}}^2\right)\,ds \\
    &+ \int_0^{T-t} \exp({\bracket{\bp}{\sigX[] \otimes \widecheck{\mathbb{W}}_{s}^{-1}}})\bracket{\bp}{\sigX[] \otimes \word{1}\otimes\widecheck{\mathbb{W}}_{s}^{-1}}\,dW_s.
\end{align}
{The stochastic integral is a true martingale since
$
\exp({\bracket{\bp}{\sigX[] \otimes \widecheck{\mathbb{W}}_{s}^{-1}}}) = \exp({\bracket{~_{\sigX[] }|\bp}{ \widecheck{\mathbb{W}}_{s}^{-1}}}),
$
is bounded by Proposition~\ref{prop:boundedness_dot_product_B} (recall that $~_{\sigX[] }|\bp \in \mathcal{B}$ by Lemma~\ref{lemma:recentering_property_class_B}), while $\bracket{\bp}{\sigX[] \otimes \word{1}\otimes\widecheck{\mathbb{W}}_{s}^{-1}}$ is a finite linear combination of the elements of $\widecheck{\mathbb{W}}_{s}^{-1}$.} Thus, taking expectations kills the martingale term, so that using \eqref{eq:hat_upside_down}, the change of variables $u = T - s$, and the fact that $\sigW[T-u] \overset{d}{=} \sigW[u, T]$, we get
\begin{align}
     \E&
\left[\exp({\bracket{\bp}{\sigX[] \otimes \widehat{\mathbb{W}}_{t, T}}})\right] = \exp({\bracket{\bp}{\sigX[]}}) \\
     &+ \int_t^{T} \E\left[\exp({\bracket{\bp}{\sigX[] \otimes \widehat{\mathbb{W}}_{u, T}}})\left(\bracket{\bp}{\sigX[] \otimes \left(\word{0} + \dfrac{1}{2}\word{11}\right)\otimes\widehat{\mathbb{W}}_{u, T}} + \dfrac12\bracket{\bp}{\sigX[] \otimes \word{1}\otimes\widehat{\mathbb{W}}_{u, T}}^2\right)\right]\,du.
\end{align}
The left-hand side is equal to $\bracket{\bu_t}{\sigX[]}$ by Theorem \ref{theorem:analyticity_general}, whereas the right-hand side can be rewritten using Proposition~\ref{prop:projections_ppties} as follows
\begin{equation}
\label{eq:heat_eq_before_exchange}
\bracket{\bu_t}{\sigX[]} = \exp({\bracket{\bp}{\sigX[]}}) +\int_t^T \E \left[\exp({\bracket{\bp|_{\sigW[s, T]}}{\sigX[]}})\bracket{\mathscr{R}(\bp|_{\sigW[s, T]})}{\sigX[]} \right]\,ds\,, \quad t \in [0, T]\,, \quad \sigX[] \in G\,,
\end{equation}
where the nonlinear operator $\mathscr R$ is defined as $\mathscr R \bell := \mathscr L \bell + \frac{1}{2}(\bell|_{\word 1})\shupow{2}\,$.
Moreover, since $\shuexp{\bell}|_{\word i} = \shuexp{\bell} \shuprod \bell|_{\word i}$ for $\word i \in \{\word 0, \word 1\}\,$, we get that $\mathscr L \bxi_t = \mathscr L \shuexp{\bp|_{\sigW[t,T]}} = \shuexp{\bp|_{\sigW[t,T]}} \shuprod \mathscr R (\bp|_{\sigW[t,T]})\,$, and \eqref{eq:heat_eq_before_exchange} becomes
\begin{equation}\label{eq:almost_heat_eq}
  \bracket{\bu_t}{\sigX[]} = \exp(\bracket{\bp}{ \sigX[]}) + \int_t^T \E \left[\bracket{\mathscr L \bxi_s}{\sigX[]} \right] ds = \bracket{\shuexp{\bp}}{ \sigX[]} + \int_t^T \bracket{\mathscr L \bu_s}{\sigX[]} ds\,, \quad t \in [0, T]\,, \quad \sigX[] \in G\,,
\end{equation}
where the last equality follows from $\E[\|\mathscr L \bxi_t\|_{\sigX[]}]$ being finite, thanks to Lemma \ref{lem:enormous_bound_shift} and to the fact that $\E[( 1 + M_{\bp}(\sigW[t,T]))^r]$ is finite for any $r \geq 1\,$. Finally, \eqref{eq:bound_norm_shit_u} shows that $\sup_{t \in [0, T]}\|\mathscr L \bu_t\|_{\sigX[]} < + \infty\,$, which allows us to perform the interchange $\int_t^T \bracket{\mathscr L \bm u_s}{\sigX[]}\, ds = \bracket{\int_t^T  \mathscr L \bm u_s\, ds}{\sigX[]}$ in \eqref{eq:almost_heat_eq}. Combining this with Lemma \ref{lem:uniq_coef}, we obtain
$$
\bm u_t = \shuexp{\bm p} + \int_t^T \mathscr L \bu_s ds\,, \quad t \in [0, T]\,,
$$
thus proving the first two points of Theorem \ref{theorem:existence_heat_eq}.
\end{proof}

\begin{sqremark}\label{rmk:markovian_heat}
    The Markovian framework considered in Section~\ref{section:heuristices} and
    in \cite*{abijaber2024fourier}, where
    $$
    P((W_s)_{s \in [0, T]}) = q(W_T) + \int_0^T r(W_s)\,ds,
    $$
    and $q$ and $r$ are two polynomials of even degree with the leading coefficients
    having negative real parts, can be recovered when taking
    $$
    \bp = \bq + \br\word{0}, \quad \deg_{\word{0}}(\bq) = \deg_{\word{0}}(\br) = 0.
    $$
    In this case, observing that
    $$
    \E\left[\exp\left(q(W_T) + \int_t^Tr(W_s)\,ds\right)\,\Bigg|\, \F_t\right] = e^{-\int_0^tr(W_s)\,ds}\bracket{\bu_t}{\sigW} = \bracket{\bu_t \shuprod \shuexp{-\br\word{0}}}{\sigW} =: \bracket{\bm v_t}{\sigW},
    $$
    implies that the solution
    to~\eqref{eq:heat_eq_from_thm} has the form
    $\bu_t = \bm{v}_t \shuprod \shuexp{\br\word{0}}$,
    where $\bm{v}_t$ satisfies
    \begin{equation*}
    \begin{cases}
        \dot{\bm{v}}_t + \mathscr{L}\bm{v}_t + \bm{v}_t \shuprod \br = 0,
        \quad t \in [0, T], \\
        \bm{v}_T = \shuexp{\bq}.
    \end{cases}
    \end{equation*}
    Notice that no term in the equation above contains the letter $\word{0}$. Hence, we
    have $\mathscr{L}\bm{v}_t = \frac{1}{2} \bm{v}_t\proj{11}$, and the solution is
    given simply by
    $$
    \bm{v}_t = \sum_{n \geq 0}\bm{v}_t^n \word{1}^{\otimes n}, \quad
    \bracket{\bm{v}_t}{\sigW} = \sum_{n \geq 0}\bm{v}_t^n\frac{W_t^n}{n!}.
    $$
    Moreover, in the presence of only the letter $\word{1}$, the shuffle product reduces
    to the Cauchy product defined by~\eqref{eq:cauchy_product_def}, so that the linear
    equation takes the form
    \begin{equation*}
    \begin{cases}
        \dot{\bm{v}}_t + \frac{1}{2}\bm{v}_t\proj{11} + \bm{v}_t * \br = 0,
        \quad t \in [0, T], \\
        \bm{v}_T = \exp^*\left(\bq\right).
    \end{cases}
    \end{equation*}
    In particular, when $\br = 0$, we recover equation~\eqref{eq:heat_eq_coef_power_ser}
    introduced in Section~\ref{sect:expansion_u_power_ser}. Returning to the initial
    problem, we obtain the following expansion of the Fourier--Laplace transform:
    \begin{align*}
        \E\!\left[\exp\!\left(q(W_T) + \int_0^T r(W_s)\,ds\right) \,\bigg|\, \F_t\right]
        = \exp\!\left(\int_0^t r(W_s)\,ds\right)\!\left(\sum_{n \geq 0}\bm{v}_t^n
        \frac{W_t^n}{n!}\right).
    \end{align*}
\end{sqremark}

\subsection{Complications with the closed form solution of the linear equation}\label{sect:explicit_sol_heat}

In this section, we show that the solution to the linear ODE \eqref{eq:heat_eq_coef_sig} can be found in closed form when the terminal condition has a finite number of non-zero terms,  i.e.~$\bsigma \in T(\C^2)$.

\begin{proposition}\label{prop:explicit_sol}
    Fix $\bsigma \in T(\C^2)$. A solution of \eqref{eq:heat_eq_coef_sig} with the terminal condition $\bu_T = \bsigma$ is given by
    \begin{equation}\label{eq:heat_eq_explicit_sol}
        \bu_t = \bsigma|_{\widehat{\mathcal{E}}_{T-t}} = e^{(T-t)\mathscr{L}}\bsigma, \quad t \in [0, T],
    \end{equation}
    where $e^{(T-t)\mathscr{L}}$ denotes the operator exponential and $\widehat{\mathcal{E}}_{T-t} = \E[\sigW[t, T]] = e^{\otimes(T-t)(\word{0} + \frac{1}{2}\word{11})}$ denotes the expected time-augmented Brownian signature.
\end{proposition}

\begin{proof}
Consider the martingale
$$
M_t = \E\left[\bracketsigW[T]{\bsigma} \Big|\, \mathcal{F}_t\right], \quad t \in [0, T].
$$
The expression on the right-hand side can be computed explicitly:
\begin{equation}\label{eq:sum_swap_1}
    M_t = \E\left[\bracketsigW[T]{\bsigma} \Big|\, \mathcal{F}_t\right] = \bracket{\bsigma}{\E[\sigW[T] |\, \mathcal{F}_t]} = \bracket{\bsigma}{\E[\sigW[t]\otimes\sigW[t,T] |\, \mathcal{F}_t]} = \bracket{\bsigma}{\sigW\otimes\widehat{\mathcal{E}}_{T-t}},
\end{equation}
where we used Chen's identity and the fact that $\sigW$ is $\F_t$-measurable. By the first part of Proposition~\ref{prop:projections_ppties}, we have
\begin{equation}\label{eq:sum_swap_2}
    M_t =\bracket{\bsigma}{\sigW\otimes\widehat{\mathcal{E}}_{T-t}} = \bracket{\bsigma|_{\widehat{\mathcal{E}}_{T-t}}}{\sigW} = \bracket{\bu_t}{\sigW},
\end{equation}
where we set $\bu_t = \bsigma|_{\widehat{\mathcal{E}}_{T-t}}$. 

To prove the second part of \eqref{eq:heat_eq_explicit_sol}, we note that $\bsigma|_{\word{0} + \frac12\word{11}} = \bsigma\proj{0} + \dfrac{1}{2}\bsigma\proj{11} = \mathscr{L}\bsigma$, and hence, for $k \geq 0$,
\begin{equation}\label{eq:proj_k_times}
    \bsigma|_{(\word{0} + \frac12\word{11})^{\otimes k}} = \mathscr{L}^k\bsigma,
\end{equation}
so that 
$$
\bu_t = \bsigma|_{\widehat{\mathcal{E}}_{T-t}} = \bsigma|_{e^{\otimes(T-t)(\word{0} + \frac12\word{11})}} = e^{(T-t)\mathscr{L}}\bsigma,
$$
where we used the linearity of the shift and \eqref{eq:proj_k_times}. {Note that the exponential $\bu_t = e^{(T-t)\mathscr{L}}\bsigma$ contains only a finite number of nonzero terms $\frac{1}{k!}(T-t)^k\mathscr{L}^k\bsigma$. Differentiating it implies that $\bu_t$ is a solution of \eqref{eq:heat_eq_coef_sig}.}
\end{proof}
\begin{sqremark}
   The first part $\bu_t = \bsigma|_{\widehat{\mathcal{E}}_{T-t}}$ of the representation \eqref{eq:heat_eq_explicit_sol} has been used  in \cite*[Theorem 5.2]{abi2024path} and \cite*[Lemma 4.4]{abijaber2025signature}.
\end{sqremark}

The reader can easily verify that the representation~\eqref{eq:heat_eq_explicit_sol} still holds when $\bsigma$ belongs to a larger class of exponentially dominated coefficients $\mathcal{A}_{\exp}$ (see \cite*[Subsection~3.4 and Section~5]{abi2024path}). Without giving the precise definition of this class, which is not further used in this paper, we mention only that the power series analogue of this class consists of the coefficients $\bsigma \in T(\C)$ such that 
$$
\sum_{n \geq 0}|\bsigma^{n}| |x|^n \leq C_1 \exp({C_2|x|}), \quad x \in \R,
$$
for some $C_1, C_2 > 0$, which guarantees that the expectation and the bracket can be interchanged in \eqref{eq:sum_swap_1} and that the summation order can be changed in \eqref{eq:sum_swap_2}.

However, in our case of interest, $\bsigma = \shuexp{\bp}$, the terminal condition does not belong to $T(\C^2)$ and, in general, is not exponentially dominated either (as soon as $\deg_{\word{1}}(\bp) \geq 2$). In this case, the equalities \eqref{eq:sum_swap_1} and \eqref{eq:sum_swap_2} no longer hold, as illustrated in the following example.

\begin{sqexample}
    Consider $\bsigma = \shuexp{(-24 \cdot \word{1111})}$, i.e., $\bracketsigW{\bsigma} = \exp({-W_t^4})$. The bracket and the expectation in \eqref{eq:sum_swap_1} cannot be interchanged; indeed, for $t = 0$,
    \begin{align*}
        \E\left[\exp({-W_T^4})\right] = \E\left[\sum_{k \geq 0}\dfrac{(-1)^kW_T^{4k}}{k!} \right] \neq \sum_{k \geq 0}\dfrac{(-1)^k\E\left[W_T^{4k}\right]}{k!} = \sum_{k \geq 0}\dfrac{(-1)^k T^{2k}(4k)!}{2^{2k}(2k)!k!}\,,
    \end{align*}
    since the latter series diverges.
    This is also the case for the coefficients given by \eqref{eq:heat_eq_explicit_sol}  (see Example~\ref{ex:eplicit_sol_pol}).
\end{sqexample}

This motivates us to characterize the coefficients $\bu_t$ as the solution to the linear differential equation \eqref{eq:heat_eq_coef_sig} and subsequently transform it into a numerically more tractable Riccati equation, rather than using the divergent closed-form solution \eqref{sect:explicit_sol_heat}.

\subsection{On the uniqueness for the linear equation}\label{subsec:uniqueness_heat}
Theorem \ref{theorem:existence_heat_eq} produces a solution $\bm u$ to \eqref{eq:heat_eq_from_thm}, but this equation admits infinitely many solutions. Uniqueness fails without further growth conditions, as in the classical \cite{tychonoff1935theoremes} example for the heat equation on $\R$ (see Remark \ref{remark:non_uniqueness_Tychonoff}). The following theorem identifies a class of solutions for which uniqueness holds. The first condition is a regularity condition, intrinsic to the operator $\mathscr L\,$, which ensures that It\^o's formula can be applied. The second condition controls the growth of the solution, and is calibrated to the Gaussian tail of $\sup_{s \leq t} |W_s|\,$. In the following, we say that $(\bm u_t)_{t \in [0, T]}$ is a solution to the linear equation \eqref{eq:heat_eq_coef_sig} if $\bm u_{\cdot}^{\wv}$ is continuously differentiable on $[0\,,T]\,$, and $ \dot{\bm u}_t + \mathscr L \bm u_t = 0$ for all $t \in [0\,, T]\,$.

\begin{definition}
    \label{def:admissible solution}
    A solution $(\bu_t)_{t \in [0, T]}$ to the linear equation \eqref{eq:heat_eq_coef_sig} is called It\^o-admissible if $\|\bm u_t\|_{\sigW[s]}$ is a.s. finite for all $(t,s) \in [0\,,T]^2\,$, and for any word $\wv \neq \emptyword\,$, there exists a locally bounded function $f_{\wv}\colon \R_+ \to \R_+$ such that 
    $$
    \sup_{t \in [0, T]} \left(\|\bm u_t|_{\word v} \|_{\sigW[s]}\right) \leq f_{\wv}\left(\sup_{0 \leq r \leq s} |W_r|\right)\,, \quad s \in [0\,,T]\,.
    $$
\end{definition}

The It\^o admissibility in Definition \ref{def:admissible solution}, along with the fact that $\dot{\bm u} + \bm u|_{\word 0} + \frac{1}{2} \bm u|_{\word{11}} = 0$, ensure that we can apply It\^o's formula \eqref{eq:right_ito} to $t \mapsto \bracket{\bm u_t|_{\wv}}{\sigW[t]}$ for any word $\word v \in V\,$. The solution $\bm u$ constructed in Theorem \ref{theorem:analyticity_general} is It\^o admissible thanks to \eqref{eq:bound_norm_shit_u} and Lemma \ref{lem:sig_bound}, which allows one to control $M_{\bp}(\sigW[s])$ by $\sup_{0 \leq r \leq s}|W_r|^{\deg_{\word 1}(\bp)}\,$.

The following theorem gives uniqueness of It\^o-admissible solutions under a sub-Gaussian growth condition of the associated linear functional.
\begin{theorem}
    \label{thm:uniqueness_heat_eq}
 Let $(\bu_t)_{t \in [0, T]}$ be an It\^o-admissible solution to \eqref{eq:heat_eq_coef_sig}, with terminal condition $\bu_T = 0\,$. Assume moreover that there exists $g : \R_+ \to \R_+$ with $g(x) = o(e^{c x^{2}})$ as $x \to \infty$ for some $c > 0\,$, such that 
    \begin{equation}\label{eq:sub_gaussian_growth}
        \sup_{t \in [0, T]}|\bracket{\bm u_t}{\sigW[s]}| \leq g\left(\sup_{0 \leq r \leq s}|W_r|\right)\,, \quad s \in [0\,,T]\,.
    \end{equation}
    Then $\bu \equiv 0$ on $[0\,,T]\,$.
\end{theorem}
\begin{proof}
Let $\ell \geq 2$ be an integer such that $T / \ell < 1 / (2c)\,$, where $c > 0$ is the constant in the decay condition on $g\,$, and define $h := T / \ell\,$. The proof proceeds by iteration over intervals of length $h\,$. On each interval, a localization argument combined with the decay condition on $g$ yields $\bm u_t \equiv 0\,$, and the result will follow by patching. We fix $t \in (T-h\,, T)$ and $s \in [0\,, t - T + h]\,$.
 The It\^o admissibility of $\bm u$ along with the fact that $\dot{\bu}_t + \mathscr{L}\bu_t = 0$ ensure that we can apply It\^o's lemma to 
    $$
    Z_r := \bracket{\bm u_{t - s + r}}{\sigW[r]}\,, \quad  r \in [0\,, T- t + s]\,,
    $$
    yielding
    \begin{align*}
    Z_r &= Z_0 + \int_0^r \bracket{ \dot{\bm u}_{t-s + u} + \mathscr L \bm u_{t - s + u}}{\sigW[ u]} \, du + \int_0^r \bracket{\bm u_{t-s+ u}|_{\word 1}}{\sigW[ u]} \, dW_u \\
    &= Z_0 + \int_0^r \bracket{\bm u_{t-s+ u}|_{\word 1}}{ \sigW[ u]} \, dW_u \,,\quad  r \in [0\,,T-t + s]\,,
    \end{align*}
    so that $Z$ is a local martingale. 
    {Observe that the process $Z$ satisfies:
    $$
    Z_s = \bracket{\bu_t}{\sigW[s]}, \quad Z_{T - t + s} = \bracket{\bu_T}{\sigW[T-t+s]} = 0,
    $$
    where the last equality follows from the terminal condition $\bu_T = 0$.}
    For any $n \geq 1\,$, the stopping time $\tau_n := \inf\{r \geq 0\,:\, |W_r| \geq n\}$ is almost surely finite. By continuity of the Brownian motion, we have $\tau_n \to +\infty$ almost surely as $n \to + \infty\,$. We also have the bound
    $$
    |Z_{r \wedge \tau_n}| \leq g\left(\sup_{0 \leq u \leq r \wedge \tau_n} |W_u|\right) \leq g(n)\,, \quad r \in [0\,,T-t + s]\,, \quad n \geq 1\,,
    $$
    so that $(\tau_n)_{n \geq 1}$ is a localizing sequence of stopping times for $Z\,$. We use this to bound $\bracket{\bm u_t}{\sigW[s]} = Z_s$ as follows
    \begin{align*}
        |Z_{s \wedge \tau_n}| &= |\E[Z_{(T-t + s) \wedge \tau_n}\,|\,\mathcal F_s]| \\
        &= |\E[Z_{\tau_n} \indic{\tau_n \leq T-t + s} \, |\, \mathcal F_s]|\\
        &\leq \E[|Z_{\tau_n}| \indic{\tau_n \leq T-t + s}\,|\,\mathcal F_s]\\
        &\leq g(n)\,\P(\tau_n \leq T-t+s \,|\,\mathcal F_s)\,, \quad n \geq 1.
    \end{align*}
     Taking expectations in the previous yields 
    \begin{equation}
    \label{eq:bound_exp_Z_proof_uniqueness}
    \E[|Z_{s \wedge \tau_n}|] \leq g(n)\, \P(\tau_n \leq T - t+ s)\,, \quad n \geq 1\,.
    \end{equation}
    It remains to quantify the behavior of $\P(\tau_n \leq T - t + s)$ as $n \to + \infty\,$. This can be done using classical Gaussian tail estimates:
    {\begin{align*}
        \P(\tau_n \leq T - t + s) = \P\left(\sup_{ r \in [0,T-t + s]}|W_r| \geq n\right) \leq 4\exp \left(-\frac{n^2}{2(T- t +s)} \right) \leq 4\exp(-cn^2),
    \end{align*}
    where in the last inequality we used the range of $t$ and $s$ ensuring that $0 < T - t + s \leq h <  (2c)^{-1}\,$.}
    Hence, applying Fatou's lemma in \eqref{eq:bound_exp_Z_proof_uniqueness} as $n$ goes to infinity, we obtain $\bracket{\bm u_t}{\sigW[s]} = 0\,\, a.s.\,$, by virtue of the growth assumption on the function $g\,$.

    By the same arguments as in the proof of \citet*[Lemma 3.9]{aqsha2026solving}, the fact that $\bracket{\bm u_t}{\sigW[s]} = 0$ for all $s \in [0\,, t - T + h]\,$, combined with It\^o admissibility of $\bm u|_{\word v}$ for any word $\wv \neq \emptyword\,$, implies $\bm u_t = 0\,$.
    Thus, we have proved that $\bm u_t = 0$ for all $t \in [T-h\,, T]\,$. We can repeat this process for $t \in [T - 2h\,, T -h]\,$, using the new terminal condition $\bm u_{T - h} = 0$ to verify that $\bm u$ vanishes on this interval as well. Iterating this until $[T - \ell h\,, T - (\ell -1)h] = [0\,, T / \ell]\,$, we finally obtain that $\bm u \equiv 0$ on $[0\,, T]\,$, which concludes the proof.
\end{proof}
We note that our growth condition on $\bracket{\bm u_t}{\cdot}$ reduces to the one from \cite{tychonoff1935theoremes} in the Markovian case $\bm u_t = \sum_{n \geq 0} \bm u^k_t \word{1}^{\otimes k}$ of Section \ref{sect:expansion_u_power_ser}. This condition is sharp as \cite{tychonoff1935theoremes} shows that for any $\varepsilon$, a non-zero solution satisfying $|\bracket{\bm u_T}{\mathbbm x}| < \exp(x^{2 + \varepsilon})$ and $\bm u_T = 0$ can be constructed.

Finally, we are able to prove that our solution belongs to this uniqueness class in the following Corollary.
\begin{corollary}\label{cor:uniqueness_heat}
    The solution $\bu$ constructed by probabilistic representation in Theorem \ref{theorem:existence_heat_eq}, with $\bm p \in \mathcal B\,$, is the unique It\^o-admissible solution to \eqref{eq:heat_eq_from_thm} satisfying the sub-Gaussian growth condition \eqref{eq:sub_gaussian_growth}.
\end{corollary}
\begin{proof} Let
$W$ be the Brownian motion that appears in the definition of $\bu$ in Theorem \ref{theorem:existence_heat_eq}. Since $\bu$ is Itô-admissible, it suffices to prove that it satisfies \eqref{eq:sub_gaussian_growth} for any Brownian motion $B$ independent of $W$, so that an application of Theorem~\ref{thm:uniqueness_heat_eq} to the difference between $\bu$ and any other solution satisfying the required conditions yields the result. To prove \eqref{eq:sub_gaussian_growth} for $\bu$, we
    let $B$ be a  Brownian motion independent of $W$ and $\widehat{\mathbb B}$ its time augmented signature. Then, taking $\bxi_t := \shuexp{\bp|_{\sigW[t,T]}}$ as in \eqref{eq:def_xi}, we have $\bracket{\bxi_t}{\widehat{\mathbb B}_{s}} = \exp(\bracket{\bp}{\widehat{\mathbb B}_{s} \otimes \sigW[t,T]})\,$. By Chen's identity \eqref{eq:Chen}, $\widehat{\mathbb B}_{s} \otimes \sigW[t,T]$ is the time-augmented signature of the concatenation of the paths $(B_r)_{r \in [0,s]}$ and $(W_r - W_t)_{r \in [t,T]}\,$, and thus it is the time-augmented signature of a continuous semimartingale. Hence, we can apply Proposition \ref{prop:boundedness_dot_product_B} which states that $|\bracket{\bxi_t}{\widehat{\mathbb B}_{s}}| = \exp(\bracket{\Re(\bp)}{\widehat{\mathbb B}_{s} \otimes \sigW[t,T]}) \leq \exp(C)$ almost surely, locally uniformly in $s$ and $t\,$. Therefore, this is also the case for $|\bracket{\bm u_t}{\widehat{\mathbb B}_{s}}| \leq \E_W[|\bracket{\bxi_t}{\widehat{\mathbb B}_{s}}|]\,$, which shows that $\bu$ satisfies \eqref{eq:sub_gaussian_growth}.
\end{proof}

\subsection{Riccati equation for the Brownian signature}
\label{subsec:existence_Riccati} 
Theorem 
\ref{theorem:existence_heat_eq} and the shuffle-logarithm $\bpsi = \log^{\shuprod}\bu$ together yield our second main result: the deterministic coefficients $(\bpsi_t)_{t \in [0, T]}$ constructed in Theorem \ref{theorem:log_analyticity_general} satisfy the \textit{signature Riccati equation}:
\begin{equation}\label{eq:Riccati_coefs}
    \dot{\bm \psi}_t + \mathscr{R} \bm \psi_t = 0\,, \quad t \in [0, T]\,,
\end{equation}
with terminal condition $\bm \psi_T = \bm p\,$, where the nonlinear operator $\mathscr R$ is given by $$\mathscr R \bell := \bell|_{\word 0} + \frac{1}{2} \bell|_{\word{11}} + \frac{1}{2}(\bm \bell|_{\word 1})\shupow{2}\,, \quad \bell \in \eTAC[2]\,.$$
This is the tensor algebra analogue of the classical Hamilton--Jacobi--Bellman equation
$$
\partial_t \psi + \frac{1}{2} \partial_{xx} \psi + \frac{1}{2} (\partial_x \psi)^2 = 0\,,
$$
obtained from the linear equation via the Cole--Hopf transform, and reduces to it exactly in the power series case of Section \ref{subsec:power_series_psi}. Crucially, \eqref{eq:Riccati_coefs} is the equation one would naturally seek to solve in order to represent the conditional log-Fourier--Laplace transform $\log \E[\exp(\bracket{\bp}{\sigW[T]})\,|\,\mathcal F_t] = \bracket{\bpsi_t}{\sigW[t]}\,$, and Theorem \ref{theorem:existence_riccati} below provides the first rigorous proof that this equation admits a solution for general $\bp \in \mathcal B\,$. This is obtained not by solving the equation directly, which appears intractable in the infinite-dimensional setting, but by deriving it from the linear equation via the algebraic Cole--Hopf transform at the level of the tensor algebra coefficients.

\begin{theorem}\label{theorem:existence_riccati}
    Let $\bm p \in \mathcal B\,$. Assuming that $\E[\exp(\bracket{\bm p}{\sigW[t,T]})] \neq 0$ for all $t \in [0,T]\,$, and letting $\bm \psi$, $\varepsilon$, and $r$ be as in Theorem \ref{theorem:log_analyticity_general}, we have
    \begin{enumerate}
        \item For any $\word v \in V\,$, the function $t \mapsto \bm \psi_t^{\word v}$ is $C^1$ on $[0\,, T]\,$.
        \item The following equation is satisfied:
    \begin{equation}\label{eq:Riccati_BM}
        \begin{cases}
            \dot{\bpsi}_t  + \mathscr{R}\bpsi_t = 0\,, \quad t \in [0, T]\\
            \bm \psi_T = \bm p
        \end{cases}
        \end{equation}
        \item {For any $\word v \neq \emptyword\,$, there exists a constant $C'_{\bp, T, \wv} > 0$ such that for all $\sigX[] \in G_{\bp, r}$,} 
        $$
        \sup_{t \in [0,T]}\|\bpsi_t|_{\word v}\|_{\sigX[]} \leq C'_{\bp, T, \wv} \left(\left(1 + \frac{1}{\varepsilon \inf_{t \in [0,T]} |\bu_t^{\emptyword}|}\right)(1 + M_{\bp}(\sigX[])) \exp\left(C_{\bp, T} (1 + M_{\bp}(\sigX[]))^{\deg_{\word 1}(\bp)} \right)\right)^{|\wv|}.
        $$
    \end{enumerate}
\end{theorem}
\begin{proof}


We recall that for $\word v \neq \emptyword\,$,
$$
\bm \psi_t^{\word v} =  \sum_{n \geq 1}\,\frac{(-1)^{n-1}}{n}\,\left[\left( \frac{\overline{\bm u}_t}{\bm u_t^{\emptyword}}\right)\shupow{n}\right]^{\word v}\,, \quad t \in [0, T]\,,
$$
where the sum has a finite number of non-zero terms. Since $\bm u$ is $C^1$ on $[0\,,T]$ and $\bm u^{\emptyword}$ is assumed not to hit zero, we deduce that $t \mapsto \bm \psi_t^{\word v}$ is also $C^1$. In the case $\word v = \emptyword\,$, $\bm \psi^{\emptyword}$ is defined as the unique continuous function such that $\bm u^{\emptyword} = \exp (\bm \psi^{\emptyword})$ and $\bm \psi^{\emptyword}_T = \bm p^{\emptyword}\,$. Since $\bm u^{\emptyword}$ is $C^1\,$, never hits zero, and $\bu^{\emptyword}_T = \exp(\bp^{\emptyword})\,$, the function $\bp^{\emptyword} - \int_{\cdot}^T \frac{\dot{\bm u}_s^{\emptyword}}{\bm u_s^{\emptyword}} ds$ is {a continuous complex logarithm of $\bu_t^{\emptyword}$} with terminal condition $\bp^{\emptyword}\,$, and therefore coincides with $\bpsi^{\emptyword}$, which shows that it is continuously differentiable on $[0\,, T]$. This proves the first point.

We now prove that $\bm \psi$ satisfies \eqref{eq:Riccati_BM}.
    Using the fact that $\mathscr L \shuexp{\bm q} = \shuexp{\bm q}\shuprod \mathscr R \bm q\,$, {the relationship $\bu_t = \shuexp{\bpsi_t}$}, and the differential equation satisfied by $\bm u$, we obtain
    \begin{equation}
\label{eq:algebraic_deriv_psi_proof_riccati}
    \shuexp{\bm \psi_t}\shuprod\dot{\bm \psi}_t = \dot{\bm u}_t = -\mathscr{L}\bm u_t = -\shuexp{\bm \psi_t}\shuprod \mathscr R \bm \psi_t\,, \quad t \in [0, T]\,.
    \end{equation}

{Multiplying \eqref{eq:algebraic_deriv_psi_proof_riccati} by $\shuexp{-\bpsi_t}$, yields the Riccati equation \eqref{eq:Riccati_BM}. The terminal condition $\bpsi_T = \bp$ follows from $\bu_T = \shuexp{\bp}$ and the choice of the complex logarithm ensuring $\bpsi_T^{\emptyword} = \bp^\emptyword$, recall Section~\ref{subsec:linear_fct}.
}
{We now move to the proof of the third point of the theorem. For $\word v \neq \emptyword\,$, we have 
$$
   \bm \psi_t|_{\word v} = \sum_{n \geq 1}\,\frac{(-1)^{n-1}}{n}\,\left.\left(\frac{\overline{\bm u}_t}{\bm u_t^{\emptyword}} \right)\shupow{n} \right|_{\word v}\,, \quad t \in [0, T]\,,
   $$
   since $\emptyword|_{\word v} = 0$. Applying iteratively 
   $(\bell \shuprod \bell')|_{\word i} = \bell \shuprod \bell'|_{\word i} + \bell|_{\word i} \shuprod \bell'$ and {$(\bell\shupow{n})\proj{i} = n \bell\shupow{(n-1)}\shuprod\bell\proj{i}$}, for $\word i \in \{\word 0, \word 1\},$ one easily derives that 
   $$
   \bm \psi_t|_{\word v} = \sum_{k = 1}^{|\word v|} \left(\sum_{n \geq k} (-1)^{n-1}\frac{(n-1)!}{(n-k)!}\,\left(\frac{\overline{\bm u}_t}{\bm u_t^{\emptyword}} \right)\shupow{(n - k)} \right)\shuprod P_k(t)\,, \quad t \in [0\,,T]\,,
   $$
   where for each $1 \leq k \leq |\wv|\,$, $P_k(t)$ is a finite sum of terms of the form 
   $$
   \left.\frac{\overline{\bm u_t}}{\bm u^{\emptyword}_t}\right|_{\word{w_1}} \shuprod \ldots \shuprod\left.\frac{\overline{\bm u_t}}{\bm u^{\emptyword}_t}\right|_{\word{w_k}}\,,
   $$
   with $|\word{w_1}| + \ldots + |\word{w_k}| = |\word v|\,$. Hence, using \eqref{eq:bound_norm_shit_u} to bound 
   $$
\left\|\left.\frac{\overline{\bu_t}}{\bu^{\emptyword}_t} \right|_{\word w}\right\|_{\sigX[]} \leq \frac{1}{|\bm u^{\emptyword}_t|}\|\bu_t|_{\word w}\|_{\sigX[]} \leq \frac{C_{\bp, T, \word w}}{|\bu_t^{\emptyword}|} (1 + M_{\bp}(\sigX[])) \exp\left(C_{\bp, T}(1 + M_{\bp}(\sigX[]))^{\deg_{\word 1}(\bp)}\right)\,, \quad t \in [0\,,T]\,, \quad \sigX[] \in G\,,
   $$
   we obtain 
   $$
   \|P_k(t)\|_{\sigX[]} \leq \frac{C'_{\bp, T, \word v}}{|\bm u^{\emptyword}_t|^k} (1 + M_{\bp}(\sigX[]))^{|\word v|} \exp\left(|\word v| C_{\bp, T}(1 + M_{\bp}(\sigX[]))^{\deg_{\word 1}(\bp)}\right)\,, \quad t \in [0\,, T]\,, \quad 1 \leq k \leq |\word v|\,, \quad \sigX[] \in G\,.
   $$
   for some $C'_{\bp, T, \wv} \geq 0\,$. Finally, for any $\sigX[] \in G_{\bp, r}\,$, we have $\|\bu_t\|_{\sigX[]} \leq (2 - \varepsilon) |\bu_t^{\emptyword}|$ for all $t \in [0\,, T]$ (see Theorem \ref{theorem:log_analyticity_general}), so that $\| \overline{\bu_t} / \bm u^{\emptyword}_t\|_{\sigX[]} \leq 1 - \varepsilon\,$. This yields 
   \begin{align*}
       \|\bpsi_t|_{\word v}\|_{\sigX[]} &\leq \sum_{k = 1}^{|\word v|} \|P_k(t)\|_{\sigX[]} \sum_{n \geq k} \frac{(n-1)!}{(n-k)!} \left\|\frac{\overline{\bu_t}}{\bu_t^{\emptyword}} \right\|^{n-k} \\
       &= \sum_{k = 1}^{|\word v|} \|P_k(t)\|_{\sigX[]} \frac{(k-1)!}{\left(1 - \left\|\frac{\overline{\bu_t}}{\bu_t^{\emptyword}} \right\|\right)^k} \\
       &\leq \sum_{k = 1}^{|\wv|}\|P_k(t)\|_{\sigX[]} (k-1)! \varepsilon^{-k}\,, \quad t \in [0\,,T]\,, \quad \sigX[] \in G_{\bp, r}\,.
   \end{align*}
   Combining this with our bound for $\|P_k(t)\|_{\sigX[]}\,$, we obtain,
   \begin{align*}
       \sup_{t \in [0,T]}\|\bpsi_t|_{\word v}\|_{\sigX[]} &\leq C'_{\bp, T, \wv} \left(\sum_{k = 1}^{\word |\wv|} \frac{(k-1)!}{\varepsilon^k \inf_{t \in [0,T]} |\bu_t^{\emptyword}|^k} \right) (1 + M_{\bp}(\sigX[]))^{|\wv|} \exp\left(|\wv| C_{\bp, T} (1 + M_{\bp}(\sigX[]))^{\deg_{\word 1}(\bp)} \right) \\
       &\leq C'_{\bp, T, \wv} \left(\left(1 + \frac{1}{\varepsilon \inf_{t \in [0,T]} |\bu_t^{\emptyword}|}\right)(1 + M_{\bp}(\sigX[])) \exp\left( C_{\bp, T} (1 + M_{\bp}(\sigX[]))^{\deg_{\word 1}(\bp)} \right)\right)^{|\wv|}, \quad \sigX[] \in G_{\bp, r}\,.
   \end{align*}
   for a possibly new constant $C'_{\bp, T, \wv} \geq 0$ in the second inequality, that we do not relabel. This finishes the proof of Theorem \ref{theorem:existence_riccati}.}
\end{proof}


\begin{corollary}[Change of variables in the Riccati equation]\label{corollary_riccati_with_source}
    Consider a Riccati equation 
    \begin{equation}\label{eq:Riccati_BM_with_source}
    \begin{cases}
        \bm \dot{\bpsi}_t + \bpsi_t|_{\word 0} + \frac{1}{2}\, \bm \bpsi_t|_{\word{11}} + \frac{1}{2}\, \bm \bpsi_t|_{\word 1}\shupow{2} + \bpsi_t|_{\word 1}\shuprod\bgamma + \bm f = 0,\quad t \in [0, T]\\
        \bm \bpsi_T = \bp,
    \end{cases}
    \end{equation}
    where $\bp, \bgamma, \bm f \in \TA[2]$. If the coefficient 
    $$
    \widetilde\bp := \bp + \left(\bm f - \dfrac{1}{2}\bgamma\proj{1} - \dfrac12\bgamma\shupow{2}\right)\word{0} + \gamma\word{1} \in \mathcal{B},
    $$
    then there exists a solution $\bpsi$ and $r > 0$ such that
    \begin{enumerate}
        \item For any $\word v \in V\,$, the function $t \mapsto \bm \psi_t^{\word v}$ is $C^1$ on $[0\,, T]\,$.
        \item For any $\word v \in V\,$, $\bm \psi_t|_{\word v} \in \mathcal{A}_{\widetilde\bp, r}$.
    \end{enumerate}
\end{corollary}
\begin{proof}
    Denote $\bq = \left(\bm f - \dfrac{1}{2}\bgamma\proj{1} - \dfrac12\bgamma\shupow{2}\right)\word{0} + \gamma\word{1}$ and consider a change of variables
    $$
    \bchi_t = \bpsi_t + \bq, \quad t \in [0, T].
    $$
    It is straightforward to verify that $\bchi_t$ solves the Riccati equation \eqref{eq:Riccati_coefs} with terminal condition $\bchi_T = \tilde \bp$. The result now follows from Theorem~\ref{theorem:existence_riccati}, using the fact that $\tilde \bp$ has real-valued coefficients so that $\E[\exp(\bracket{\tilde \bp}{\sigW[\cdot, T]})]$ does not vanish on $[0\,,T]\,$.
\end{proof}

\begin{sqremark}
    In the Markovian framework presented in Remark~\ref{rmk:markovian_heat}, the solution
    to the Riccati equation~\eqref{eq:Riccati_BM} can be found in the form
    $\bpsi_t = \bchi_t + \br\word{0}$, where $\deg_{\word{0}}(\bchi_t) = 0$ and
    $\bchi_t = \sum_{n\geq 0}\bchi_t^n\word{1}^{\otimes n}$ solves
    \begin{equation*}
    \begin{cases}
        \dot{\bchi}_t + \frac{1}{2}\bchi_t\proj{11} + \frac{1}{2}(\bchi_t\proj{1})^{*2}
        + \br = 0, \quad t \in [0, T], \\
        \bchi_T = \bq,
    \end{cases}
    \end{equation*}
    which recovers~\eqref{eq:riccati_coef_power_ser} from
    Section~\ref{subsec:power_series_psi} when $\br = 0$. The expansion of the Fourier--Laplace transform
    \begin{align*}
        \E\!\left[\exp\!\left(q(W_T) + \int_0^T r(W_s)\,ds\right) \,\bigg|\, \F_t\right]
        = \exp\!\left(\int_0^t r(W_s)\,ds + \sum_{n \geq 0}\bchi_t^n
        \frac{W_t^n}{n!}\right), \quad t \in \llbracket 0, \tau\rrbracket,
    \end{align*}
    holds up to the stopping time $\tau = \inf\{t \geq 0\colon |W_t| = r\}$ for some $r > 0$.
\end{sqremark}

\begin{sqremark}
The Riccati equation~\eqref{eq:Riccati_BM} is, in general, infinite-dimensional  but it reduces to a finite-dimensional system for certain coefficients satisfying $\deg_{\word{1}}(\bp) \le 2$. Indeed, writing \eqref{eq:Riccati_BM} elementwise for each $\word{v} \in V$,
we obtain
\begin{equation}\label{eq:Riccati_elementwise}
    \dot\bpsi_t^{\word{v}} + \bpsi_t^{\word{v0}} + \frac{1}{2}\bpsi_t^{\word{v11}}
    + \frac{1}{2}\sum_{\word{u}, \word{w}}\bpsi_t^{\word{u}}\bpsi_t^{\word{w}}
    \bracket{\word{v}}{\word{u}\proj{1}\shuprod\word{w}\proj{1}} = 0,
    \quad \bpsi_T^{\word{v}} = \bp^{\word{v}}.
\end{equation}
That is, the coefficient corresponding to the word $\word{v}$ depends on those
corresponding to $\word{v0}$, $\word{v11}$, and all the pairs $(\word{u}, \word{w})$
such that $\word{v}$ appears in the shuffle product
$\word{u}\proj{1}\shuprod\word{w}\proj{1}$. This implies that if
$\deg_{\word{1}}(\bp) \leq 2$, then $\deg_{\word{1}}(\bpsi_t) \leq 2$ for all
$t \in [0, T]$, and that $\deg_{\word{1}}(\bpsi_t) = \infty$ whenever
$\deg_{\word{1}}(\bp) > 2$. Hence, $\deg_{\word{1}}(\bp) \leq 2$ is a necessary
condition for the Riccati equation to be finite-dimensional. For instance, in the linear-quadratic case
$$
\bp = \bp^{\emptyword}\word{\emptyword} + \bp^{\word{1}}\word{1}
+ \bp^{\word{11}}\word{11} + \bp^{\word{10}}\word{10}
+ \bp^{\word{110}}\word{110},
\quad \Re(\bp^{\word{11}}) < 0, \ \Re(\bp^{\word{110}}) < 0,
$$
one can show using~\eqref{eq:Riccati_elementwise} that the Riccati equation
reduces to the following finite-dimensional system:
\begin{equation*}
    -\dot\bpsi^{\emptyword}_t = \frac{1}{2}\bpsi^{\word{11}}_t
    + \frac{1}{2}(\bpsi_t^{\word{1}})^2, \quad
    -\dot\bpsi^{\word{1}}_t = \bpsi^{\word{10}}_t
    + \bpsi^{\word{1}}_t\bpsi^{\word{11}}_t, \quad
    -\dot\bpsi^{\word{11}}_t = \bpsi^{\word{110}}_t
    + (\bpsi_t^{\word{11}})^2, \quad
    -\dot\bpsi_t^{\word{10}} = 0, \quad
    -\dot\bpsi_t^{\word{110}} = 0,
\end{equation*}
and that the other components of $\bpsi$ are zero. This yields that $\bracket{\bpsi_t}{\sigW[t]}$ is a finite sum of terms involving $(1, W_t, W_t^2, \int_0^t W_sds, \int_0^t W_s^2 ds)$. 
The expansion \eqref{eq:logcharsigexpansion} is hence global, and we recover the standard Fourier--Laplace transform representation for the affine process $(W_t, W_t^2)$.
However, the condition $\deg_{\word{1}}(\bp) \leq 2$ is not sufficient for $\deg(\bpsi_t)$ to be finite: it is easy
to verify that if $\bp^{\word{011}} \neq 0$, the solution $\bpsi_t$ has infinitely many
nonzero terms.
\end{sqremark}

Regarding uniqueness of solutions to \eqref{eq:Riccati_coefs}, it is more complicated to define the admissible solutions since the series $\bracket{\bpsi_t}{\sigW[t]}$ converges
only locally, while the localization proof in Theorem~\ref{thm:uniqueness_heat_eq}
clearly requires global convergence. Moreover, the non-linear character of the equation does not allow to easily handle the difference of two solutions. Thus, we define the solution
of~\eqref{eq:Riccati_BM} using the Cole--Hopf transform and the notions of admissibility from Section \ref{subsec:uniqueness_heat}.
\begin{definition}
    A solution $(\bpsi_t)_{t \in [0, T]}$ to the Riccati equation~\eqref{eq:Riccati_BM}
    is called admissible if $(\shuexp{\bpsi_t})_{t \in [0, T]}$ is an
    Itô-admissible solution to the linear equation~\eqref{eq:heat_eq_coef_sig}, in the sense of Definition \ref{def:admissible solution}, that satisfies the sub-Gaussian growth condition \eqref{eq:sub_gaussian_growth}.
\end{definition}
\begin{proposition}\label{prop:riccati_uniqueness}
    The solution $(\bpsi_t)_{t \in [0, T]}$ constructed in
Theorem~\ref{theorem:existence_riccati} is the unique admissible solution
    to~\eqref{eq:Riccati_BM}.
\end{proposition}
\begin{proof}
    We have already shown that $\bu_t = \shuexp{\bpsi_t}$ is the unique
    Itô-admissible solution of~\eqref{eq:heat_eq_coef_sig}, satisfying \eqref{eq:sub_gaussian_growth}, so that
    $(\bpsi_t)_{t \in [0, T]}$ is an admissible solution of~\eqref{eq:Riccati_BM}.
    If $\bpsi'$ is another admissible solution, then, by
    Corollary~\ref{cor:uniqueness_heat}, $\shuexp{\bpsi_t} =
    \shuexp{\bpsi_t'}$ for $t \in [0, T]$. Hence, $\bpsi_t$ and $\bpsi'_t$ may
    differ only by $2\pi k_t i\, \emptyword$ for some $k_t \in \mathbb{Z}$, see Section~\ref{subsec:linear_fct}. However, the
    terminal condition and the continuity requirement for the solution imply
    $\bpsi = \bpsi'$.
\end{proof}

\subsection{The recentering technique}\label{section:recentering}

In this subsection, we tackle the locality of the log-Fourier--Laplace expansion in Theorem \ref{theorem:log_analyticity_general} and provide a numerical illustration for the counterexample introduced in Subsection~\ref{sect:global_exp_counterexample}.

Recall that we showed that the complex function $z \mapsto u(t, z) = \E[\exp(-(z + W_{t, T})^4)]$ has infinitely many zeros on the imaginary axis. These zeros can also be observed numerically, as well as their impact on the computation of the log-Fourier--Laplace transform using $\bm \psi_t\,$. Indeed, $u(t,ix)$ can be efficiently computed via \eqref{eq:I_for_MC} using a Monte Carlo method or approximating the expectations with quadratures. Figure~\ref{fig:zeros_u} provides evidence that $u(0, ix)$, with $T = 0.25$, vanishes on $\C$ with the first zero corresponding to $r \approx 1.5$. It also illustrates the asymptotic approximation given by Theorem~\ref{theorem:zeros_quartic_exp_main}.

Moreover, we provide further evidence of a finite radius of convergence by numerically solving the Riccati equation \eqref{eq:Riccati_BM}, truncated at order $M_{\max}=140$, and computing the coefficients \(\widehat{\bpsi}_0\) using the scheme introduced by \citet*[Section~4]{abijaber2024fourier}. Figure~\ref{fig:psi_series} shows the truncated series 
$\widehat\psi(0, x) =\sum_{k=1}^M \widehat\bpsi_{0}^k x^k / k!$ for different values of $M$ between $5$ and $25$. We clearly see that the convergence radius of the function $\widehat\psi(0, \cdot)$ corresponds precisely to the first zero $r$ introduced above. This highlights a fundamental limitation of the approach based on the log-Fourier--Laplace transform expansion via the Riccati equation: the expansion
$$
\log \mathbb E\!\left[
\exp\!\big(\langle \bm \bp,\sigW[T]\rangle\big)
\Big|\sigW[t]
\right]
=
\langle \bm \psi_t,\sigW[t]\rangle,
$$
holds only up to a positive stopping time, which can be arbitrarily small, and this result cannot be improved.

\begin{figure}[H]
    \centering
    \begin{subfigure}{0.48\textwidth}
        \centering
        \includegraphics[width=\linewidth]{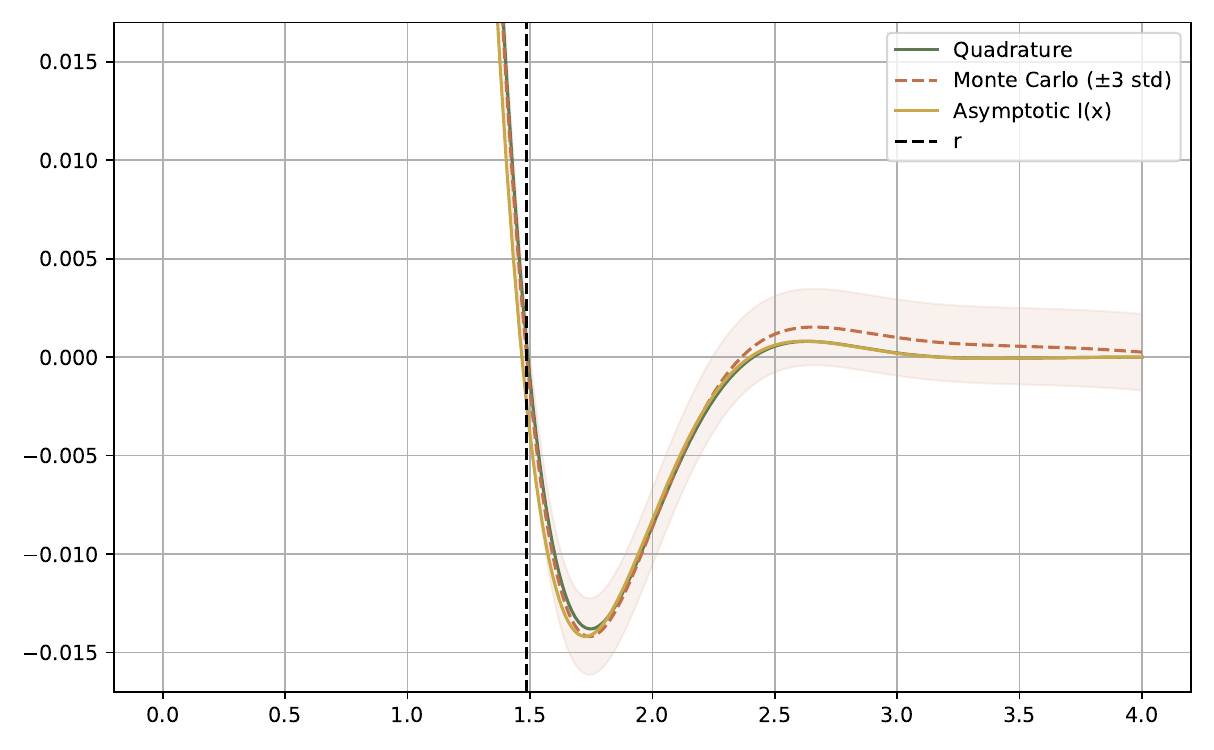}
        \caption{The function $x \mapsto u(0, ix)$ for $T = 0.25$.}
        \label{fig:zeros_u}
    \end{subfigure}
    \hfill
    \begin{subfigure}{0.48\textwidth}
        \centering
        \includegraphics[width=\linewidth]{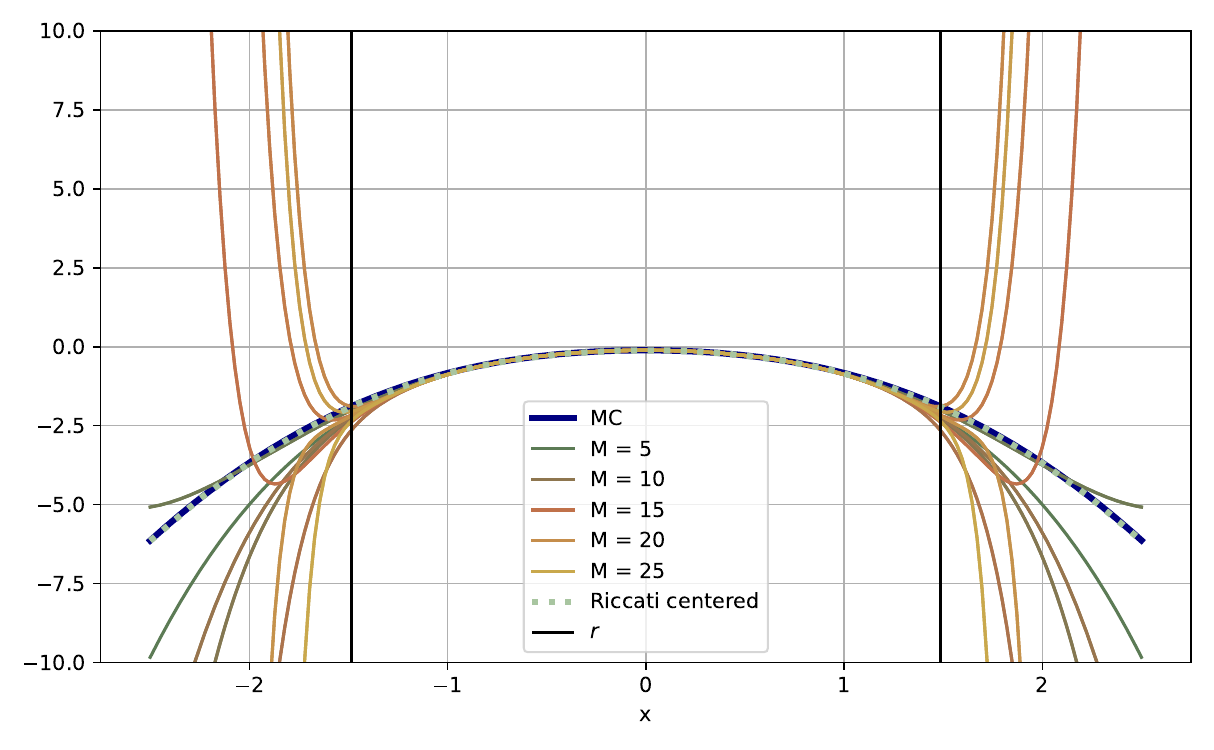}
        \caption{Convergence of the power series expansion of $\widehat\psi(0, \cdot)$ .}
        \label{fig:psi_series}
    \end{subfigure}
    \caption{Numerical illustrations of the solution $u$ and the function $\psi$. Left: zeros of $u(0, \cdot)$ along the imaginary axis. Right: power series approximations of $\widehat\psi(0, x)$ for truncation orders ranging from $5$ to $25$. The dark blue line is the Monte Carlo estimate using $10^6$ samples; the dotted green line corresponds to the expansions obtained via the recentered Riccati equation \eqref{eq:Riccati_recentered}. The horizontal bars correspond to the first zero $r$ of the function $x \mapsto u(0, ix)$.}
    \label{fig:numerical_illustrations}
\end{figure}

To overcome this difficulty and recover a global representation, we propose to consider the expansions centered at $\sigW$ rather than around $\emptyword$ as considered before, making use of Chen's identity \eqref{eq:Chen}, the property of shifts given by Proposition~\ref{prop:projections_ppties}, and Lemma~\ref{lemma:recentering_property_class_B} specific to the class of coefficients $\bp \in \mathcal B\,$.

The key idea is to rewrite the sum $\bracket{\bm p}{\sigX[] \otimes \sigW[t,T]}$ as $\bracket{{}_{\sigX[]}|\bm p}{\sigW[t,T]}$ using Proposition \ref{prop:projections_ppties}. Since $\emptyword \in G$ and $M_{\bm p}(\emptyword) = 0\,$, we can apply Theorem \ref{theorem:existence_riccati} to obtain 
$$
\E\left[\exp (\bracket{\bm p}{\sigX[] \otimes \sigW[t,T]}) \right] = \E\left[\exp (\bracket{\!~_{\sigX[]}|\bm p }{\sigW[t,T]})\right] = \exp  (\bracket{\bm \psi_t^{\sigX[]}}{\emptyword})\,, \quad t \in [0, T]\,,
$$
where ${\bm \psi}^{\sigX[]}_t$ solves the Riccati equation with shifted terminal condition:
\begin{equation}\label{eq:Riccati_recentered}
\begin{cases}
    \dot{\bm \psi}^{\sigX[]}_t + \mathscr{R}\bpsi^{\sigX[]}_t = 0\,, \quad t \in [0\,,T]\\
    \bm \psi_T^{\sigX[]} = ~_{\sigX[]}|\bp
\end{cases}
\end{equation}
which is well defined for any $\sigX[] \in G\,$, since $~_{\sigX[]}|\bp \in \mathcal{B}$ by Lemma~\ref{lemma:recentering_property_class_B}. This is summarized in the following theorem.

\begin{theorem}\label{theorem:recentering}
    Let $\bp \in \mathcal{B}$ and $\mathbb{X} \in G\,$. Assume that $\E[\exp(\bracket{\bp}{\sigX[] \otimes \sigW[t,T]})] \neq 0$ for all $t \in [0\,,T]\,$. Then
    $$
    \E\left[\exp (\bracket{\bm p}{\sigX[] \otimes \sigW[t,T]}) \right] =  \exp  (\bracket{\bm \psi_t^{\sigX[]}}{\emptyword})\,, \quad t \in [0, T]\,,
    $$
    where $\bpsi_t^{\sigX[]}$ is a solution to the recentered Riccati equation \eqref{eq:Riccati_recentered}. In particular, if $\E[\exp(\bracket{\bp}{\sigW[T]}) \,|\, \F_t] \neq 0\,$ a.s., for all $t \in [0, T]$, we have
    $$
\E\left[\exp\left(\bracket{\bm p}{\sigW[T]}\right)\,\Big\vert \,\mathcal F_t\right] = \exp\left(\bracket{\bm \psi_t^{\sigW[t]}}{\emptyword}\right) \,, \quad t \in [0, T]\,.
    $$
\end{theorem}


Hence, we recover a global representation of the log-Fourier--Laplace transform, at the price of solving a new Riccati equation at each time $t\,$, with terminal condition $_{\sigW[t]}|\bp\,$. We discuss practical implementation and numerical aspects related to recentering in the accompanying paper \cite*{riccatisigcontrol}.

\section{Joint Fourier--Laplace transform of $\int_0^T\langle{\sigma, \,}{\sigW}\rangle dW_t$ and its bracket}\label{sect:laplace_sig_mart}

In this section, we apply the results obtained above to characterize the joint Fourier--Laplace transform of the martingale
\begin{equation}\label{eq:sig_mart}
    M_t = \int_0^t\langle{\bsigma, \,}{\sigW[s]}\rangle \, dW_s,
\end{equation}
and its quadratic variation $\langle M\rangle$ via the infinite-dimensional Riccati equation introduced in Section~\ref{sect:heat_and_riccati_sig}. 

\begin{theorem}\label{theorem:Fourier_laplace}
    Let $\bsigma\in T(\C^2)$ such that $\deg(\bsigma) = N$ for some odd $N$ such that 
    $\lambda := \bsigma^{\word{1}\conpow{N}} \neq 0$, and let $M$ be the martingale defined by \eqref{eq:sig_mart}. Suppose that $v = (v_1, v_2) \in \C^2$ is such that 
    $$
    v_1 \in \{z\in \C \colon \Re(\lambda z) < 0\} \cup \{0\}, \quad  \Re\left(\lambda^2 v_2\right) < 0,    
    $$
    and $\E[\exp(v_1 M_T + v_2 \langle M \rangle_T) | \F_t] \neq 0\,$ a.s. for all $t \in [0, T]$.
    Then,
    \begin{equation}
        \E\left[\exp\left(v_1 M_T + v_2 \langle M\rangle_T\right)\, |\, \F_t\right] = \exp(\bracket{\bpsi_t^{\sigW}}{\emptyword})\,, \quad t \in [0\,,T]\,,
    \end{equation}
    where $\bpsi^{\sigW[t]}$ is a solution to the Riccati equation \eqref{eq:Riccati_recentered} with $\sigX[] = \sigW[t]\,$,
    and $\bp$ is given by
    \begin{equation}\label{eq:p_laplace_sig_mart}
        \bp = v_1\left(\bsigma\word{1} - \dfrac12\bsigma\proj{1}\word{0}\right) + v_2 \bsigma\shupow{2}\word{0}.
    \end{equation}
\end{theorem}

\begin{proof}
    Using the Itô--Stratonovich relationship, we write 
    \begin{align*}
        \E\left[\exp\left(v_1 M_T + v_2 \langle M\rangle_T\right)\right] &= \E\left[\exp\left(v_1\int_0^T\bracketsigW{\bsigma}\,dW_t + v_2\int_0^T\bracketsigW{\bsigma}^2\,dt\right)\right] \\
        &= \E\left[\exp\left(v_1\int_0^T\bracketsigW{\bsigma}\circ dW_t - \dfrac{v_1}{2}\int_0^T\bracketsigW{\bsigma\proj{1}}\,dt + v_2\int_0^T\bracketsigW{\bsigma}^2\,dt\right)\right] \\
        &= \E\left[\exp\left(\bracketsigW[T]{\bp}\right)\right],
    \end{align*}
    where $\bp$ is defined by \eqref{eq:p_laplace_sig_mart}. To complete the proof, it remains to verify that $\bp \in \mathcal{B}$.
    Since $N$ is odd, we let $N = 2n - 1$ for some $n \geq 1$. It follows that 
    \begin{equation}\label{eq:sigma_sigvol}
        \bsigma = \lambda\word{1}\conpow{(2n-1)} + \bell, \quad \deg_{\word{1}}(\bell) < 2n - 1.
    \end{equation}
     A straightforward computation shows that 
    \begin{equation}
    \label{eq:p_decomp_proof_sigvol}
    \bp = \lambda v_1 \word{1}^{\otimes 2n} + \left(\lambda^2 v_2 \binom{4n-2}{2n-1} \word{1}^{\otimes(4n-2)} + \bell' \right)\word{0} + v_1 \bell \word{1}
    \end{equation}
    where $\bell' := - \frac{v_1}{2}(\lambda \word 1^{\otimes(2n-2)} + \bell|_{\word 1}) + 2 \lambda v_2 (\word{1}^{\otimes(2n-1)}\shuprod \bell) + v_2 \bell\shupow{2}\,$. Using the fact that $\deg_{\word 1}(\bell) < 2n-1\,$, we have $\deg_{\word 1}(\bell') < 4n-2\,$, and $\deg_{\word{1}}(\bell \word{1}) < 2n\,$. Hence, since $4n-2 > 2n$ for any $ n\geq 1\,$, \eqref{eq:p_decomp_proof_sigvol} corresponds to Definition \ref{def:class_B} with $m = 2n - 1\,$, $\alpha = -\lambda v_1\,$, $\beta = -\lambda^2 v_2 \binom{4n-2}{2n-1}\,$, $\bq = 0\,$, $\br = - \bell' / \beta\,$, $\bm s = \bell \word{1} / (- \lambda \beta)\,$, and $\gamma = 0\,$. This proves that $\bp \in \mathcal{B}\,$, by virtue of the assumptions on $v_1$ and $v_2\,$. Applying Theorem~\ref{theorem:recentering} concludes the proof.
\end{proof}

Theorem~\ref{theorem:Fourier_laplace} can be applied in mathematical finance to compute the joint Laplace transform of the log-price and the integrated variance in the signature volatility model.
Let $(W, W^\perp)$ denote a two-dimensional standard Brownian motion and let $B = \rho W + \sqrt{1-\rho^2}W^\perp$ for some $\rho \in [-1, 1]$. We consider the signature volatility model studied by \cite{abi2025signature}:
\begin{align}\label{eq:sig_vol}
    \dfrac{dS_t}{S_t} = \bracket{\bsigma}{\sigW}\,dB_t,
\end{align}
where $\sigW$ denotes the signature of $t \mapsto (t, W_t)$ and $\bsigma\in T(\R^2)$ with $\deg(\bsigma) = N$. It was shown in \cite[Theorem 3.1]{jaber2025martingalepropertymomentexplosions} that for $N \geq 2$, the price process $S$ in \eqref{eq:sig_vol} is a true martingale if and only if $\rho = 0$ or
\begin{equation}\label{eq:martingality_condition}
    N \text{ is odd} \quad \text{and} \quad \rho \bsigma^{\word{1}^{\otimes N}} < 0.
\end{equation}
Let $U_t := \int_0^t\bracket{\bsigma}{\sigW[s]}^2 ds$ denote the integrated variance. 

\begin{corollary}\label{cor:sig_vol_Fourier_laplace}
    Let $\bsigma \in \TA[2]$ with $\deg(\bsigma) = N$ satisfy either the assumption  \eqref{eq:martingality_condition} {or $\rho = 0$ and $\bsigma^{\word{1}\conpow{N}} \neq 0$}. Suppose that $u, v \in \C$ are such that 
    \begin{equation}\label{eq:sigvol_laplace_cond}
        \Re(u) > 0, \quad \Re\left(\dfrac{u^2(1-\rho^2) - u}{2} + v\right) < 0,
    \end{equation}
    and $\E\left[\exp({u\log(S_T / S_0) + v U_T})\,|\,\mathcal F_t\right] \neq 0$ a.s. for all $t \in [0\,,T]\,$.
    Then,
    \begin{equation}
        \E\left[\exp\left({u\log\left({S_T}/{S_0}\right) + v U_T}\right)\,\big|\, \F_t\right] = \exp\left( \bracket{\bpsi_t^{\sigW}}{\emptyword}\right),
    \end{equation}
    where $\bpsi^{\sigW[t]}$ is a solution to the Riccati equation \eqref{eq:Riccati_recentered} with $\sigX[] = \sigW[t]\,$,
    and $\bp$ is given by
    \begin{equation}\label{eq:p_sig_vol}
        \bp = u\rho\left(\bsigma\word{1} - \dfrac12\bsigma\proj{1}\word{0}\right) + \left(\dfrac{u^2(1-\rho^2) - u}{2} + v \right)\bsigma\shupow{2}\word{0}.
    \end{equation}
\end{corollary}

\begin{proof}
    The log-price in the signature volatility model \eqref{eq:sig_vol} is given by
    \[
    \log(S_T/S_0) = \rho\int_0^T\bracketsigW[s]{\bsigma}\,dW_s + \sqrt{1-\rho^2}\int_0^T\bracketsigW[s]{\bsigma}\,dW_s^\perp - \dfrac12\int_0^T\bracketsigW[s]{\bsigma}^2\,ds.
    \]
    By conditioning on $W^\perp$, we obtain 
    \[
    \E\left[\exp(u\log(S_T/S_0) + vU_T)\,\big|\, \F_t\right] = \E\left[\exp\left(u\rho\int_0^T\bracketsigW[s]{\bsigma}\, dW_s + \left(\dfrac{u^2(1-\rho^2) - u}{2} + v\right)\int_0^T\bracketsigW[s]{\bsigma}^2\, ds\right)\,\Bigg|\, \F_t\right].
    \]
    The result now follows from Theorem~\ref{theorem:Fourier_laplace} with
    \[
    v_1 = u\rho, \qquad v_2 = \left(\dfrac{u^2(1-\rho^2) - u}{2} + v\right).
    \]
\end{proof}

{\begin{sqremark}
    The linear case $N = 1$, for any values of $\rho \in [-1, 1]$ and
    $\bsigma^{\word{1}\conpow{N}} \in \mathbb{R}$, can be easily handled using the standard
    tools of affine SDEs, and is hence not discussed here. Indeed, in this case, the model is
    Markovian in $(B_t, W_t)$, and one can easily verify that the Riccati equation is in fact
    finite-dimensional.
\end{sqremark}
}

\begin{sqremark}
    Equation \eqref{eq:Riccati_BM} differs from the Riccati equation derived in \cite{abi2025signature}:
    \begin{equation}\label{eq:cf_riccati_lag}
    \begin{cases}
        \dot{\bchi}_t + \bchi_t\proj{0} + \dfrac12\bchi_t\proj{11} + \dfrac12\bchi_t\proj{1}\shupow{2} + u\rho\, \bsigma \shuprod \bchi_t\proj{1} + \left(\dfrac{u^2 - u}{2} + v\right)\bsigma\shupow{2} = 0\,,\quad t \in [0\,,T] \\
        \bchi_T = 0.
    \end{cases}
    \end{equation}
    However, it is easily verified that $\bchi_t = \bpsi_t - \bp$, so that the existence of the solution for \eqref{eq:cf_riccati_lag} follows immediately (see Corollary~\ref{corollary_riccati_with_source}).

\end{sqremark}

\begin{sqremark}[Lewis' formula] For the option pricing formula in \cite{Lewis2001}, one needs to compute 
\[
\E\left[\exp(u\log S_T)\right], \quad u \in \C, \quad \Re(u) = \dfrac{1}{2},
\]
which is a specific case of Corollary~\ref{cor:sig_vol_Fourier_laplace} with $v = 0$. Moreover, when $\Re(u) \in (0, 1)$ and under the martingality assumption \eqref{eq:martingality_condition}, the condition \eqref{eq:sigvol_laplace_cond} is satisfied.
Note that if either condition in \eqref{eq:martingality_condition} is violated, then $\bracket{\bp}{\sigW}$ in the proof of Theorem~\ref{theorem:Fourier_laplace} is not bounded from above, and the result no longer holds.
\end{sqremark}



\appendix
{\section{Bounds for time-augmented signatures}\label{sect:proof_of_sig_bound}

Before proceeding to the proof of Lemma~\ref{lem:sig_bound}, we recall the notion of \cite{Lyndon1954} words that helps in obtaining recursive bounds for signature elements.
The extended tensor algebra $T((\C^2))$ forms an algebra when endowed with a shuffle product, and it will be further referred to as the shuffle algebra.

\begin{definition}
    A Lyndon word in an alphabet $A$ is a non-empty word that is strictly lexicographically greater than any of its non-trivial cyclic rotations. 
\end{definition}

\begin{sqexample}
    The first Lyndon words in the alphabet $A = \{ \word{0}, \word{1}\}$ are given by
    $$
    \word{1}, \word{0}, \word{10}, \word{110}, \word{100}, \word{1110}, \word{1100}, \word{1000}, \ldots
    $$
    We note that no Lyndon word ends with the lexicographically greatest letter $\word{1}$, except for the word $\word{1}$ itself. 
\end{sqexample}

A remarkable property of Lyndon words is given by Radford's theorem \citep{Radford1979ANR}: the shuffle algebra $T(\C^2)$ is generated by Lyndon words; that is, each element $\bell \in T(\C^2)$ can be represented by a shuffle polynomial
\begin{equation}\label{eq:polynomial_repr}
    \bell = \sum_{k=1}^m \alpha_k \, \word{v_{k, 1}}\shuprod\word{v_{k, 2}}\shuprod\ldots\shuprod\word{v_{k, n_{k}}},
\end{equation}
for some $m$ and for some, possibly coinciding, Lyndon words $\word{v_{k, j}}$.


\begin{proof}[Proof of Lemma~\ref{lem:sig_bound}]
    We proceed by induction on $|\word{v}|$. If $|\word{v}| = 0$ or $|\word{v}| = 1$, the inequality holds trivially.

    For $|\word{v}| \geq 2$, the representation \eqref{eq:polynomial_repr} yields
    $$
    \word{v} = \sum_{k=1}^m \alpha_k \word{w_{k, 1}} \shuprod\ldots\shuprod\word{w_{k, i_k}} \shuprod \word{0}\shupow{p_k},
    $$
    where each $\word{w_{k, j}}$ is a Lyndon word (either $\word{1}$ or of the form $\word{w'0}$ such that $|\word{w_{k, j}}|_{\word{1}} > 0$). Moreover, we have
    \begin{align}
        |\word{v}| &= |\word{w_{k, 1}}| + \ldots + |\word{w_{k, i_k}}| + p_k, \\
        |\word{v}|_{\word{1}} &= |\word{w_{k, 1}}|_{\word{1}} + \ldots + |\word{w_{k, i_k}}|_{\word{1}}.    
    \end{align}

    We will show that the desired bound holds for each term $\bracket{\word{w_{k, 1}} \shuprod\ldots\shuprod\word{w_{k, i_k}}\shuprod \word{0}\shupow{p_k}}{\sigXhat[s, t]}$ for all $k$. Since $\word{0}\shupow{p_k}$ corresponds to the deterministic term ${(t-s)^{p_k}} \leq T^{|\word{v}|}$, it is sufficient to bound $\bracket{\word{w_{k, 1}} \shuprod\ldots\shuprod\word{w_{k, i_k}}}{\sigXhat[s, t]}$. Hence, if $i_k = 0$, the bound holds trivially.
    
    If $i_k = 1$ and $\word{w_{k, 1}} = \word{1}$, the bound is immediate. When $i_k = 1$ and  $\word{w_{k, 1}}$ is a Lyndon word of the form $\word{w_{k, 1}} = \word{w_{k, 1}'0}$, we apply the induction hypothesis to $\word{w_{k, 1}'}$:
    \begin{equation}\label{eq:recursive_bound_lyndon}
    \begin{aligned}
        |\bracket{\word{w_{k, 1}}}{\sigXhat[s, t]}| &= \left|\int_s^t\bracket{\word{w_{k, 1}'}}{\sigXhat[s, r]}\,dr\right| \leq \int_s^t|\bracket{\word{w_{k, 1}'}}{\sigXhat[s, r]}|\,dr \\
        &\leq C_{\word{w_{k, 1}'}, T}\int_s^t\left(1 + \left|X_{s,r}\right|^{|\word{w_{k, 1}'}|_{\word{1}}} + \int_s^r|X_{s,u}|^{|\word{w_{k, 1}'}|_{\word{1}}}\,du\right)dr \\
        &\leq C_{\word{w_{k, 1}'}, T}\left(T + \int_s^t\left(\left|X_{s,r}\right|^{|\word{w_{k, 1}'}|_{\word{1}}} + \int_s^t|X_{s,u}|^{|\word{w_{k, 1}'}|_{\word{1}}}\,du\right)dr\right) \\
        &\leq \tilde C_{\word{w_{k, 1}'}, T} \left(1 + \int_s^t|X_{s,r}|^{|\word{v}|_{\word{1}}}\,dr\right),
    \end{aligned}
    \end{equation}
    for some $ \tilde C_{\word{w_{k, 1}'}, T} > 0$.

    Finally, if $i_k > 1$, we apply Young's inequality 
    \begin{align}
        |\bracket{\word{w_{k, 1}}\shuprod\ldots\shuprod\word{{w_{k, i_k}}}}{\sigXhat[s, t]}| 
        &\lesssim |\bracket{\word{w_{k, 1}}}{\sigXhat[s, t]}|^{|\word{v}|_{\word{1}} / |\word{w_{k, 1}}|_{\word{1}}} + \ldots + |\bracket{\word{w_{k, i_k}}}{\sigXhat[s, t]}|^{|\word{v}|_{\word{1}} / |\word{w_{k, i_k}}|_{\word{1}}},
    \end{align}
    and the induction hypothesis for each $\word{w_{k, j}}$:
    \begin{align}
        |\bracket{\word{w_{k, j}}}{\sigXhat[s, t]}|^{|\word{v}|_{\word{1}} / |\word{w_{k, j}}|_{\word{1}}} &\leq \left(C_{\word{w_{k, j}}, T}\left(1 + \left|X_{s,t}\right|^{|\word{w_{k, j}}|_{\word{1}}} + \int_s^t|X_{s,r}|^{|\word{w_{k, j}}|_{\word{1}}}\,dr\right)\right)^{|\word{v}|_{\word{1}} / |\word{w_{k, j}}|_{\word{1}}} \\
        &\lesssim 1 + \left|X_{s,t}\right|^{|\word{v}|_{\word{1}}} + \int_s^t|X_{s,r}|^{|\word{v}|_{\word{1}}}\,dr 
    \end{align}
    This finishes the proof of the first inequality in \eqref{eq:sig_bound}.

    The second inequality follows immediately from the first since
    \begin{align*}
        |\sigXhat[s,t]^{\word{v0}}| \leq \int_s^t |\sigXhat[s,r]^{\word{v}}|\,dr &\leq C_{\word{v}, T}\int_s^t\left(1 + |X_{s,r}|^{|\word{v}|_{\word{1}}} + \int_s^u|X_{s,u}|^{|\word{v}|_{\word{1}}}\,du\right)dr \lesssim 1 + \int_s^t|X_{s,r}|^{|\word{v}|_{\word{1}}}\,dr,
    \end{align*}
    where we used the same argument as in \eqref{eq:recursive_bound_lyndon}.
\end{proof}
}

\section{Proof of the existence of zeros in the quartic Fourier--Laplace transform}
\label{app:existence_zeros}
\begin{proof}[Proof of Theorem \ref{theorem:zeros_quartic_exp_main}]
    We first rewrite $I$ as an integral:
    $$
    I(x) = \frac{1}{\sqrt{2\pi}\sigma}\int_{\R}\,\exp\left({-(ix + y)^4 - \frac{y^2}{2\sigma^2}}\right)\,dy\,, \quad x > 0\,.
    $$
    For any $x >0\,$, the function $f_x(z) := \exp\left(-(ix + z)^4 - z^2 / (2\sigma^2) \right)$ is holomorphic on $\C\,$. By Cauchy's theorem, its integral along the complex rectangle composed by the edges $(r\,, r - ix\,, - r - ix\,, -r)\,$, with $r > 0\,$, is zero. Moreover, we can bound the integrals on the vertical lines of the contour by computing the real part of the exponent,
    \begin{align*}
        \left|\int_{-x}^0 f_x(\pm r + iy) dy  \right| &\leq  \int_{-x}^{0}|f_x(\pm r + i y)| dy \\
        &= \int_{-x}^0 \exp \left(-r^4 + 6 (x + y)^2r^2 - (x + y)^4 + \frac{y^2 - r^2}{2\sigma^2}  \right) dy \\
        &\leq x \exp \left(-r^4 + 6 x^2 r^2 + \frac{x^2 - r^2}{2 \sigma^2} \right) \overset{r \to + \infty}{\rightarrow} 0\,, \quad x > 0\,,
    \end{align*}
    so that taking $r \to + \infty$ in the contour identity yields
    $$
    \int_{\R}\,f_x(y)\,dy =  \int_{\R}\,f_x(y - ix)\,dy\,, \quad x > 0\,,
    $$
    which reads
    \begin{equation}\label{eq:I_for_MC}
        I(x) = \frac{1}{\sqrt{2\pi}\sigma}\,\exp\left({\frac{x^2}{2\sigma^2}}\right)\,\int_{\R}\,\exp\left(-y^4 -\frac{y^2}{2\sigma^2} + i \frac{xy}{\sigma^2}\right)\,dy\,, \quad x > 0\,.
    \end{equation}
    We now rescale the integral so that the dominant quartic and oscillatory terms appear at the same order of $x$ in the exponent. This prepares the expression for the saddle-point analysis.
    Setting $a := (4\,\sigma^2)^{-1/3} > 0\,$, and making the change of variable $t = a^{-1} x^{-1/3}\, y\,$, we find
    \begin{align}
    \label{eq:def_J_proof_zeroes}
    I(x) &= \frac{a\,x^{1/3}}{\sqrt{2\pi}\sigma}\,\exp\left({\frac{x^2}{2\sigma^2}}\right)\,\int_{\R}\, \exp\left({-a^4\,x^{4/3}\,t^4 - \frac{a^2\,x^{2/3}}{2 \sigma^2}\,t^2 + i \frac{a\,x^{4/3}}{\sigma^2}\,t}\right)\,dt \nonumber\\
    &= \frac{a\,x^{1/3}}{\sqrt{2\pi}\sigma}\,\exp\left({\frac{x^2}{2\sigma^2}}\right)\,J(x^{4/3})\,, \quad x > 0\,,
    \end{align}
    where 
    $$
    J(x) := \int_{\R}\,\exp\left({-x\,\left(a^4\,t^4 - i \frac{a}{\sigma^2}\,t\right) - \sqrt{x}\,\frac{a^2}{2\sigma^2}\,t^2}\right)\,dt\,, \quad x > 0\,. 
    $$
    We now aim to find an equivalent of $J$ as $x \to + \infty\,$, using complex saddle point methods. We define $F(z) := a^4\,z^4 - i a z / \sigma^2\,$, and $G(z) := a^2 z^2 / (2 \sigma^2)\,$. Since $F(t)$ and $G(t)$ are respectively scaled by $x$ and $\sqrt{x}$ in the latter integral, we can expect that the saddle points of $F$ will represent the main contribution. Hence, we look for solutions to $F'(z) = 0$ on $\C\,$. Using our specific choice for the constant $a$ above, we have $F'(z) = 4 a^4 z^3 - ia / \sigma^2 = 4^{-1/3} \sigma^{-8/3} (z^3 - i)\,$,
    which equals zero for $z \in \{\exp\left(i \pi / 6\right)\,, \exp(5 i \pi / 6)\,, -i\}\,$. Among these critical points, only $z_+ := \exp(i \pi / 6)$ and $z_- := \exp(5 i \pi / 6)$ satisfy $\Re(F''(z)) > 0\,$. Thus, the leading asymptotics will be obtained by adding the two local contributions near $z_+$ and $z_-\,$, meanwhile the third saddle point $-i$ does not lie on the relevant contour of integration. Since $z_+$ and $z_-$ are symmetrical with respect to the imaginary axis, the sum of their contributions will be real and oscillatory.

    The function $\phi_x(z) := \exp\left( -x\,F(z) - \sqrt{x} \,G(z)\right)$ is also holomorphic on $\C\,$. Similarly to $f_x\,$, we can show that its vertical integrals vanish as the real part goes to $\pm \infty\,$, thanks to the quartic term in the exponent. Hence, by Cauchy's theorem,
    $$
    J(x) = \int_{\R}\,\phi_x(t)\,dt  = \int_{\R}\,\phi_x(t + i\,\Im(z_+))\,dt\,,\quad x > 0\,.
    $$
    Since $\Im(z_-) = \Im(z_+)$ and $\Re(z_-) = - \Re(z_+)\,$, we can split the integral as follows 
    \begin{align*}
        J(x) &= \int_{- \infty}^0 \,\phi_x(t + i\,\Im(z_+))\, dt + \int_{0}^{+\infty} \,\phi_x(t + i\,\Im(z_+))\, dt \\
        &= \int_{- \infty}^{\Re(z_+)} \,\phi_x(t + z_-)\, dt + \int_{-\Re(z_+)}^{+\infty} \,\phi_x(t + z_+)\, dt\,, \quad x > 0\,.
    \end{align*}
    This decomposition separates the contributions of the two relevant saddles. We write $J =: J_- + J_+\,$, where $J_-$ and $J_+$ denote the contribution associated with $z_-$ and $z_+\,$, respectively. We now expand the exponent around each saddle. Since $F'(z_{\pm}) = 0\,$, the linear term in the leading phase $-x F(t + z_{\pm})$ vanishes, and the first nontrivial contribution is quadratic: 
    \begin{align*}
    \phi_x(t + z_{\pm}) &= \exp \left(-x (F(z_{\pm}) + \frac{1}{2} F''(z_{\pm})t^2 + \frac{1}{6} F^{(3)}(z_{\pm})t^3 + \frac{1}{24} F^{(4)}(z_{\pm})t^4) - \sqrt{x}(G(z_{\pm}) + G'(z_{\pm})t + \frac{1}{2}G''(z_{\pm}) t^2)\right) \\
    &= \exp\left(-x F(z_{\pm}) - \sqrt{x} G(z_{\pm}) -a^4 x (6 z_{\pm}^2 t^2 + 4 z_{\pm}t^3 + t^4) - \frac{a^2}{2\sigma^2}\sqrt{x}(2 z_{\pm} t + z_{\pm}^2 t^2)\right)\,, \quad x > 0\,, \quad t \in \R\,.
    \end{align*}
    Finally, we perform the change of variable $u = \sqrt{x} t$ in $J_+$ to extract the leading terms, which gives 
    $$
    J_{+}(x) = \frac{\exp(-x F(z_+) - \sqrt{x}G(z_+))}{\sqrt{x}} \, \int_{- \sqrt{x} \Re(z_+)}^{+ \infty} \exp \left(- 6 a^4 z_+^2 u^2 - \frac{4 a^4 z_+}{\sqrt{x}} u^3 - \frac{a^4}{x} u^4 - \frac{a^2 z_+}{\sigma^2} u - \frac{a^2 z_+^2}{2 \sigma^2 \sqrt{x} }u^2\right) du\,.
    $$
    Formally, the integral converges to
    the Gaussian integral 
    $$\int_{\R} \exp \left(- 6 a^4 z_+^2 u^2 - \frac{a^2 z_+}{\sigma^2} u  \right) du = \frac{\sqrt{\pi}}{\sqrt{6} a^2 z_+}\exp\left(\frac{1}{24 \sigma^4}\right)\,.$$ 
    To make this rigorous, we establish an integrable bound independent of $x\,$.
    The real part of the exponent is given by 
    \begin{align*}
        \Re\left(- 6 a^4 z_+^2 u^2 - \frac{4 a^4 z_+}{\sqrt{x}} u^3 - \frac{a^4}{x} u^4 - \frac{a^2 z_+}{\sigma^2} u - \frac{a^2 z_+^2}{2 \sigma^2 \sqrt{x} }u^2\right) &= -3a^4 u^2 - \frac{2\sqrt{3}a^4}{ \sqrt{x}}u^3 - \frac{a^4}{x} u^4 - \frac{\sqrt{3}a^2}{2\sigma^2} u - \frac{a^2}{4 \sigma^2 \sqrt{x}} u^2 \\
        &\leq - a^4 u^2 \left(3 + \frac{2\sqrt{3}}{\sqrt{x}} u + \frac{u^2}{x} \right) - \frac{\sqrt{3}a^2}{2 \sigma^2} u \\
        &= - a^4 u^2 \left(\sqrt{3} + \frac{u}{\sqrt{x}} \right)^2 - \frac{\sqrt{3}a^2}{2 \sigma^2} u
    \end{align*}
    so that, for $u \geq - \sqrt{x} \Re(z_+) = - \sqrt{x} \sqrt{3} / 2\,$, we obtain
    $$
    \left|\exp \left(- 6 a^4 z_+^2 u^2 - \frac{4 a^4 z_+}{\sqrt{x}} u^3 - \frac{a^4}{x} u^4 - \frac{a^2 z_+}{\sigma^2} u - \frac{a^2 z_+^2}{2 \sigma^2 \sqrt{x} }u^2\right) \right| \leq \exp \left( - \frac{3a^4}{4} u^2 - \frac{\sqrt{3}a^2}{2 \sigma^2} u\right)\,,
    $$
    which is independent of $x$ and integrable on $\R\,$. We can thus apply the dominated convergence theorem, leading to 
    $$
    J_+(x) = \frac{\exp(-x F(z_+) - \sqrt{x} G(z_+))}{\sqrt{x}}\left( \frac{\sqrt{\pi}}{\sqrt{6} a^2 z_+} \exp\left(\frac{1}{24 \sigma^4}\right) + o(1)\right)
    $$
    as $x \to + \infty\,$. Similar computations yield 
    $$
    J_-(x) = \frac{\exp(-x F(z_-) - \sqrt{x} G(z_-))}{\sqrt{x}}\left(-  \frac{\sqrt{\pi}}{\sqrt{6}a^2 z_-} \exp \left(\frac{1}{24 \sigma^4}\right) + o(1)\right)\,,
    $$
    where the minus sign appears by taking the principal branch of the square-root in $\sqrt{z_-^2} = - z_-\,$.

    Using the fact that $z_- = - \overline{z_+}\,$, we get that $F(z_-) = \overline{F(z_+)}$ and $G(z_-) = \overline{G(z_+)}\,$, and thus 
    \begin{align*}
        J(x) = 2\sqrt{\frac{\pi}{6}} \frac{\exp(\frac{1}{24 \sigma^4})}{a^2 \sqrt{x}} \Re\left(\frac{\exp(- x F(z_+) - \sqrt{x} G(z_+))}{z_+} \right) + o \left(\frac{\exp(- x F(z_+) - \sqrt{x} G(z_+))}{\sqrt{x}}\right)\,.
    \end{align*}
    It remains to compute $F(z_+)$ and $G(z_+)$ explicitly, yielding 
    $$
    J(x) = \frac{2 \sqrt{\pi}}{a^2\sqrt{6 x}} \exp(\Psi(x)) (\cos(\Phi(x)) + o(1))\,,
    $$
    where 
    $$
    \Psi(x) := \left(\frac{a^4}{2} - \frac{a}{2\sigma^2} \right)x - \frac{a^2}{4 \sigma^2} \sqrt{x} + \frac{1}{24 \sigma^4} 
    \,, \quad \text{and} \quad \Phi(x) := \left(\frac{\sqrt{3} a^4}{2} - \frac{\sqrt{3}a}{2\sigma^2}\right) x + \frac{\sqrt{3}a^2}{4 \sigma^2}\sqrt{x} + \frac{\pi}{6}\,.
    $$
    replacing $a$ with $(4 \sigma^2)^{-1/3}$ in the latter and plugging the result into \eqref{eq:def_J_proof_zeroes} finishes the proof of Theorem \ref{theorem:zeros_quartic_exp_main}.
\end{proof}

\bibliographystyle{plainnat}
\bibliography{refs.bib}

\end{document}